\documentclass[12pt]{article}
\usepackage{geometry}
\usepackage{amsthm}
\usepackage{varioref}
\usepackage{amsmath,amstext,amsthm,a4,amssymb,amscd}   
\usepackage[T1]{fontenc}
\usepackage[utf8]{inputenc}
\usepackage{lmodern}
\usepackage{array}
\usepackage{latexsym}
\usepackage{fancyhdr}
\usepackage{mathrsfs}\let\cal\mathscr
\usepackage[all]{xy}
\usepackage{makeidx}   
\usepackage{color}
\usepackage{pifont}
\usepackage{amsfonts,amssymb}
\usepackage{hyperref}
\usepackage{cleveref}
\usepackage[numbers,square]{natbib}
\usepackage[disable]{todonotes}
\usepackage{verbatim}
\usepackage{relsize}

\usepackage{avant}
\usepackage[T1]{fontenc}      

\usepackage{amsfonts,amssymb}  
 
\usepackage{graphicx}

\usepackage{natbib}
\newcommand \Om {\Omega}
\newcommand \om {\omega}

\renewcommand \i {\sqrt{-1}}
\renewcommand \leq {\leqslant}
\renewcommand \geq {\geqslant}

\newcommand{\norm}[1]{\left\Vert #1\right\Vert}

\newcommand \delt {\partial_t}

\DeclareMathOperator{\End}{End}

\DeclareMathOperator{\Tr}{Tr}

\DeclareMathOperator{\Ker}{Ker}

\DeclareMathOperator{\Spec}{Spec}

\DeclareMathOperator{\Td}{Td}
\DeclareMathOperator{\ch}{ch}

\DeclareMathOperator{\re}{Re}

\DeclareMathOperator{\SU}{SU}

\DeclareMathOperator{\Diff}{Diff}

\DeclareMathOperator{\ad}{ad}
\DeclareMathOperator{\Aut}{Aut}
\DeclareMathOperator{\CS}{CS}

\newcommand \dbar {\overline{\partial}}

\newcommand \< {\mathcal{h}}
\renewcommand \> {\mathcal{i}}

\newcommand \cinf {\CC^\infty}

\newcommand \Id {{\rm Id}}

\renewcommand \epsilon {\varepsilon}
\renewcommand \AA {{\cal A}}
\newcommand \CC {{\cal C}}

\newcommand \HH {{\cal H}}

\newcommand \JJ {{\cal J}}

\newcommand \MM {{\cal M}}

\newcommand \PP {{\cal P}}

\newcommand \TT {{\cal T}}

\newcommand \Q {{\mathcal Q}}
\newcommand \Tau {\mathcal T}

 \def\cC{\mathscr{C}}

\def\cO{\mathscr{O}}

\def\Im{{\rm Im}}

\def\cP{\mathscr{P}}

\def\cK{\mathscr{K}}
\def\cT{\mathscr{T}}

\newcommand{\til}[1]{\widetilde{#1}}

\newcommand \dt {\frac{\partial}{\partial t}}
\newcommand \du {\frac{\partial}{\partial u}}

\newcommand \D[1] {\frac{\partial}{\partial #1}}
\newcommand \Dk[2] {\frac{\partial^{#1}}{\partial {{#2}^{#1}}}}

\newcommand \R {\mathbb R}

\newcommand \C {\mathbb C}

\newcommand \N {\mathbb N}

\newcommand \Z {\mathbb Z}

\newcommand \fl {\rightarrow}

\newcommand \ignore[1] {}

\theoremstyle{plain}
\newtheorem{theorem}{Theorem}[section]
\newtheorem{lem}[theorem]{Lemma}
\newtheorem{cor}[theorem]{Corollary}
\newtheorem{prop}[theorem]{Proposition}
\theoremstyle{definition}
\newtheorem{defi}[theorem]{Definition}

\newtheorem{nota}[theorem]{Notation}

\numberwithin{equation}{section}

\crefname{equation}{}{}
\newtheorem*{ackn*}{Acknowledgements}
\crefname{lem}{Lemma}{Lemmas}
\crefname{theorem}{Theorem}{Theorems}
\crefname{cor}{Corollary}{Corollaries}
\crefname{ex}{Example}{Examples}
\crefname{defi}{Definition}{Definitions}
\crefname{prop}{Proposition}{Propositions}
\crefname{section}{Section}{Sections}
\crefname{subsection}{Section}{Sections}
\crefname{rmk}{Remark}{Remarks}
\crefname{nota}{Notation}{Notations}

\hypersetup{
    colorlinks=true,
    linkcolor=red,
    filecolor=magenta,      
    urlcolor=cyan,
}

\begin{document}

\title{\bf{Geometric quantization of symplectic maps
and Witten's asymptotic conjecture}}
\author{Louis IOOS}
\date{}

\maketitle

\begin{abstract}

We use the theory of Berezin-Toeplitz operators of
Ma and Marinescu to study the spaces of holomorphic sections
of a prequantizing line bundle over compact Kähler manifolds
under deformations of the complex structure. We show that
the parallel transport in the induced vector bundle over
the deformation space behaves like a Toeplitz operator, and compute
its first coefficient. We then use this result
to establish a semi-classical trace formula for the induced
quantization of symplectic maps, and give an application to
Witten's asymptotic expansion conjecture for the quantum
representations of the mapping class group.
\end{abstract}

\todo{garder "théorie de BT op de MM"?}


\section{Introduction}
%

Geometric quantization is a set of geometric
methods to construct a quantum mechanical system, represented
by a Hilbert space of quantum states, from the underlying
system of classical mechanics, represented by a symplectic manifold
$(X,\om)$. In this context, we require
$(X,\om)$ to be \emph{prequantized}, so that
$X$ is endowed with a Hermitian line bundle $(L,h^L)$
together with a Hermitian connection
$\nabla^L$ of curvature $R^L\in\Om^2(X,\C)$ satisfying
\begin{equation}\label{preq}
\om=\frac{\sqrt{-1}}{2\pi} R^L\,.
\end{equation}
Assume now that $(X,\om)$ admits a compatible
integrable complex structure
$J\in\End(TX)$, making $(X,J,\om)$ into a \emph{Kähler manifold}
and $(L,h^L,\nabla^L)$ into a holomorphic Hermitian line
bundle equipped with its \emph{Chern connection}. 
Then the \emph{Kähler quantization} of $(X,\om)$ at level
$p\in\N^*$ is the Hilbert space $\HH_p$ of
$L^2$-holomorphic sections of
$L^p:=L^{\otimes p}$ for the natural
$L^2$-Hermitian product \cref{L2}. The integer $p\in\N^*$
represents a quantum number, usually inversely proportional
to the \emph{Planck constant}, and
asymptotic results when $p$ tends to infinity are then supposed
to describe the so-called \emph{semi-classical limit}, when
the scale gets
so large that we recover the laws of classical mechanics as an
approximation of the laws of quantum mechanics.
In this paper, we will restrict to the case
of $(X,\om)$ \emph{compact} of dimension
$\dim X=2n$, so that $(X,\om)$ represents a bounded mechanical
system and the associated Hilbert space of quantum states $\HH_p$
is finite dimensional for all $p\in\N^*$.

A fundamental problem in this context
is the dependence of the quantization
on the choice of a complex structure $J\in\End(TX)$.
A natural way to study this question is to consider
the quantization of $(X,\om)$ at level $p\in\N^*$
as a Hermitian vector bundle $\HH_p$ over a space $B$ of compatible complex structures, whose fibre at $b\in B$
identifies with the induced Hilbert space $\HH_{p,b}$ of
holomorphic sections. One can then compare
the quantizations associated with different complex structures
via parallel transport with respect to the
natural \emph{$L^2$-connection} \cref{connectionL2} on the
vector bundle $\HH_p$ over $B$.
This point of view can be applied in particular
to the quantization of \emph{symplectic maps},
that is diffeomorphisms $\varphi:X\fl X$ preserving $\om$, which
we require in addition to lift to the prequantization
$(L,h^L,\nabla^L)$. In the particular case when $\varphi:X\fl X$
preserves the complex structure, one can define its quantization
at level $p\in\N^*$ as the induced unitary operator on
holomorphic sections of $L^p$, but a symplectic map does not preserves
any compatible complex structure in general.
\todo{mettre la quantif des symplectomorphismes après le Th.1?}
Instead, consider a path
\begin{equation}\label{Jt}
\{J_t\in\End(TX)\}_{t\in\R}
\end{equation}
of compatible complex structures such that $J_0:=J$ and
$J_1:=\varphi^*J$. Then for any $p\in\N^*$, there is an induced
pullback map
$\varphi^*_p:\HH_{p,1}\fl\HH_{p,0}$
from the space of holomorphic sections
of $L^p$ with respect to $J_1$ to the
space of holomorphic sections of $L^p$ with
respect to $J_0$.
One can then consider the parallel transport
$\Tau_p:\HH_{p,0}\fl\HH_{p,1}$ with respect
to the $L^2$-connection along this path to get by composition
a unitary operator
\begin{equation}\label{quantmap}
\varphi^*_p\Tau_p:\HH_{p,0}\longrightarrow\HH_{p,0}\,,
\end{equation}
giving a geometric definition
of the quantization of the symplectic map $\varphi:X\fl X$.
%

In this paper, we use the theory of Bergman kernels
of Ma and Marinescu in \cite{MM08a}
to study the parallel transport $\Tau_p$
as $p\fl +\infty$. To describe our main result,
let $\pi:\R\times X\fl\R$ be the fibration of complex manifolds
with fibre $X_t:=(X,J_t)$ at $t\in\R$, and write
$\tau^{K_X}_t:K_{X_0}\fl K_{X_t}$ for
the parallel transport in the induced
\emph{relative canonical line bundle}
$K_X:=\det(T^{(1,0)*}X)$ over $\R\times X$ with respect
to the natural connection \cref{nabdt=PdtP} along
the horizontal directions of the fibration.
Using the identification\break
$T^{(0,1)}X_t\simeq T^{(1,0)*}X_t$
induced by the $\C$-bilinear form $g^{TX}_t:=\om(\cdot,J_t\cdot)$
over $TX_\C$ for all $t\in\R$, we denote by
$\det(\overline{\Pi_t^0}):K_{X,0}\fl K_{X,t}$ the
line bundle isomorphism induced by the projection
$\overline{\Pi_t^0}:T^{(0,1)}X_0\fl T^{(0,1)}X_t$
on $T^{(0,1)}X_t$ with kernel $T^{(1,0)}X_0$ inside
$TX_\C$.
The following theorem, which follows from
\cref{Approx} and \cref{mutgeom}, expresses the parallel transport
$\Tau_{p,t}:\HH_{p,0}\fl\HH_{p,t}$ with respect to the
$L^2$-connection \cref{connectionL2} as a \emph{Toeplitz operator}
from one quantum space to another, and is the central result of
this paper.
\begin{theorem}\label{Approxintro}
There exists a family of functions
$\{\mu_{l,t}\in\CC^\infty(X,\C)\}_{l\in\N}$,
smooth in $t\in[0,1]$, such that for all $k\geq 0$, there exists $C_k>0$ such that for all $p\in\N^*$,
\begin{equation}\label{genToepdefintro}
\Big\|\Tau_{p,t}-\sum_{l=0}^{k-1} p^{-l}\,P_{p,t\,}\mu_{l,t}\,P_{p,0}\Big\|\leq C_k p^{-k}\,,
\end{equation}
in operator norm, where $P_{p,t}:\cinf(X,L^p)\fl\HH_{p,t}$ is
the orthogonal projection with respect to the
$L^2$-Hermitian product on the
space of holomorphic sections of $L^p$ over $X_t$.

Furthermore, via the canonical isomorphism $\End(K_{X,0})\simeq\C$,
we have the following formula for the first coefficient,
\begin{equation}\label{Taucoeffint}
\overline{\mu}_{0,t}^{2}=\det(\overline{\Pi_t^0})^{-1}\tau^{K_X}_t\,.
\end{equation}
\end{theorem}

\cref{Approxintro} generalizes and refines
the asymptotic expansion established by Andersen in
\cite[Th.6]{And06} for the parallel transport in the endomorphism bundle $\End(\HH_p)$ as $p\fl+\infty$ induced by
\emph{Hitchin connections}, as defined in \cref{hitsec} following
\cite[\S\,1]{And12}. In particular, Andersen uses his expansion
to establish in \cite[Th.1]{And06} the asymptotic faithfulness
of the quantum representations of the mapping class group, while
our explicit formula \cref{Taucoeffint} for the first term
of the expansion \eqref{genToepdefintro} will be crucial here
to obtain Witten's formula for
the quantum representations of the mapping class group
in \cref{WRTintro,WRTintro2}.

On the other hand,
\cref{Approxintro} has already been applied
in \cite{Ioo19} to the quantization
of Hamiltonian flows $\varphi_t:X\fl X,\,t\in\R$, via the parallel transport
over the induced path $\{\varphi_t^*J\in\End(TX)\}_{t\in\R}$
of complex structures. In particular,
our explicit formula \eqref{Taucoeffint}
for the first coefficient is used in \cite[Th.1.2,\,Th.1.3]{Ioo19}
to establish a Gutzwiller trace formula in this
context. The parallel transport over complex structures
induced by Hamiltonian flows has also been studied by
Foth and Uribe in \cite{FU07}, where it is shown that it satisfies a Schrödinger equation, and by Charles in \cite{Cha07}, who studies
its semi-classical properties in the metaplectic case.
%
%

In \cref{locsec}, we use \cref{Approxintro} to estimate the
trace of the quantization \cref{quantmap} of a symplectic map
$\varphi:X\fl X$ as $p\fl +\infty$, showing in particular
that it localizes around the fixed point set $X^\varphi\subset X$
of $\varphi$. In the following theorem, which follows from
\cref{indeq} and \cref{coeffgeom},
we assume for simplicity that $X^\varphi$ is connected.
\begin{theorem}\label{indeqintro}

Assume that $X^\varphi$ is a smooth 
submanifold with
$TX^\varphi=\Ker(\Id_{TX}-d\varphi)$.
Then there are densities
$\nu_r$ over $X^\varphi$ for all $r\in\N$,
such that for any $k\in\N$ and as $p\fl +\infty$,
\begin{equation}\label{indeqfleintro}
\Tr_{\HH_p}[\varphi^*_p\Tau_p]=p^\frac{\dim X^\varphi}{2}
\lambda^p
\left(\sum_{r=0}^{k-1} p^{-r} \int_{X^\varphi}\nu_r
+O(p^{-k})\right)\,,
\end{equation}
where $\overline{\lambda}\in\C$ is the value of the action of
$\varphi:X\fl X$ on $L$ over $X^\varphi$.
Furthermore, there is an explicit local formula for $\nu_0$,
given in \cref{nu0}.

If $X^\varphi$ is a complex submanifold of $X$ and
if $\varphi$ preserves a complex subbundle $N$
transverse to $TX^\varphi\subset TX$ over $X^\varphi$, then
\begin{equation}\label{nu0intro}
\nu_0=(-1)^\frac{2n-\dim X^\varphi}{4}(\varphi^{K_X}\tau^{K_X,-1})^{1/2}
\det{}_{N}(\Id_N-d\varphi|_N)^{-1/2}|dv|_{TX/N},
\end{equation}
for some natural choices of square roots, where
$\tau^{K_X}:K_{X_0}\fl K_{X_1}$ is the parallel transport
with respect
to the natural connection \cref{nabdt=PdtP}
and $|dv|_{TX/N}$ is the density over $X^\varphi$
induced by $g^{TX}_0:=\om(\cdot,J_0\,\cdot)$ and the decomposition
$TX=TX^\varphi\oplus N$ as in \eqref{dvX=dvNdvXN}.
\end{theorem}


Note that the assumption of \cref{nu0intro} is
automatically satisfied when $\dim X^\varphi=0$ or when
$\varphi$ is holomorphic. The general version of \cref{indeqintro},
which is \cref{indeq}, gives 
a general formula for the higher order term of 
\eqref{indeqfleintro} without
the assumption of \cref{nu0intro}, and
includes the tensor product with a general
Hermitian vector bundle $E$ over the fibration
$\pi:\R\times X\fl\R$.
In \cref{isosec}, we first establish
the general version of \cref{indeqintro} as above in the case
$\dim X^\varphi=0$, which generalizes a result of Charles
in \cite[Th.5.3.1]{Cha10},
who only handles the \emph{metaplectic} case,
that is taking $E=K_X^{1/2}$ to be
a square root of the relative canonical bundle.
He then applies it in \cite[Th.1.2]{Cha16} to prove an asymptotic
projective version of Witten's asymptotic conjecture for the
quantum representations of the mapping class group.
As we will see below, it is crucial to be able to consider
the case without metaplectic correction to handle the
actual conjecture, as formulated in the language of modular
functors of Segal in \cite[\S\,5]{Seg04}.
Furthermore, \cref{indeqintro} deals with $X^\varphi$ of arbitrary
dimension, which allows to extend the conjecture to
that case.

In case the symplectic map $\varphi$ preserves the complex structure,
so that the pullback map $\varphi_p^*$ preserves $\HH_{p}$ for
all $p\in\N^*$, we recover via \cref{indeq}
the asymptotics of
$\Tr_{\HH_p}[\varphi_p^*]$ as $p\fl+\infty$
of the following \emph{equivariant Riemann-Roch-Hirzebruch} formula
\begin{equation}\label{eqRRH}
\Tr_{\HH_p}[\varphi_p^*]=\int_{X^\varphi}
\Td_{\varphi^{-1}}(T^{(1,0)}X)\ch_{\varphi^{-1}}(L^p)\,,
\end{equation}
where $\Td_{\varphi^{-1}}(T^{(1,0)}X)$ represents the
equivariant Todd class of $(X,J)$ and
$\ch_{\varphi^{-1}}(L^p)$ represents the equivariant Chern
character of $L^p$
for the action induced by $\varphi^{-1}$,
as defined in \cite[Def.\,1.5]{BG00}.

In \cref{Appli}, we show how \cref{indeqintro} can be applied
to establish \emph{Witten's asymptotic expansion conjecture}
for the quantum representations of the mapping class group,
as described for instance in \cite[Conj.1.1]{And13}, and 
to compute the formula given by Witten in \cite[(2.17)]{Wit89}
for the first coefficient.
For simplicity and concreteness, we will only consider the
case $X:=\MM$ to be the \emph{moduli space}
of gauge equivalence classes
of flat $\SU(m)$-connections
over $\Sigma\backslash D$ with holonomy around the boundary
equal to $e^{\frac{2\sqrt{-1}\pi d}{m}}$ 
for some fixed $d\in(\Z/m\Z)^*$, where
$\Sigma$ is a compact oriented surface of 
genus $g\geq 2$ and $D\subset\Sigma$ is an embedded disk.
Then as explained in \cref{MM,LCS}, the moduli space
$\MM$ admits a natural structure of a compact
prequantized symplectic manifold.
We then consider the map $\varphi:\MM\fl\MM$ naturally
induced by an orientation preserving diffeomorphism
$f\in\Diff^+(\Sigma,D)$
preserving $D$ pointwise. Then $\varphi$ is a smooth symplectic map
lifting naturally to the prequantization, and
only depends on $f\in\Diff^+(\Sigma,D)$
up to isotopies of $\Sigma$ preserving $D$, so that it represents
an element of the \emph{mapping class group} of
$\Sigma$ acting symplectically on $\MM$.

An element $\sigma\in\cT_\Sigma$ of the \emph{Teichmüller space}
of $\Sigma$ induces canonically a compatible
complex structure $J_\sigma\in\End(T\MM)$, and we can consider the associated quantum bundle $\HH_p$ over
$\cT_\Sigma$ for all $p\in\N^*$, called the \emph{Verlinde bundle}.
It is endowed with a canonical connection for all
$p\in\N^*$, which was
introduced independently by Hitchin in \cite[Th.3.6]{Hit90}
and Axelrod, S. Della Pietra and Witten in \cite[\S\,4.b]{ADW91}.
Note that this canonical connection is defined in \cite{Hit90} only
up to an additive scalar form, which is fixed canonically in \cite{ADW91}.
Let $\Tau_p$ be the induced parallel transport along a path
$\gamma:[0,1]\fl\cT_\Sigma$ joining $\sigma\in\cT_\Sigma$
to $f^*\sigma\in\cT_\Sigma$. Then
the trace
$\Tr_{\HH_p}[\varphi^*\Tau_p]$ is the
\emph{anomalous Witten-Reshetikhin-Turaev invariant}
of the mapping torus $\Sigma_f$
with one link lifting the canonical fibration
$\pi_f:\Sigma_f\fl\R/\Z$
colored by $e^{\frac{2\sqrt{-1}\pi d}{m}}$, as considered
in Segal's definition of conformal field theory in \cite[\S\,4]{Seg04}
after Witten's description in \cite[(2.1)]{Wit89} using path
integrals.

In fact, let $\MM_{f}$ be the moduli space
of gauge equivalence classes of flat $\SU(m)$-connections
over the mapping torus $(\Sigma\backslash D)_f$ of
$f\in\Diff^+(\Sigma,D)$ with holonomy around
the boundary equal to $e^{\frac{2\sqrt{-1}\pi d}{m}}$,
and assume that $\MM_f$ admits a smooth structure as in
\cref{MM}.
Then there is a smooth covering $r:\MM_f\fl\MM^\varphi$
defined by restriction on any fibre, where $\MM^\varphi\subset\MM$
is the fixed
point set of $\varphi$. For any $[A_f]\in\MM_f$,
the \emph{Chern-Simons invariant}
of $[A_f]\in\MM_f$ is given by the formula
\begin{equation}\label{corCS}
\CS([A_f])=\frac{m}{8\pi^2}\left(\Tr[\xi_1\xi_2]+
\int_{\Sigma_f}
\Tr\big[\alpha_f\wedge d\alpha_f+\frac{2}{3}
\alpha_f\wedge\alpha_f\wedge\alpha_f\big]\right)\in\R/\Z\,,
\end{equation}
where we took a connection form
$\alpha_f\in\Om^1(\Sigma_f,\mathfrak{su}(m))$
of $[A_f]$ equal to $\xi_1 d\theta_1+\xi_2 d\theta_2$ on the
boundary of $(\Sigma/D)_f$ seen as a torus with coordinates
$\theta_1,\,\theta_2\in S^1$, with $\xi_1,\,\xi_1\in\mathfrak{su}(m)$ constant.

On the other hand, the path $\gamma:[0,1]\fl\cT_\Sigma$
in Teichmüller space
determines a complex
structure on the relative tangent bundle of the canonical fibration
$\pi_f:\Sigma_f\fl\R/\Z$ up to isotopy, and we endow the fibres
with the compatible hyperbolic Riemannian metric.
For any $[A_f]\in\MM_f$, write $[\ad A_f]$ for the induced connection
on the trivial adjoint $\mathfrak{su}(m)$-bundle over the whole
$\Sigma_f$.
\begin{theorem}\label{WRTintro}
Let $\MM_f=\coprod_{j=1}^q\MM_{f,j}$
be the decomposition into connected components. 
Then there exist densities
$\nu_r$ over $\MM_f$ for all $r\in\N$,
such that for any $k\in\N$, the anomalous Witten-Reshetikhin-Turaev
invariant satisfies as $p\fl +\infty$,
\begin{equation}\label{WAC}
\Tr_{\HH_p}[\varphi^*_p\Tau_p]=\frac{1}{m}\sum_{j=1}^q
p^{\frac{\dim\MM_{f,j}}{2}}
e^{2\sqrt{-1}\pi p \CS_j}(\sqrt{-1})^{k_j}
\left(\sum_{r=0}^{k-1} p^{-r} \int_{\MM_{f,j}}\nu_r
+O(p^{-k})\right),
\end{equation}
where $k_j\in\Z/4\Z$ and where $\CS_j\in\R/\Z$ is the constant
value of the Chern-Simons
invariant \cref{corCS} over $\MM_{f,j}$.
Furthermore, there is an explicit
local formula for $\nu_0$.

If $r:\MM_f\fl\MM_\Sigma$ is holomorphic and
if $\varphi$ preserves a complex subbundle
transverse to $T\MM^\varphi\subset T\MM$ over $\MM^\varphi$,
we get
\begin{equation}\label{unWRTexp}
\nu_{0}([A_f])=
\exp\left(\frac{\sqrt{-1}\pi}{4}\eta^0(\ad A_f)\right)
\big|\tau_{\Sigma_f}(\ad A_f)\big|^{1/2},
\end{equation}
for any $[A_f]\in\MM_f$, where $\eta^0(\ad A_f)$ is the adiabatic limit
of the $\eta$-invariant of the odd signature operator of
$\ad A_f$ restricted to odd forms over the Riemannian fibration
$\pi_f:\Sigma_f\fl\R/\Z$, and
$|\tau_{\Sigma_f}(\ad A_f)|^{1/2}$ is the square root
of the absolute value of the Reidemeister torsion of $\ad A_f$,
seen as a density over $\MM_f$ via Poincaré duality.
\end{theorem}
Note that the hypothesis of \cref{unWRTexp} is automatically 
satisfied
when $\dim\MM_f=0$ or when $f$ preserves $\sigma\in\cT_\Sigma$.
The presence of the $\eta$-invariant in
the formula \cref{unWRTexp} is a consequence of the holonomy
theorem of Bismut and Freed in \cite[Th.3.16]{BF86b} as well as the study
of holomorphic determinant line bundles of Bismut, Gillet and Soulé
in \cite[\S\,1]{BGS88c}, and
depends on the chosen path in Teichmüller space.
On the other hand, as showed in \cite[Th.4.9]{Hit90} and
in \cite[\S\,4.b]{ADW91},
the canonical connection over the Verlinde bundle is projectively
flat. By the explicit computation of its curvature in
\cite[\S\,4.b]{ADW91} and in the language of modular functors of
\cite[(5.9)]{Seg04}, the
\emph{Witten-Reshetikhin-Turaev invariant} of the
mapping torus with one colored link as above
is computed by the trace $\Tr_{\til{\HH}_p}[\varphi^*_p\Tau_p]$,
where $\Tau_p$ is the flat parallel transport in
\begin{equation}\label{Hptilde}
\til{\HH}_p:=\HH_p\otimes\det(\dbar_\Sigma)
^{-\frac{(m^2-1)p}{2(p+m)}}\,,
\end{equation}
induced by the canonical projectively flat connection
on $\HH_p$ and the Chern connection associated with
the \emph{Quillen metric} on the
holomorphic \emph{determinant line bundle} $\det(\dbar_\Sigma)$
of the universal family of
$\dbar$-operators
over $\cT_\Sigma$, for which fractional powers exist
as Teichmüller space is contractible and on which
$\varphi$ lifts naturally up to choices of
compatible lifts on each fractional power.
The following result then computes its
first order as $p\fl+\infty$.
\begin{theorem}\label{WRTintro2}
The Witten-Reshetikhin-Turaev invariant
defined above
satisfies the expansion \cref{WAC}, with first coefficient as in
\cref{unWRTexp} given for any $[A_f]\in\MM_{f}$ by
\begin{equation}\label{WRTexp}
\nu_0([A_f])=
\exp\left(\frac{\sqrt{-1}\pi}{4}\rho(\ad A_f)\right)
\big|\tau_{\Sigma_f}(\ad A_f)\big|^{1/2},
\end{equation}
where $\rho(\ad A_f)$ is the topological $\rho$-invariant
of $\ad A_f$ defined in \cref{rhodef}.
\end{theorem}
As explained for instance in \cite[p.10]{And13},
the choice of a lift of $\varphi$ to fractional
powers of $\det(\dbar_\Sigma)$ corresponds to the canonical
central extension of the mapping
class group induced by the Atiyah $2$-framing of
\cite{Ati90}, which can be used to get the corresponding correction in
\cite[(2.23)]{Wit89}. The map associating an element of this extension
to the endomorphism $\varphi^*_p\Tau_p\in\End(\til\HH_p)$
is called a
\emph{quantum representation of the mapping class group}.
Another approach followed by Charles
in \cite[\S\,7]{Cha16} is to consider the metaplectic correction,
replacing $L^p$ by $L^p\otimes K_\MM^{1/2}$ for all $p\in\N^*$.
However, the corresponding trace does not
coincide with the Witten-Reshetikhin-Turaev invariant,
and this approach
does not imply \cref{WRTintro} nor \cref{WRTintro2}, even in the
case $\dim\MM_f=0$, which is the only one handled in
\cite[Th.1.2]{Cha16}.
By \cref{hitchintoepmeta}, we recover
\cite[Th.1.2]{Cha16} as a consequence of the general version
of \cref{indeqintro}.

The Verlinde
bundle with its canonical connection is defined for much more
general moduli spaces, the one considered above being the
standard smooth model. The above approach applies as soon as
the moduli space has a natural smooth structure, and
\cref{Approxintro,indeqintro} should extend to the case
of orbifolds, so that one could get the expansion
\cref{WAC} for much more general quantum representations of the
mapping class group.
In the case when $f\in\Diff^+(\Sigma,D)$ preserves
$\sigma\in\cT_\Sigma$, so that
$\varphi:\MM\fl\MM$ preserves $J_\sigma\in\End(T\MM)$,
Andersen applies in \cite[Th.1.1]{And13}
a singular version of
the equivariant Riemann-Roch-Hirzebruch formula \cref{RRH}
to establish the asymptotic expansion for general
singular moduli spaces, and Andersen and Himpel
identify the first coefficient with \cref{WRTexp}
in \cite[Th.1.6]{AH12}.
In \cite[Th.1.2]{AP18}, Andersen and Petersen work in the case of Hitchin connections of \cref{hitsec},
but under a weaker assumption on the fixed point set, to
get the asymptotic expansion \cref{WAC}
for a larger class of mapping tori. However, they do not
compute the first coefficient.


This paper is based on the theory of Berezin-Toeplitz operators
of Ma and Marinescu in \cite{MM08b},
and can be seen as an extension of this theory to families.
More specifically, we work in the context of the
Berezin-Toeplitz quantization introduced in \cite{ILMM17}
of general prequantized symplectic manifolds,
using the spaces $\HH_p$ of
\emph{almost holomorphic sections}
of $L^p$ for all $p\in\N^*$ large enough, defined as the direct
sum of the eigenspaces associated with the small eigenvalues of
the \emph{renormalized Bochner Laplacian} \cref{deltapphi}
of Guillemin and Uribe \cite{GU88}. These spaces coincide with
the space of holomorphic sections for $p\in\N^*$ large
enough in the Kähler case, and \cref{indeq,Approx},
which are the general versions of
\cref{indeqintro,Approxintro}, are proved in this generality.
A family version of Berezin-Toeplitz quantization has first
been studied by Ma and Zhang in \cite{MZ07},
where they show that the curvature of the quantum
bundle with respect to the $L^2$-connection
is a Toeplitz operator, and compute the first
two coefficients.

The theory of Berezin-Toeplitz operators in the Kähler case was
first developed by Bordemann,
Meinreken and Schlichenmaier in \citep{BMS94} and Schlichenmaier in
\citep{Sch00}. Their approach is based on
the work of Boutet de Monvel and Sjöstrand on the Szegö kernel
in \cite{BdMS75}, and
the theory of Toeplitz structures developed by Boutet de Monvel
and Guillemin in \citep{BdMG81}.
In this context, Zelditch first introduced in \cite{Zel97}
a quantization of symplectic maps based on a unitary
version of \citep{BdMG81}, which does not depend on the choice
of a path of complex structures. This quantization is defined up
to a phase function in its symbol, while the amplitude
\cite[(3.10)]{Zel97} of its symbol coincides with the amplitude
\cref{barmut} of the local model in flat space
on the diagonal, which does not depend on a path of complex
structures. This shows via \cref{Tautologiclem} that the first
coefficient \cref{Taucoeffint} in \cref{Approxintro} also depends
on the path only up to a phase function.
The local model of \cref{locmod}
has also been considered in another form
by Kirwin and Wu in \cite{KW06}, where
they deal with paths of complex structures on flat space given by
geodesics in the Siegel upper half-space. In particular,
we recover \cite[Th.3.4]{KW06} as a special case of \cref{ptloc}.

In \cref{defsec}, we introduce the basic notions of fibrations used
in this paper, and recall the results from the theory of the Bergman kernel
of \cite{MM07} in this setting. In \cref{genTeopsec},
we study the composition of Bergman
kernels associated with different complex structures,
introduce the corresponding generalization of a Toeplitz operator
and establish the general version of
\cref{Approxintro}. In \cref{locsec},
we use this result to establish the general version of
\cref{indeqintro}.
In \cref{Appli}, we introduce the notion of a Hitchin
connection and relate it to our notion of Toeplitz connection,
then apply our results to establish
\cref{WRTintro,WRTintro2}.

\begin{ackn*}
The author wants to thank Prof. Xiaonan Ma for his constant support,
and Long Mai for a useful discussion on the last section of this paper.
The author also wants to thank the anonymous
referee for useful comments
and suggestions. 
This work was supported by the grant DIM-RDF
from Région Ile-de-France.
\end{ackn*}

\section{Prequantized fibrations and Bergman kernels}\label{defsec}

In \cref{setting}, we introduce the notion of a
prequantized fibration
and construct the associated quantum bundle over its base.
In \cref{famberg},
we recall the results of \cite{MM07} in this setting.
In \cref{trivfib}, we specialize these notions in the case of
a trivialized fibration.

\subsection{Setting}\label{setting}

Let $\pi: M\fl B$ be a submersion of smooth manifolds with compact fibre $X$
such that $\dim X=2n$, and assume $B$ connected. Let $TX$ be
the associated \emph{relative tangent bundle} over $M$, defined as the kernel
of $d\pi$ inside the tangent bundle $TM$ of $M$.

\begin{defi}\label{quantfib}
A submersion $\pi: M\fl B$ as above together with a 
Hermitian line bundle with Hermitian connection
$(L,h^{L},\nabla^{L})$ over $M$ is called a \emph{prequantized
fibration} if the restriction of the curvature
$R^{L}\in\Om^2(M,\C)$ of $\nabla^L$ to $TX\subset TM$
is non degenerate.
\end{defi}

The \emph{relative symplectic form} $\om\in\Om^2(M,\R)$ of the fibration
is defined by \cref{preq},
and induces by restriction a symplectic structure
on the fibres of $\pi$.
For any $x\in M$, let $T^H_x M$ be the subspace of $T_xM$ defined by
\begin{equation}\label{hor}
T^H_x M=\{v\in T_x M~|~\om(v,u)=0,~\forall u\in T_xX\}.
\end{equation}
Then \cref{hor} defines a subbundle $T^HM$ of $T M$, called the \emph{horizontal tangent bundle}. We have
\begin{equation}\label{split}
TM=TX\oplus T^HM,
\end{equation}
and $d\pi$ induces an isomorphism between $T^H M$ and $\pi^*TB$ over $M$.
For any
$x\in M$ and $v\in T_{\pi(x)}B$, we write $v^H_x\in T^H_xM$ for its
horizontal lift via \cref{split}, so that $d\pi.v^H_x=v$.
Let $P^{X},\,P^H$ be the projections from $TM$ to $TX,\,T^H M$, and 
for any form $\alpha\in\Om^\bullet(M,\C)$, write $\alpha^X,\,\alpha^H$
for the
restriction of $\alpha$ to $TX,\,T^H M$. Then \cref{hor} implies
$\om=\om^X+\om^H$ via \cref{split}.

Let $J\in\End(TX)$ over $M$ satisfying $J^2=-\Id_{TX}$ be \emph{compatible} with
$\om$, so that $\om$ is $J$-invariant and that
$g^{TX}:=\om(\cdot,J\cdot)$
defines an Euclidean metric on $TX$ over $M$.
%
%
%
We call $J\in\End(TX)$ a \emph{relative almost complex structure}
over $M$ compatible with $\om$, and $g^{TX}$ the associated
\emph{relative Riemannian metric}.

Let $g^{TB}$ be a metric on $TB$, which we lift to a metric on $T^HM$
via $d\pi$.
Using \cref{split}, we endow $TM$ with the unique metric
$g^{TX}\oplus g^{TB}$ restricting to $g^{TX},\,g^{TB}$ on $TX,\,TB$ and for
which $TX$ and $T^HM$ are orthogonal. Write $\nabla^{TM}$ for the associated
Levi-Civita connection on $TM$, and following Bismut
in \cite[Def.\,1.6]{Bis86}, define the
\emph{vertical Levi-Civita connection} $\nabla^{TX}$
on $TX$ by the formula
\begin{equation}\label{LCTX}
\nabla^{TX}=P^{X}\nabla^{TM}P^{X}\,.
\end{equation}
Then $\nabla^{TX}$ induces the Levi-Civita connection of $g^{TX}$ on the
fibres of $\pi$, and by \cite[Th.1.2]{BGS88}, this definition does
not depend on the choice of $g^{TB}$.

If $h^F$ is a Hermitian metric on a complex vector bundle $F$, we write
$\<\cdot,\cdot\>_F$ for the associated Hermitian product. Let
$TX_\C=TX\otimes_\R \C$ be
the complexification of $TX$. The complex structure
$J$ on $TX$ induces a splitting
\begin{equation}\label{splitc}
TX_\C=T^{(1,0)}X\oplus T^{(0,1)}X
\end{equation}
into the eigenspaces of $J$ corresponding to the eigenvalues $\sqrt{-1}$
and $-\sqrt{-1}$ respectively. We denote by $P^{(1,0)},\,P^{(0,1)}$ the
corresponding projections from $TX_\C$ to $T^{(1,0)}X,\,T^{(0,1)}X$.
Let $h^{T^{(1,0)}X},\,h^{T^{(0,1)}}X$ be the Hermitian metrics on
$T^{(1,0)}X,\,T^{(0,1)}X$ induced by $g^{TX}$
via \cref{splitc}, and define Hermitian connections on
$(T^{(1,0)}X,h^{T^{(1,0)}X}),\,(T^{(0,1)}X,h^{T^{(0,1)}X})$ by
\begin{equation}\label{nab10=Pnab}
\begin{split}
\nabla^{T^{(1,0)}X}&=P^{(1,0)}\nabla^{TX}P^{(1,0)},\\
\nabla^{T^{(0,1)}X}&=P^{(0,1)}\nabla^{TX}P^{(0,1)}.
\end{split}
\end{equation}
The \emph{relative canonical line bundle} is the
line bundle $K_X=\det(T^{(1,0)*}X)$ over $M$,
equipped with $h^{K_X},\,\nabla^{K_X}$ induced by
$h^{T^{(0,1)}X},\,\nabla^{T^{(0,1)}X}$.
Let $(E,h^E,\nabla^E)$ be a Hermitian vector bundle with
Hermitian connection over $M$. For any $p\in\N^*$, write
$L^p$ for the $p$-th tensor power of $L$, and set 
\begin{equation}\label{Ep}
E_p=L^p\otimes E.
\end{equation}
Let $h^{E_p},\nabla^{E_p}$ be induced by $h^L, h^E$ and
$\nabla^L,\nabla^E$ on $E_p$.

For any $b\in B$, set $X_b=\pi^{-1}(b)$. Let $\om_b$ and $g^{TX}_b$ be
the symplectic form and Riemannian metric induced on $X_b$ by 
restriction of $\om$ and $g^{TX}$, and let 
$dv_{X_b}$ be the volume form on $X_b$ defined by
\begin{equation}\label{Liouville}
dv_{X_b}=\frac{\om_b^n}{n!}.
\end{equation}
Then $dv_{X_b}$ is the Riemannian volume form of
$(X_b,g^{TX}_b)$ compatible with the orientation induced by $\om_b$.
Let $L_b,\,E_b,\,E_{p,b},\,T^{(1,0)}X_b$ be the vector
bundles over $X_b$ induced by restriction of
$L,\,E,\,E_p,\,T^{(1,0)}X$.
We can then
see $\cinf(M,E_p)$ as the space of smooth sections of an infinite
dimensional vector bundle with fibre $\cinf(X_b,E_{p,b})$ at $b\in B$.
For any $p\in\N^*$, we endow it with the $L^2$-Hermitian product
$\<\cdot,\cdot\>_p$, defined for any
$b\in B$ and $s_1, s_2\in\cinf(X_b,E_{p,b})$ by
\begin{equation}\label{L2}
\<s_1,s_2\>_{p,b}=\int_{X_b}\<s_1(x),s_2(x)\>_{E_{p,b}} dv_{X_b}(x).
\end{equation}
%
For any $p\in\N^*$, the \emph{Bochner Laplacian} $\Delta^{E_p}$ is the differential operator acting on $\cinf(M,E_p)$ by the formula
\begin{equation}\label{deltalpe}
\Delta^{E_p}=-\sum_{j=1}^{2n}\left[(\nabla^{E_p}_{e_j})^2-\nabla^{E_p}_{\nabla^{TX}_{e_j}e_j}\right],
\end{equation}
where $\{e_j\}_{j=1}^{2n}$ is any local orthonormal frame of $(TX,g^{TX})$.
For any $p\in\N^*$ and any Hermitian endomorphism $\Psi\in\cinf(M,\End(E))$,
the \emph{renormalized Bochner Laplacian} is the second order differential
operator $\Delta_{p,\Psi}$ acting on $\cinf(M,E_p)$ by
\begin{equation}\label{deltapphi}
\Delta_{p,\Psi}=\Delta^{E_p}-2\pi np-\Psi,
\end{equation}
where $\Psi$ denotes the operator of pointwise multiplication by $\Psi$. 
We fix the endomorphism
$\Psi\in\cinf(M,\End(E))$ and simply write $\Delta_p$ for
the associated renormalized Bochner Laplacian.

For any $b\in B$ and $p\in\N^*$, let $L^2(X_b,E_{p,b})$ be the completion
of $\cinf(X_b,E_{p,b})$ with respect to $\<\cdot,\cdot\>_{p,b}$.
Then $\Delta_p$ induces by restriction an elliptic self-adjoint
operator $\Delta_{p,b}$ on $L^2(X_b,E_{p,b})$,
which by standard elliptic theory has discrete spectrum contained in $\R$.
The following theorem comes essentially from \cite[Cor.1.2]{MM02}, and can
be deduced from the manifest uniformity in the parameters of its
proof in \cite[\S\,3]{MM02}.

\begin{theorem}\label{specdeltapphi}
For any $U\subset\subset B$, there exist $\til{C}, C>0$ such that for any
$b\in U$ and $p\in\N^*$,
\begin{equation}\label{spectil}
\Spec(\Delta_{p,b})\subset[-\til{C},\til{C}]\cup[4\pi p-C,+\infty[.
\end{equation}
Furthermore, the direct sum 
\begin{equation}
\HH_{p,b}=\bigoplus_{\substack{\lambda\in\Spec(\Delta_{p,b})\\ \lambda\in [-\til{C},\til{C}]}}\Ker(\lambda-\Delta_{p,b})
\end{equation}
is naturally included in $\cinf(X_b,E_{p,b})$, and there is $p_0\in\N$
such that for any $b\in U$ and $p\geq p_0$, we have
\begin{equation}\label{RRH}
\dim\HH_{p,b}=\int_X\Td(T^{(1,0)}X_b)\ch(E_b)\exp(p\om_b),
\end{equation}
where $\Td(T^{(1,0)}X_b)$ represents the Todd class of $T^{(1,0)}X_b$ and $\ch(E_b)$ represents the Chern character of $E_b$.
\end{theorem}
The right hand side of \cref{RRH} is an integer depending
smoothly on $b\in B$, thus
locally constant in $b\in B$. In view of \cref{specdeltapphi},
we will assume $B$ compact from now on, and fix $p_0\in\N$ with
$\til{C}<4\pi p_0-C$ and such that \cref{RRH} is satisfied
for all $p\geq p_0$ and $b\in B$. As $B$ is connected, this implies
that $\dim\HH_{p,b}$ does not depend on $b\in B$.
Following \cite[\S\,9.2]{BGV04},
we can define the orthogonal projection operator $P_{p,b}$ from
$\cinf(X_b,E_{p,b})$ to $ \HH_{p,b}$ with respect to the
$L^2$-Hermitian product \cref{L2} by the following contour integral
in the complex plane,
\begin{equation}
P_{p,b}=\int_\Gamma \left(\lambda-\Delta_{p,b}\right)^{-1}d\lambda,
\end{equation}
where $\Gamma$ is a circle of center $0$ and radius $a>0$ such that
$\til{C}<a<4\pi p-C$. This shows that the projection operators
$P_{p,b}$ depend smoothly on $b\in B$, and as the dimension of
$\Im(P_{p,b})=\HH_{p,b}$ is constant in $b\in B$, this defines
a finite
dimensional bundle over $B$ with fibre $\HH_{p,b}$ at $b\in B$.
We write this bundle 
$\HH_p$, and we endow it with the Hermitian structure $h^{\HH_p}$
given by the
$L^2$-Hermitian product \cref{L2} restricted to $\HH_{p,b}$ for all $b\in B$.
By convention, we take $\HH_{p,b}=\{0\}$ for all $b\in B$,
so that $\HH_p$ is the trivial bundle over $B$, whenever $p<p_0$.
The family $\{\HH_p\}_{p\in\N^*}$ of Hermitian vector
bundles over $B$ is called the \emph{quantum bundle} of
$\pi:M\fl B$.

As $P_{p,b}$ is a projection operator on a finite dimensional
space for all
$p\in\N^*$, it has smooth Schwartz kernel with respect
to $dv_{X_b}$, depending smoothly on $b\in B$. Specifically, let
$M\times_B M$ be the fibred product of $M$ with itself over $B$, and
write $\pi_1,\,\pi_2$ for the first and second projections from
$M\times_B M$ to $M$. For any vector bundles $F_1$ and $F_2$
over $M$, we define the fibred tensor product of $F_1$ and $F_2$
over $B$ as a vector bundle $F_1\boxtimes_B F_2$ over $M\times_B M$
by the formula
\begin{equation}\label{boxdef}
F_1\boxtimes_B F_2=\pi_1^*F_1\otimes\pi_2^*F_2.
\end{equation}
Then there exists $P_p(.,.)\in\cinf(M\times_B M,E_p\boxtimes_BE_p^*)$
such that for all $b\in B,\,s\in\cinf(X_b,E_{p,b})$ and $x\in X_b$, we get
\begin{equation}\label{Bergdef}
\left(P_{p,b}s\right)(x)=\int_{X_b}P_p(x,y)s(y)dv_{X_b}(y).
\end{equation}
We call $P_p(\cdot,\cdot)$ the relative \emph{Bergman kernel} of $E_p$
over $\pi:M\fl B$. Taking $s\in\cinf(M,E_p)$ in
\cref{Bergdef}, this defines in turn a projection operator $P_p$ from
$\cinf(M,E_p)$ to $\Ker(\Delta_p)=\cinf(B,\HH_p)$, called the relative
\emph{Bergman projection} on $\HH_p$ over $\pi:M\fl B$.
On the other hand, for any $K_{1,p}(\cdot,\cdot),\,K_{2,p}(.,.)\in\cinf(M\times_B M,E_p\boxtimes_BE_p^*)$, we write $K_{1,p}K_{2,p}(.,.)\in\cinf(M\times_B M,E_p\boxtimes_BE_p^*)$ for the kernel defined for any $b\in B,\,x,y\in X_b$ by
\begin{equation}\label{compoker}
K_{1,p}K_{2,p}(x,y)=\int_X K_{1,p}(x,w)K_{2,p}(w,y)dv_{X_b}(w).
\end{equation}
%
%
%

For any $v\in\cinf(B,TB)$,
consider $\nabla^{E_p}_{v^H}$ as a first order differential
operator acting on $\cinf(M,E_p)$.
This can be seen as the contraction with $v\in\cinf(B,TB)$
of a connection on the
infinite dimensional vector bundle of fibrewise smooth sections.
The \emph{$L^2$-Hermitian connection} on $\HH_p$
over $B$ is defined for any $v\in\cinf(B,TB)$ by
\begin{equation}\label{connectionL2}
\nabla^{\HH_p}_v=P_p\nabla^{E_p}_{v^H} P_p.
\end{equation}
By an argument of \cite[Th.1.14]{BGS88}, $\nabla^{E_p}_{v^H}$ preserves
the $L^2$-Hermitian product \cref{L2}, so that $\nabla^{\HH_p}$
is a Hermitian connection on $(\HH_p,h^{\HH_p})$.

A prequantized fibration is called \emph{holomorphic}
if $\pi:M\fl B$ is a holomorphic
submersion between complex manifolds, and if $(L,h^L),\,(E,h^E)$ are
holomorphic Hermitian bundles equipped with their
\emph{Chern connections} $\nabla^L,\,\nabla^E$, which are the unique Hermitian
connections preserving the holomorphic structure.
The relative complex structure
$J\in\End(TX)$ induced by the natural holomorphic
structure in the fibres is then compatible with $\om$,
and this makes $\pi:M\fl B$ into
a \emph{Kähler fibration} in the sense of \cite[Def.1.4]{BGS88}.
Set $\Psi=-\i\Lambda_\om R^E$ in \cref{deltapphi}, where $\Lambda_\om$
denotes the contraction by $\om$.
%
%
For a proof of the following proposition, 
we refer to \cite[\S\,1.4.3]{MM07}, \cite[Th.1.1]{BV89},
\cite[Th.3.2]{BK92} and the references therein.

\begin{prop}\label{hol}
Let $\pi:M\fl B$ be a holomorphic prequantized fibration with $B$ compact.
Then there exists $p_0\in\N$ such that for any $p\geq p_0$
and $b\in B$, the space $\HH_{p,b}$
coincides with the space of holomorphic sections of $E_{p,b}$ inside
$\cinf(X_b,E_{p,b})$, and the bundle
$\HH_p$ over $B$ has a natural holomorphic structure,
for which the $L^2$-Hermitian 
connection $\nabla^{\HH_p}$ on $\HH_p$
is the Chern connection of $(\HH_p,h^{\HH_p})$.
\end{prop}
Let us finally set the following convention : for any $k\in\N$, if
$A,\,B\in\End(\R^k)$ are symmetric matrices with $A$ positive,
we define $\det{}^{\frac{1}{2}}(A+\sqrt{-1}B)$ to be the square root of
$\det(A+\sqrt{-1}B)$ determined by analytic continuation along the
path $t\mapsto A+t\sqrt{-1}B$ for $t\in[0,1]$.

\subsection{Bergman kernels}\label{famberg}
In this section, we will use systematically the fact that all the
estimates described in \cite[Chap.4, \S\,8.3.2]{MM07} are uniform with
respect to parameters.
Let $\pi:M\fl B$ be a prequantized fibration equipped with a relative
almost complex structure $J\in\End(TX)$ compatible with $\om$. We assume $B$ compact, and
fix $p_0\in\N$ as in \cref{specdeltapphi} for $U=B$.

For any $m\in\N$ and all $p\in\N^*$, we denote by $|\cdot|_{\CC^m}$ the
$\CC^m$-norm on $E_p\boxtimes_BE_p^*$ induced by
$h^L, h^E, \nabla^L, \nabla^E$. Recall that $X_b$ denotes the
fibre of $\pi:M\fl B$ at $b\in B$, and $g^{TX}_b$ the restriction of
$g^{TX}$ to $X_b$. Let $d^{X_b}(\cdot,\cdot)$ be the Riemannian distance on
$(X_b,g^{TX}_b)$. The following result comes essentially from
\cite[(2.5)]{LMM16}, built on arguments of
\cite[\S\,8.3.3]{MM07},\,\cite[\S\,1.1]{MM08a}.

\begin{prop}\label{theta}
For any $m,k\in\N,\,\epsilon>0$ and $\theta\in\,]0,1[\,$, there is $C>0$ such that for all $p\in\N^*,\,b\in B$ and $x,x'\in X_b$ satisfying $d^{X_b}(x,x')>\epsilon p^{-\frac{\theta}{2}}$, the following estimate holds,
\begin{equation}
|P_p(x,x')|_{\CC^m}\leq Cp^{-k}.
\end{equation}
\end{prop}
\cref{theta} shows that the study of the asymptotics in $p\in\N^*$
of the Bergman kernel localizes around the diagonal of the fibres.
To describe asymptotic estimates for the Bergman kernel near the diagonal,
we will need the following definition.

\begin{defi}\label{chart}
Let $\psi$ be a smooth map from the total space of $TX$ to $M$, and for any $x\in M$, write $\psi_x:T_xX\fl M$ for the restriction of $\psi$ to $T_xM$. We say that $\psi$ is a \emph{smooth family of vertical charts} if for any $x\in M$, the image of $\psi_x$ is included in $X_{\pi(x)}$, such that $\psi_x(0)=x$ and its differential $d\psi_{x,0}:T_xX\fl T_xM$ at $0$ induces the canonical injection of $T_xX$ into $T_xM$.
\end{defi}

Using partitions of unity, it is easy to see that smooth families of
vertical charts always exist. We fix one of them for the rest of the section. 
For any $x\in M$ and $\epsilon>0$, let $|\cdot|_x$ be the norm of
$(T_xX,g^{T_xX})$ and let $B^{T_xX}(0,\epsilon)$ be the open ball
in $T_xX$ of radius $\epsilon$. Choose $\epsilon_0>0$ so small that for
any $x\in M$, the map $\psi_x$ induces a diffeomorphism between
$B^{T_xX}(0,\epsilon_0)$ and a neighborhood $U_x$ of $x$ in $X_{\pi(x)}$.
Identify $L,\,E$ over $U_x$ with $L_x,\,E_x$ through parallel
transport with respect to $\nabla^{L},\,\nabla^{E}$ along radial lines
of $B^{T_xX}(0,\epsilon_0)$ in the chart induced by $\psi_x$.
Pick a unit section $e_x\in L_x$ and use it to identify $L_x$ with $\C$.
In this way, the vector bundle $E_p$ is identified with $E_x$ over $U_x$
for all $p\in\N^*$, and this identification can be made smoothly in $x\in M$.
Under the natural isomorphism $\End(L^p)\simeq\C$, our formulas do not
depend on this identification.

For any $p\in\N^*$ and any kernel
$K_p(\cdot,\cdot)\in\cinf(M\times_B M,E_p\boxtimes_BE_p^*)$,
we write
$K_{p,x}(Z,Z')\in\End(E_x)$
for its image in this trivialization evaluated at
$Z,Z'\in B^{T_xX}(0,\epsilon_0)$.
Note that $K_{p,x}$ is
a section of the pullback bundle $\pi_0^*\End(E)$ over an open set of
the fibred product $\pi_0:TX\times_M TX\fl M$ of the total space of $TX$ with
itself over $M$. On the other hand, if $\cal{K}\in\pi_0^*\End(E)$
is such a section and $F_x(Z,Z')\in\End(E_x)$
is polynomial in $Z,Z'\in T_xX$ for any $x\in M$,
we write
\begin{equation}
F\cal{K}_x(Z,Z')=F_x(Z,Z')\cal{K}_x(Z,Z').
\end{equation}
%
Recalling that two Riemannian metrics induce equivalent distances
in a continuous way with respect to parameters, we can precise
\cref{theta} in the following way.
\begin{cor}\label{thetagal}
For any $m,k\in\N,\,\epsilon>0$ and $\theta\in \,]0,1[\,$, there is $C>0$ such that for all $p\in\N^*$ and $x\in X,\,Z,Z'\in B^{T_xX}(0,\epsilon_0)$ such that $|Z-Z'|_x>\epsilon p^{-\frac{\theta}{2}}$, the following estimate holds,
\begin{equation}
|P_{p,x}(Z,Z')|_{\CC^m}\leq Cp^{-k}.
\end{equation}
\end{cor}
We use the following explicit local model for the Bergman kernel
from \cite[(3.25)]{MM08b}, for any $x\in M,\,Z,Z'\in T_xX$,
\begin{equation}\label{PPreal}
\PP_x(Z,Z')=\exp\left(-\frac{\pi}{2}|Z-Z'|_x^2-\pi\sqrt{-1}\om_x^X(Z,Z')\right).
\end{equation}
Let $|\cdot|_{\CC^m(M)}$ denote the $\CC^m$-norm on $\pi_0^*\End(E)$
over $TX\times_M TX$ induced by $h^E$ and by derivation by
$\nabla^{\pi_0^*\End(E)}$ in the direction of $M$ via the Levi-Civita
connection. We can now state the following fundamental result on the 
near diagonal expansion of the Bergman kernel,
which was first established in
\cite[Th.4.18']{DLM06} for spin$^c$ Dirac
operators, hence in particular in the Kähler case. In this
form, it comes essentially from \cite[Th.2.1]{LMM16}.
\begin{prop}\label{asy}
There is a family $\{J_{r,x}(Z,Z')\in\End(E_x)\}_{r\in\N}$ of
polynomials in $Z,Z'\in T_xX$ of the same parity as $r$ and smooth in
$x\in M$, such that for any $\epsilon>0,\,k,m,m'\in\N$ and
$\delta\in\,]0,1[\,$, there is $\theta\in\,]0,1[$ and $C>0$
such that for all
$x\in M,\,p\in\N^*,
\,|Z|,|Z'|<\epsilon_0 p^{-\frac{\theta}{2}}$, 
\begin{multline}\label{asyfla}
\sup_{|\alpha|+|\alpha'|\leq m}\Big|\Dk{\alpha}{Z}\Dk{\alpha'}{Z'}\big(p^{-n}P_{p,x}(Z,Z')\\
-\sum_{r=0}^{k-1} J_r\PP_x(\sqrt{p}Z,\sqrt{p}Z')\left.p^{-\frac{r}{2}}\big)\right|_{\CC^{m'}(M)}\leq Cp^{-\frac{k-m}{2}+\delta}.
\end{multline}
Furthermore, for all $x\in M,~Z,\,Z'\in T_xX$, we have $J_{0,x}(Z,Z')=\Id_{E_x}$.


\end{prop}
%
%
We will often consider expansions of this type in the sequel.
To this end, we introduce the following notation.

\begin{nota}\label{notcong}
For any family $\{K_p(.,.)\in\cinf(M\times_B M,E_p\boxtimes_BE_p^*)\}_{p\in\N^*}$, we write
\begin{equation}\label{cong}
p^{-n}K_{p,x}(Z,Z')\cong\sum\limits_{r=0}^{\infty}\Q_r\cK_x(\sqrt{p}Z,\sqrt{p}Z')p^{-\frac{r}{2}} + \cO(p^{-\infty}),
\end{equation}
for $\cal{K}_x(Z,Z')\in\End(E_x)$ smooth in $x\in M,\,Z,Z'\in T_xX$, and
$\{\Q_{r,x}(Z,Z')\in\End(E_x)\}_{r\in\N}$ a family of polynomials in
$Z,Z'\in T_xX$, smooth in $x\in M$, if there exists $\epsilon>0$ such
that for any $k,m,m'\in\N$ and $\theta\in\,]0,1[\,$, there is $C>0$,
such that for all $x\in M$
and $Z,Z'\in B^{T_xX}(0,\epsilon_0)$ with
$|Z-Z'|_x>\epsilon p^{-\frac{\theta}{2}}$,
\begin{equation}\label{thetacong}
|K_{p,x}(Z,Z')|_{\CC^m}\leq Cp^{-k}.
\end{equation}
and such that for all $\delta\in\,]0,1[$, there is
$\theta\in\,]0,1[$ and $C>0$ such that 
for all $x\in M,\,|Z|,|Z'|<\epsilon p^{-\frac{\theta}{2}}$,
\begin{multline}\label{defcong}
\sup_{|\alpha|+|\alpha'|\leq m}\left|\Dk{\alpha}{Z}\Dk{\alpha'}{Z'}
\big(\right.p^{-n}K_{p,x}(Z,Z')\\
-\sum\limits_{r=0}^{k-1}\Q_r\left.\cK_x(\sqrt{p}Z,\sqrt{p}Z') p^{-\frac{r}{2}}
\big)\right|_{\CC^{m'}(M)}\leq Cp^{-\frac{k-m}{2}+\delta}.
\end{multline}
\end{nota}
%
%
For any $f\in\cinf(M,\End(E))$, the 
\emph{Berezin-Toeplitz quantization} of $f$ is the family
of endomorphisms of $\HH_p$ for all $p\in\N^*$,
defined for any $b\in B$ by the formula
\begin{equation}\label{BTdef}
P_{p,b}fP_{p,b}:\cinf(X_b,E_{p,b})\fl\cinf(X_b,E_{p,b}),
\end{equation}
where $f$ denotes the operator of pointwise multiplication by $f$. By \cref{compoker}, this operator
 admits a smooth Schwarz kernel with respect to $dv_{X_b}$.
The following asymptotic expansion in $p\in\N^*$
is a consequence of
\cref{thetagal} and \cref{asy}. It was first
established in \cite[Lem.4.6]{MM07} for spin$^c$ Dirac
operators, hence in particular in the Kähler case. In this
form, it comes essentially from \cite[Lem.3.3]{ILMM17}.

\begin{lem}\label{Toepasy}
There is a family $\{\Q_{r,x}(f)(Z,Z')\in\End(E_x)\}_{r\in\N}$ of polynomials in $Z,Z'\in T_xX$ of the same parity as $r$, smooth in $x\in M$, such that
\begin{equation}
p^{-n}(P_{p,b}fP_{p,b})(Z,Z')\cong\sum\limits_{r=0}^{\infty}\Q_r(f)\PP_x(\sqrt{p}Z,\sqrt{p}Z')p^{-\frac{r}{2}} + \cO(p^{-\infty}).
\end{equation}
Furthermore, for all $x\in M,\,Z,Z'\in T_xX$, we have $\Q_{0,x}(f)(Z,Z')= f(x)$.
\end{lem}
%

\subsection{Trivialization of fibrations}\label{trivfib}

In \cref{pt}, we will be mainly concerned with the study of prequantized
fibrations restricted over paths parametrized by $t\in [0,1]$.
Keeping this in mind, we consider in this section the case of a
prequantized fibration $\pi:M\fl[0,1]$.

Let $\partial_t$ be the canonical vector field of $[0,1]$ and
recall that $\partial^H_t$ denotes
its horizontal lift to $T^HM$ in $TM$. Set $X:=\pi^{-1}(0)$.
Then there is a unique diffeomorphism $\tau$ between $[0,1]\times X$ and $M$ such that for any $t\in[0,1]$ and $x_0\in X$,
\begin{equation}\label{tausmall}
\tau(0,x_0)=x_0\quad\text{and}\quad\dt\tau(t,x_0)=\partial^H_{t,x_0}.
\end{equation}
The fibration coincides via $\tau$ with
the first projection
$\pi:[0,1]\times X\fl X$.
By \cref{preq} and \cref{hor}, for any vertical
vector field $v\in\cinf(M,TX)$, we have
\begin{equation}\label{RL(H,v)=0}
R^L(\partial^H_t,v)=0.
\end{equation}
This shows parallel transport with respect to $\nabla^L$ 
along horizontal paths of $[0,1]\times X$ via $\tau$ identifies
$(L,h^L,\nabla^L)$ over $[0,1]\times X$ with the pullback
of a fixed Hermitian line bundle with connection over $X$,
which we still denote by $(L,h^L,\nabla^L)$. In particular, we have
\begin{equation}\label{trivnabL}
\nabla^L_{\delt^H}=\dt\quad\text{and}\quad \left[\dt,\nabla^L\right]=0.
\end{equation}
By \cref{preq}, we deduce that $\om$ is identified via $\tau$ with
the pullback of a fixed symplectic form $\om^X\in\Om^2(X,\R)$ over $X$,
so that the restriction of $\om$ to $X_t\simeq X$ does not depend on
$t\in[0,1]$.
Then the volume form on the fibre $X$ induced by $\om$ as in \cref{Liouville}
does not depend on $t\in[0,1]$ either.
We write $J_t,\,g^{TX}_t,\,E_t$ for the restriction of $J,\,g^{TX},\,E$
to $X\simeq X_t$ over $t\in[0,1]$.
Let
$\tau^E_t:E_0\fl E_t$ be the vector bundle isomorphism induced by
parallel transport with respect to $\nabla^E$ along horizontal curves of
$[0,1]\times X$.

Let $\psi$ be a smooth family of vertical charts over
$\pi:[0,1]\times X\fl [0,1]$ such that its restriction
$\psi_t:TX\fl X$ over the fibre at
$t\in[0,1]$ does not depend on $t$. Fix $x_0\in X$, and consider the
trivialization around $x_0$ as in \cref{famberg}.
Identify $T_{x_0}X$ with $\R^{2n}$ using an orthonormal basis
$\{e_j\}_{j=1}^n$ of $(T_{x_0}X,g^{TX}_0)$ such that
\begin{equation}\label{basis}
J_0 e_{2j}=e_{2j+1}\quad\text{and}\quad J_0e_{2j+1}=-e_{2j}.
\end{equation}
Then $\om^X_{x_0}$ induces the standard symplectic form $\Om$ on $\R^{2n}$
in this identification. Write $|\cdot|_t$ for the norm on $\R^{2n}$ induced
by $g^{TX}_t$, so that in particular $|\cdot|_0$ is the standard norm of
$\R^{2n}$. Then the local model \cref{PPreal} of the Bergman kernel on
$(T_{x_0}X,J_t)$ becomes
\begin{equation}\label{PPtriv}
\PP_{t,x_0}(Z,Z')=\exp\left(-\frac{\pi}{2}|Z-Z'|_t^2-\pi\sqrt{-1}\Om(Z,Z'))\right),
\end{equation}
for all $Z,Z'\in\R^{2n}$ and $t\in[0,1]$.
%

We end this section with the following definition, which generalizes
the setting described above.

\begin{defi}\label{trivfibdef}
A prequantized fibration $\pi:M\fl B$ is said to be
\emph{tautological}
if there is a fibration map $M\simeq B\times X$ such that
the associated 
line bundle $(L,h^L,\nabla^L)$ over $M$ is the pullback of a
Hermitian
line bundle with connection over $X$ by the second projection
$\pi_2:B\times X\fl X$.
\end{defi}

If $J\in\End(TX)$ is a relative compatible almost complex structure
over a tautological fibration, we write
$J_b$ for the restriction of $J$ to $X\simeq X_b$ over $b\in B$.
Then $J$ can be seen as a family of compatible
almost complex structures on a 
fixed symplectic manifold $(X,\om^X)$, depending smoothly on $b\in B$.

\section{Toeplitz operators}\label{genTeopsec}

In \cref{locmod}, we study the local model for the
composition of Bergman
kernels associated with different complex structures
via the identifications of \cref{trivfib}.
In \cref{crit}, we study the
corresponding generalization of a Toeplitz operator.
In \cref{pt}, we use these results to show that the parallel
transport in the quantum
bundle over a path of complex structures is a such
a Toeplitz operator, proving \cref{Approxintro}.

\subsection{Local model}\label{locmod}

Let $Z:=(Z_1,\dots, Z_{2n})\in\R^{2n}$ denote the real coordinates of $\R^{2n}$. Let $\<\cdot,\cdot\>$ be the canonical scalar product on $\R^{2n}$, and write $|\cdot|$ for the associated norm. We write $J_0\in\End(\R^{2n})$ for the complex structure defined for all $1\leq j\leq n$ by
\begin{equation}
J_0\D{Z_{2j}}=-\D{Z_{2j-1}}\quad\text{and}\quad J_0\D{Z_{2j-1}}=\D{Z_{2j}}.
\end{equation}
Then $\<\cdot,\cdot\>$ is $J_0$-invariant, and
$\Om(\cdot,\cdot)=\<J_0\cdot,\cdot\>$ defines a $J_0$-invariant
antisymmetric non-degenerate form on $\R^{2n}$, called the
\emph{canonical symplectic form}.

Let now $J_t\in\End(\R^{2n})$ satisfying $J_t^2=-\Id_{\R^{2n}}$
be a smooth one-parameter family of
compatible complex structures on $\R^{2n}$,
so that $\Om$ is $J_t$-invariant and the formula
$\<\cdot,\cdot\>_t=\Om(\cdot,J_t\cdot)$
defines a scalar product on $\R^{2n}$, for all $t\in\R$. Note that
$\<\cdot,\cdot\>_t$ and $\<\cdot,\cdot\>$ are related by
\begin{equation}\label{crochett}
\<\cdot,\cdot\>_t=\left\langle(-J_0J_t)\cdot,\cdot\right\rangle.
\end{equation}
In particular $-J_0J_t\in\End(\R^{2n})$ is positive symmetric, as well as
its inverse $-J_tJ_0$. We write $|\cdot|_t$ for the norm induced on $\R^{2n}$ 
by $\<\cdot,\cdot\>_t$.

Recall the local model \cref{PPtriv} for the Bergman kernel in $\R^{2n}$
associated to $J_t$ for any $t\in\R$, which is given for any
$Z,Z'\in\R^{2n}$ by
\begin{equation}\label{PPttriv}
\begin{split}
\PP_t(Z,Z')&=\exp\left(-\frac{\pi}{2}|Z-Z'|_t^2-\pi\sqrt{-1}\Om(Z,Z'))\right)
\\
&=\exp\left(-\frac{\pi}{2}\left\langle(-J_0J_t)(Z-Z'),(Z-Z')\right\rangle
-\pi\sqrt{-1}\Om(Z,Z'))\right).
\end{split}
\end{equation}
In particular, we have $\PP_t(Z,Z')=\overline{\PP_t(Z',Z)}$.
Note that the canonical Lebesgue measure $dZ$ of $\R^{2n}$ is induced by the 
Liouville form of $\Om$, and thus corresponds to the Riemannian volume form of $\<\cdot,\cdot\>_t$ for all $t\in [0,1]$.
For any $t\in\R$, let $\HH_t\subset L^2(\R^{2n})$ be defined by
\begin{equation}
\HH_t:=\{f\in L^2(\R^{2n})~|~Z\mapsto f(Z)e^{\frac{\pi}{2}|Z|^2}
\ \text{is holomorphic for}\ J_t\}.
\end{equation}
Then as explained in \cite[\S\,1.4]{MM08a}, $\PP_t(\cdot,\cdot)$
is the Schwartz kernel with respect to $dZ$
of the orthogonal projection $\PP_t:L^2(\R^n)\fl\HH_t$, and in particular
we have $\PP_t\PP_t=\PP_t$.
The Schwartz kernel of the composition $\PP_t\PP_0$
is given for any $Z,Z'\in\R^{2n}$ by the formula
\begin{equation}\label{PtP0def}
\PP_t\PP_0(Z,Z')=\int_{\R^{2n}}\PP_t(Z,\til{Z})\PP_0(\til{Z},Z')d\til{Z}.
\end{equation}
Let $\C^{2n}=V^{(1,0)}_t\oplus V^{(0,1)}_t$
be the splitting of $\C^{2n}=\R^{2n}\otimes_\R\C$
into the eigenspaces of $J_t$ corresponding
to the eigenvalues $\sqrt{-1}$ and $-\sqrt{-1}$.
The corresponding projections $P_t^{(1,0)},\,P_t^{(0,1)}$ from $\C^{2n}$
to $V^{(1,0)}_t,\,V^{(0,1)}_t$ are given by the formulas
\begin{equation}\label{P10t}
P^{(1,0)}_t=\frac{1-\i J_t}{2}\quad\text{and}\quad P^{(0,1)}_t=\frac{1+\i J_t}{2}.
\end{equation}
For any $t\in\R$, define the symmetric positive endomorphisms
\begin{equation}\label{At0}
A_{t}^0=\left(\frac{\Id-J_0J_t}{2}\right)^{-1}\quad\text{and}\quad
A_{0}^t=\left(\frac{\Id-J_tJ_0}{2}\right)^{-1}.
\end{equation}
and set
\begin{equation}\label{Pi0t}
\Pi_t^0=A_{t}^0 P_0^{(1,0)}\quad\text{and}\quad
\Pi_0^t=A_0^t P_t^{(1,0)}.
\end{equation}
Then we have the following identities,
\begin{equation}\label{Pi0t'}
\begin{split}
\Pi_t^0 P_t^{(1,0)}=P_t^{(1,0)} \quad &\text{and}\quad
\Pi_0^t P_0^{(1,0)}=P_0^{(1,0)},\\
 P_t^{(1,0)}\Pi_t^0=\Pi_t^0 \quad &\text{and}\quad
 P_0^{(1,0)}\Pi_0^t=\Pi_0^t,
\end{split}
\end{equation}
which, together with \cref{Pi0t}, show that
$\Pi_t^0,\,\Pi_0^t\in\End(\C^{2n})$
are precisely the projection operators
onto $V_t^{(1,0)},\,V_0^{(1,0)}$ with kernel $V_0^{(0,1)},\,V_t^{(0,1)}$.
In particular, they induce a splitting
\begin{equation}\label{splitt0}
\C^{2n}=V_t^{(1,0)}\oplus V_0^{(0,1)},
\end{equation}
for all $t\in\R$.
Their complex conjugates
$\overline{\Pi_t^0},\,\overline{\Pi_0^t}\in\End(\C^{2n})$ are the
projection operators onto
$V_t^{(0,1)},\,V_0^{(0,1)}$ with kernel $V_0^{(1,0)},\,V_t^{(1,0)}$,
and induce a splitting of $\C^{2n}$ which is complex conjugate to
\cref{splitt0}.
Note that these projections all have positive symmetric real part
given by \cref{At0}.

The following result is an explicit computation of the kernel \cref{PtP0def}, which is going to be our local model in the next section.
\begin{lem}\label{PtP0}
For any $Z,Z'\in\R^{2n}$, the following formula holds
\begin{equation}\label{PtP0fla}
\PP_t\PP_0(Z,Z')=
\det(A_t^0)^{\frac{1}{2}}\exp\big(-\pi\left[\left\langle\Pi_0^t(Z-Z'),(Z-Z')\right\rangle+\i\Om(Z,Z')\right]\big).
\end{equation}
Furthermore, for any $F(Z)\in\C[Z]$ homogeneous,
there exists $\Q(F)(Z)\in\C[Z]$
of the same parity such that for any $Z,Z'\in\R^{2n}$,
\begin{equation}\label{PtP0Ffla}
\int_{\R^{2n}}\PP_t(Z,\til{Z})F(\til{Z})\PP_0(\til{Z},Z')d\til{Z}
=\Q(F)(Z)\PP_t\PP_0(Z,Z').
\end{equation}
Finally, for any $B\in\End(\C^{2n})$ and $Z,Z'\in\R^{2n}$,
the following formulas hold,
\begin{equation}\label{PtP0Bfla}
\begin{split}
\int_{\R^{2n}}\PP_t(Z,\til{Z})&\<B(Z-\til{Z}),(Z-\til{Z})\>\PP_0(\til{Z},Z')d\til{Z}\\
=&\left(\<B\overline{\Pi_t^0}(Z-Z'),\overline{\Pi_t^0}(Z-Z')\>+\frac{1}{2\pi}\Tr[A_t^0B]\right)\PP_t\PP_0(Z,Z'),\\
\int_{\R^{2n}}\PP_t(Z,\til{Z})&\<B(Z'-\til{Z}),(Z'-\til{Z})\>
\PP_0(\til{Z},Z')d\til{Z}\\
=&\left(\<B\Pi_0^t(Z-Z'),\Pi_0^t(Z-Z')\>+\frac{1}{2\pi}\Tr[A_0^tB]\right)\PP_t\PP_0(Z,Z').
\end{split}
\end{equation}
\end{lem}
\begin{proof}
Let $Z,Z'\in\R^{2n}$ be fixed, and recall that
$\Om(\cdot,\cdot)=\<J_0\cdot,\cdot\>$. Through the change of variable
$\til{Z}\mapsto\til{Z}+Z$ and using \cref{PtP0def}, we get
\begin{multline}\label{PtP0comput1}
\PP_t\PP_0(Z,Z')=\int_{\R^{2n}}\exp\Big(-\frac{\pi}{2}\big[\<(-J_0J_t)\til{Z},\til{Z}\>+|\til{Z}+(Z-Z')|^2\\
+2\sqrt{-1}\Om(Z,\til{Z})+2\sqrt{-1}\Om(\til{Z}+Z,Z')\big]\Big)d\til{Z}\\
=\int_{\R^{2n}}\exp\Big(-\frac{\pi}{2}\big[\<(\Id-J_0J_t)\til{Z},\til{Z}\>
+4\<\til{Z},P^{(0,1)}_0(Z-Z')\>\big]\Big)d\til{Z}\\
\exp\Big(-\frac{\pi}{2}\big[|Z-Z'|^2+2\sqrt{-1}\Om(Z,Z')\big]\Big).\\
\end{multline}
By \cref{At0} and the classical formula for Gaussian
integrals, we then get
\begin{multline}\label{PtP0comput3'}
\PP_t\PP_0(Z,Z')=\det(A_t^0)^{\frac{1}{2}}\exp\Big(\pi
\<A_t^0P^{(0,1)}_0(Z-Z'),P^{(0,1)}_0(Z-Z')\>\Big)\\
\exp\Big(-\frac{\pi}{2}\big[|Z-Z'|^2+
2\sqrt{-1}\Om(Z,Z')\big]\Big).
\end{multline}
%
%
From \cref{At0}, we know that $-J_0A_t^0J_0=(-J_0J_t)A_t^0=A_0^t$,
and in particular, we get the identities
\begin{equation}\label{At0+A0t=2}
\begin{split}
A_t^0+A_0^t&=(1+(-J_0J_t))A_t^0=2\,\Id_{\R^{2n}}\,\\
A_t^0J_0&=J_0 A_0^t=A_0^tJ_t\,,\\
P_0^{(1,0)}A_t^0P_0^{(1,0)}&=\frac{1}{4}\left(A_t^0-A_0^t-\sqrt{-1}J_0A_t^0
+\sqrt{-1}A_t^0J_0\right)\,.
\end{split}
\end{equation}
Recall that $A_t^0$ and $A_0^t$ are symmetric. Using
\cref{P10t}, \cref{Pi0t} and these identities, we can rewrite
\cref{PtP0comput3'} into
\begin{equation}\label{PtP0comput3}
\begin{split}
& \PP_t\PP_0(Z,Z')=\det(A_t^0)^{\frac{1}{2}}\exp\Big(-\frac{\pi}{4}\<(A_0^t-A_t^0)(Z-Z'),(Z-Z')\>\Big)\\
&\exp\Big(-\frac{\pi}{4}\big[-2\sqrt{-1}\<A_t^0J_0(Z-Z'),Z-Z'\>\big]\Big)\\
&\exp\Big(-\frac{\pi}{4}\big[\<(A_0^t+A_t^0)(Z-Z'),Z-Z'\>+4\sqrt{-1}\Om(Z,Z')\big]\Big)\\
&=\det(A_t^0)^{\frac{1}{2}}\exp\big(-\pi\left[\left\langle\Pi_0^t(Z-Z'),(Z-Z')\right\rangle+\i\Om(Z,Z')\right]\big).
\end{split}
\end{equation}
This implies \cref{PtP0fla}. The computations leading to \cref{PtP0Ffla}
and \cref{PtP0Bfla} are analogous to \cref{PtP0comput1}-\cref{PtP0comput3},
applying the classical formula for the
integral of a Gaussian function multiplied by a polynomial.
Note that the second equality of \cref{PtP0Bfla} can be deduced from the
first using $\PP_t(Z,Z')=\overline{\PP_t(Z',Z)}$ and exchanging the roles
of $0$ and $t$.
\end{proof}

We recover from \cref{PtP0fla} with $t=0$ the identity
$\PP_0\PP_0=\PP_0$ characterizing projection operators.
For $t=0$, the formulas \cref{PtP0Ffla} and \cref{PtP0Bfla} are consequences
of \cite[\S\,2]{MM08b}.
For any $t\in\R$, set
\begin{equation}\label{tilmutdef}
\mu_t=\exp\left(\int_0^t \frac{1}{4}\Tr\left[\Pi_u^0\du(-J_0J_u)\right] du\right).
\end{equation}
We deduce from \cref{PtP0} the following local model for the parallel transport in the bundle of holomorphic sections along the path $t\mapsto J_t$
for $t\in\R$.
\begin{prop}\label{ptloc}
For any $t\in[0,1]$, the following formula holds,
\begin{equation}\label{mut}
\PP_t\left(\dt\PP_t\PP_0\right)=-\frac{1}{4}\Tr\left[\Pi_t^0\dt(-J_0J_t)\right]\PP_t\PP_0.
\end{equation}
In particular, we have
\begin{equation}\label{tilmut}
\PP_t\dt \left(\mu_t\PP_t\PP_0\right)=0.
\end{equation}
Furthermore, the following equality holds,
\begin{equation}\label{barmut}
\PP_t\PP_0(0,0)=|\mu_t|^{-2}.
\end{equation}
\end{prop}
\begin{proof}
First note that for any $Z,Z'\in\R^{2n}$, we have 
\begin{equation}
\begin{split}
\dt\PP_t(Z,Z')&=-\frac{\pi}{2}\left\langle\dt (-J_0J_t)(Z-Z'),(Z-Z')\right\rangle\PP_t(Z,Z')\\
&=-\frac{\pi}{2}\left\langle\left(-J_t\dt J_t\right)(Z-Z'),(Z-Z')\right\rangle_{\mathlarger{t}}\PP_t(Z,Z').
\end{split}
\end{equation}
Differentiating the identity $J_t^2=-\Id$ with respect to $t\in\R$,
we get the formulas
$P^{(0,1)}_t\left(\dt J_t\right)=\left(\dt J_t\right) P^{(1,0)}_t$
and $\Tr[-J_t\dt J_t]=0$.
On the other hand, in \cref{P10t}-\cref{Pi0t'} and in \cref{PtP0}, we are
free to replace $0$ by $u$ for any $u\in\R$.
Setting $t=u$, we get $A_t^t=\Id_{\R^{2n}},\,
\overline{\Pi_t^t}=P^{(0,1)}_t,\,\PP_t\PP_u=\PP_t$, so that
using \cref{PtP0Bfla}, we get
\begin{equation}\label{ptloccomput1}
\begin{split}
&\int_{\R^{2n}}\PP_t(Z,\til{Z})\dt\PP_t(\til{Z},Z')d\til{Z}\\
&=-\frac{\pi}{2}\left\langle\left(-J_t\dt J_t\right)P^{(1,0)}_t(Z-Z'),P^{(1,0)}_t(Z-Z')\right\rangle_{\mathlarger{t}}\PP_t(Z,Z')\\
&-\frac{1}{4}\Tr\left[-J_t\dt J_t\right]\PP_t(Z,Z')\\
&=-\frac{\pi}{2}\left\langle\left(\dt (-J_0J_t)\right)P^{(1,0)}_t(Z-Z'),(Z-Z')\right\rangle\PP_t(Z,Z').
\end{split}
\end{equation}
This can also be deduced by the analogous computations in
\cite[\S\,2]{MM08b}.
We then deduce \cref{mut} from \cref{PtP0Bfla} and \cref{ptloccomput1},
using the fact from \cref{Pi0t'} that
$P_t^{(1,0)}\overline{\Pi_t^0}=0$ and the fact from \eqref{At0} that
$A_t^0J_0=J_tA_t^0$, so that \eqref{Pi0t} implies
$\Pi_t^0=P_t^{(0,1)}A_t^0$.
Then \cref{tilmut} is a straightforward consequence of
\cref{tilmutdef} and \cref{mut} using $\PP_t\PP_t=\PP_t$.
Finally, recall from \cref{Pi0t} that $A_0^t=\Pi_0^t+\overline{\Pi_0^t}$,
and note that
\begin{equation}\label{dtA0t}
\dt \det(A_t^0)^{\frac{1}{2}}=
-\frac{1}{4}\Tr\left[A_t^0\dt(-J_0J_t)\right]\det(A_t^0)^{\frac{1}{2}}.
\end{equation}
Then as $A_0^0=\Id$, formula \cref{barmut} follows from
\cref{PtP0fla} and \cref{tilmutdef} by integrating \cref{dtA0t}.
\end{proof}

We end this section with a study of the holomorphic properties of our local model. Let $z=(z_1,\dots, z_n)\in\C^n$ denote the complex coordinates of $\C^n\simeq\R^{2n}$, defined by $z_j=Z_{2j}+\sqrt{-1}Z_{2j+1}$ for all $1\leq j\leq n$. Then by the results of \cite[\S\,2]{MM08b}, or by explicit
computations from \cref{PPttriv}, we have the following.

\begin{lem}\label{projhol}
For any $F(Z)\in\C[Z]$, there is
$\til\Q(F)(z,\bar{z}')\in\C[z,\bar{z}']$ in the complex coordinates above
such that
\begin{equation}
\int_{\R^{2n}}\PP_0(Z,\til{Z})F(\til{Z})\PP_0(\til{Z},Z')d\til{Z}=\til\Q(F)(z,\bar{z}')\PP_0(Z,Z').
\end{equation}
\end{lem}
An important remark at this point is that the statement of \cref{projhol}
does not depend on the choice of real coordinates $Z,Z'\in\R^{2n}$, so that
\cref{projhol} holds replacing $J_0$ by $J_t$ in all that precedes,
for any $t\in\R$. In fact, the complex coordinates of $(\R^{2n},J_0)$
described above correspond to the basis of $V^{(1,0)}_0\subset\C^{2n}$
given by
\begin{equation}\label{dzdef}
\D{z_j}=P^{(1,0)}_0\D{Z_{2j}},\quad\text{for all}\quad 1\leq j\leq n.
\end{equation}
For any $t\in\R$, recall \cref{splitt0} and set
\begin{equation}\label{dztdef}
\D{z_{t,j}}=\Pi_t^0\D{Z_{2j}},\quad\text{for all}\quad 1\leq j\leq n.
\end{equation}
Let $z_t=(z_{t,1},\dots,z_{t,n})$ be the associated
complex coordinates of $(\R^{2n},J_t)$.
Then using an appropriate change of basis,
we have the following straightforward generalization of \cref{projhol}.
\begin{lem}\label{projholt}
For any $F(Z)\in\C[Z]$ and for any $t\in\R$,
there is
$\til\Q_t(F)(z_t,\bar{z}'_t)\in\C[z_t,\bar{z}'_t]$ in the complex coordinates
of $(\R^{2n},J_t)$ defined above such that
\begin{equation}
\int_{\R^{2n}}\PP_t(Z,\til{Z})F(\til{Z})\PP_t(\til{Z},Z')d\til{Z}=\til\Q_t(F)(z_t,\bar{z}'_t)\PP_0(Z,Z').
\end{equation}
\end{lem}
%
By the description of $\Pi_t^0$ as the projection operator on $V^{(1,0)}_t$
with kernel $V^{(0,1)}_0$, we see that \cref{dztdef} defines a basis
of $V^{(1,0)}_t$, so that
\begin{equation}\label{dztdzbarbasis}
\left\{\D{z_{t,j}}\,,\D{\bar{z}_j}\right\}_{j=1}^n
\end{equation}
defines a basis of $\C^n=V^{(1,0)}_t\oplus V^{(0,1)}_0$ as in \cref{splitt0}.

For any $t\in\R$ and $F\in\C[Z,Z']$, write $F\PP_t$
the operator on $L^2(\R^{2n})$ whose Schwartz kernel with respect to $dZ$ is
given by $F\PP_t(Z,Z')=F(Z,Z')\PP_t(Z,Z')$.
In the same way, we write $F\PP_t\PP_0$ for the operator with Schwartz kernel $F(Z,Z')\PP_t\PP_0(Z,Z')$. We can now state the following fundamental property of our local model.

\begin{prop}\label{PtFPtP0}
For any $F(Z,Z')\in\C[Z,Z']$, there exists 
$Q_t(F)(z_t,\bar{z}')\in\C[z_t,\bar{z}']$
of the same parity as $F$ for any $t\in\R$,
such that
\begin{equation}\label{PtFPtP0fla}
\PP_t(F\PP_t\PP_0)\PP_0=Q_t(F)\PP_t\PP_0.
\end{equation}
\end{prop}
\begin{proof}
The existence of a polynomial $Q_t(F)$ in $Z,Z'\in\R^{2n}$ of the same
parity as $F$ satisfying \cref{PtFPtP0fla} follows from two applications
of \cref{PtP0Ffla} and the fact that $\PP_t\PP_t=\PP_t$.
This together with the definition \cref{PtFPtP0fla} of $Q_t(F)$ gives
\begin{equation}\label{PtQFPtP0}
\PP_t(Q_t(F)\PP_t\PP_0)=\PP_t(F\PP_t\PP_0)\PP_0=Q_t(F)\PP_t\PP_0.
\end{equation}
Thus considering $Q_t(F)(Z,Z')$ as a polynomial in
$Z\in\R^{2n}$, we get from \cref{projholt}
that $Q_t(F)$ depends only on $z_t$ and $Z'$. Using
\begin{equation}\label{QFPtP0P0}
(Q_t(F)\PP_t\PP_0)\PP_0=\PP_t(F\PP_t\PP_0)\PP_0=Q_t(F)\PP_t\PP_0,
\end{equation}
we deduce
in the same way that $Q_t(F)$ depends only on $z_t$ and $\bar{z}'$,
from which we deduce \cref{PtFPtP0}.
\end{proof}

\subsection{Criterion for Toeplitz operators}\label{crit}

Consider the setting and notations of \cref{trivfib}, and let $p_0\in\N$
as in \cref{specdeltapphi} for $U=B$ be fixed.
Recall that we identified $E_p$ over $[0,1]\times X$
with $E_{p,t}=E_t\otimes L^p$ over $X$ for any $t\in[0,1]$, where
$(L,h^L,\nabla^L)$ does not depend on $t$. Then for any $t\in[0,1]$ and
$g_t\in\cinf(X,E_t\otimes E_0^*)$, we can define the
\emph{Berezin-Toeplitz quantization} of $g_t$ as a family
indexed by $p\in\N^*$ of
linear maps from $\HH_{p,0}$ to $\HH_{p,t}$ by the formula
\begin{equation}\label{GBTdef}
P_{p,t}g_tP_{p,0}:\cinf(X,E_{p,0})\fl\cinf(X,E_{p,t}).
\end{equation}
As $\dim\HH_{p,t}<\infty$ for all $p\in\N^*$ and $t\in\R$,
this operator admits a smooth Schwartz kernel in
$\cinf(X\times X,E_{p,t}^*\boxtimes E_{p,0})$ with respect to $dv_X$.
Recalling \cref{notcong}, we then have the following result.
\begin{lem}\label{GBTasy}
There is a family $\left\{\Q_{r,t,x_0}(g)(Z,Z')\in E_{t,x_0}
\otimes E_{0,x_0}^*\right\}_{r\in\N}$ of polynomials in $Z,Z'\in\R^{2n}$
of the same parity as $r$, smooth in $x_0\in X$ and $t\in[0,1]$,
such that the following asymptotic expansion holds,
\begin{equation}\label{GBTasyfla}
p^{-n}(P_{p,t}g_tP_{p,0})_{x_0}(Z,Z')\cong\sum\limits_{r=0}^{\infty}\Q_{r,t,x_0}(g)\PP_{t,x_0}\PP_0(\sqrt{p}Z,\sqrt{p}Z')p^{-\frac{r}{2}} + \cO(p^{-\infty}).
\end{equation}
Furthermore, for all $x_0\in X,\,t\in[0,1],\,Z,Z'\in\R^{2n}$, we have
\begin{equation}\label{GBTcoeff}
\Q_{0,t,x_0}(g)(Z,Z')= g_t(x_0).
\end{equation}
\end{lem}
\begin{proof}
Using the results of \cref{locmod}, and in particular \cref{PtP0},
the proof of \cref{GBTasyfla} and \cref{GBTcoeff} is a straightforward
adaptation of the proof of \cref{Toepasy} in \cite[Lem.3.3]{ILMM17}, using
the asymptotic expansion \cref{asyfla} of the Bergman kernel.
\end{proof}

For any $t\in\R$ and $p\in\N^*$, let $\|\cdot\|_{p,0,t}$ be the
operator
norm induced by $\|\cdot\|_{p,0}$ and $\|\cdot\|_{p,t}$ on the space
$L(\HH_{p,0},\HH_{p,t})$ of bounded operators from $\HH_{p,0}$ to
$\HH_{p,t}$.
The following fundamental result is a converse to \cref{GBTasy},
and gives a criterion for a sequence of operators in
$L(\HH_{p,0},\HH_{p,t})$ for all
$p\in\N^*$ to behave like a \emph{Toeplitz operator},
that is to admit an asymptotic expansion as $p\fl+\infty$
in terms of Berezin-Toeplitz operators \cref{GBTdef}.

\begin{theorem}\label{criterion}
Let $\{T_{p,t}\in L(\HH_{p,0},\HH_{p,t})\}_{p\in\N^*}$ be a family
of bounded operators from $\HH_{p,0}$ to $\HH_{p,t}$,
smooth in $t\in [0,1]$,
and assume that for any $p\in\N^*$ and $t\in\R$,
the induced operator
$T_{p,t}=P_{p,t}T_{p,t}P_{p,0}:\cinf(X,E_{p,0})\fl\cinf(X,E_{p,t})$
satisfies
\begin{align}\label{criterionexp}
p^{-n}T_{p,t}(Z, Z')\cong
\sum^{\infty}_{r=0}Q_{r,t,x_0}
\PP_{t,x_0}\PP_0(\sqrt{p}Z, \sqrt{p}Z')p^{-\frac{r}{2}}
+\cO(p^{-\infty}),
\end{align}
for a family
$\{\Q_{r,t,x_0}(Z,Z')\in E_{t,x_0}\otimes E_{0,x_0}^*\}_{r\in\N}$ of
polynomials in $Z,Z'\in\R^{2n}$ of the same parity as $r$, smooth in
$x_0\in X$ and $t\in[0,1]$.

Then 
%
there exist a family
$\{g_{l,t}\in\cinf(X,E_t\otimes E_0^*)\}_{l\in\N}$,
smooth in $t\in [0,1]$, such that for all $k\geq 0$, there exists $C_k>0$ such that
\begin{equation}\label{genToepdef}
\Big\|T_{p,t}-\sum_{l=0}^{k-1} p^{-l}P_{p,t}g_{l,t}P_{p,0}\Big\|_{p,0,t}\leq C_k p^{-k},
\end{equation}
for all $p\in\N^*$ and $t\in[0,1]$.
%
\end{theorem}

The proof of \cref{criterion} is parallel to the proof of the analogous
results in \cite[\S\,4.2]{MM08b}, \cite[\S\,4]{ILMM17}, and will
occupy the rest of this section. The main additional difficulty
is that we are working with two sets $z,\,z_t$ of complex coordinates
of $\R^{2n}$ for two different complex structures $J_0$ and $J_t$
as in \cref{locmod}. As it will appear in the proof of \cref{t4.7}, this
is solved using the fact that the spaces $V_0^{(1,0)}$ and $V_t^{(0,1)}$
in \cref{splitt0} are transverse in $\C^{2n}$.
Another difference is that we can't assume $T_{p,t}$
to be self-adjoint in this context.

Following \cite[\S\,4.2]{MM08b}, we will construct inductively the sequence
$\{g_{l,t}\in\cinf(X,E_t\otimes E_0^*)\}_{l\in\N}$ for any $t\in[0,1]$
such that \cref{genToepdef} holds.
Let us start with the case $k=0$ in \cref{genToepdef}.
For any $t\in[0,1]$ and $x_{0}\in X$, we set
\begin{align}\label{4.6}
g_{0,t}(x_{0})=\Q_{0,t,x_{0}}(0, 0)\in E_{t,x_0}\otimes E_{0,x_0}^*.
\end{align}
Then $g_{0,t}(x_0)$ is smooth in $t\in[0,1]$. We will show that
\begin{align}\label{4.8}
T_{p,t}=P_{p,t}g_{0,t}P_{p,0}+O(p^{-1}).
\end{align}
The proof of \cref{4.8} is the result of \cref{t4.3} and \cref{t4.9}.
In the proof of these propositions, we fix $t\in[0,1]$.

\begin{prop}\label{t4.3}
In the conditions of \textup{\cref{criterion}}, we have
\begin{align}
\Q_{0,t,x_{0}}(Z, Z')=\Q_{0,t,x_{0}}(0, 0)\in E_{t,x_{0}}\otimes E_{0,x_0}^*
\end{align}
for all $t\in[0,1],\,x_{0}\in X$ and $Z, Z'\in\R^{2n}$.
\end{prop}

\begin{proof}
The proof is divided in the series of Lemmas \cref{t4.4} -- \cref{t4.8}.
Recall from \cref{locmod} that $z'$ denotes the holomorphic coordinate in
$Z'\in\R^{2n}$ associated to $J_{0,x_0}$ as in \cref{dzdef} and that $z_t$
denotes the holomorphic coordinate in $Z\in\R^{2n}$ associated to 
$J_{t,x_0}$ as in \cref{dztdef}.
Our first observation is as follows. 

\begin{lem}\label{t4.4}
$\Q_{0,t,x_{0}}$ only depends on $z_t,\bar{z}'$,
so that there is $\Q_{t,x_0}(z_t,\bar{z}')\in E_{t,x_0}\otimes E_{0,x_0}^*$,
polynomial in $z_t,\,\bar{z}'$, such that
for all $Z,Z'\in\R^{2n}$,
\begin{equation}\label{V2.3}
\Q_{0,t,x_0}(z,\bar{z}')=\Q_{0,t,x_0}(Z,Z').
\end{equation}
\end{lem}
\begin{proof}
By \cref{criterionexp}, we know that
\begin{align}\label{4.11}
p^{-n}T_{p,t,x_{0}}(Z, Z')
\cong
\Q_{r,t,x_{0}}\PP_{t,x_0}\cP_{x_{0}}
(\sqrt{p}Z, \sqrt{p}Z')+\cO(p^{-\frac{1}{2}}).
\end{align}
By \cref{asy} in the context of \cref{trivfib}, we thus get in the
notations of \cref{locmod},
\begin{align}\begin{split}\label{4.12}
p^{-n}(P_{p,t}T_{p,t}P_{p,0})_{x_{0}}(Z, Z')
\cong
\cP_{t,x_{0}}(\Q_{0,t, x_{0}}\cP_{t,x_{0}}\PP_{0})
\cP_{0}
(\sqrt{p}Z, \sqrt{p}Z')+\cO(p^{-\frac{1}{2}}).
\end{split} \end{align}
%
%
On the other hand,
from \cref{4.11}, \cref{4.12} and the formula $P_{p,t}T_{p,t}=T_{p,t}$,
we deduce
\begin{equation}\label{4.13}
\Q_{0,t, x_{0}}\cP_{t,x_{0}}\PP_{0,x_0}
\\ =
\cP_{t,x_{0}}(\Q_{0,t, x_{0}}\cP_{t,x_{0}}\PP_{0})
\cP_{0}
\end{equation}
From \cref{PtFPtP0}, this implies that $\Q_{0,t,x_0}$ only depends on
$z_t,\bar{z}'$.
\end{proof}

Recall from \cref{trivfib} that $\tau^E_t\in E_t\otimes E_0^*$
is induced by parallel
transport with respect to $\nabla^E$ along horizontal lines of
$[0,1]\times X$ going from $0$ to $t\in[0,1]$.
For all $Z,Z'\in\R^{2n}$, let us write
\begin{equation}\label{Qxdef}
\Q_{x_0}(z_t,\bar{z}')
=\tau^{E,-1}_{t,x_0}\Q_{0,t,x_0}(z_t,\bar{z}')\in\End(E_{0,x_0})\,,
\end{equation}
so that \ref{t4.4} implies $\Q_{x_0}(Z,Z')=\Q_{x_0}(z_t,\bar{z}')$.
For any $x_0\in X$, let $\Q_{x_0}=\sum_{i\geqslant 0} \Q^{(i)}_{x_0}$
be the decomposition
of $\Q_{x_0}$ in homogeneous polynomials
$\Q^{(i)}_{x_0}$ of degree $i$. We will show that $\Q^{(i)}_{x_0}$
vanishes identically for $i>0$,
that is
\begin{align}\label{4.17}
\Q^{(i)}_{x_0}(z_t,\bar{z}')=0 \ \ \textup{for all}\ x_0\in X,\, i>0.
\end{align}
The first step is to prove
\begin{align}\label{4.18}
\Q^{(i)}_{x_0}(z_t,0)=0 \ \ \textup{for all}\ x_0\in X,\, i>0.
\end{align}
Recall the smooth family of vertical charts $\psi$ in \cref{trivfib}.
For $x\in X,\,Z'\in \R^{2n}\simeq T_xX$
and $y=\psi_{x}(Z')$, set
\begin{align}\label{4.20}\begin{split}
F^{(i)}(x, y)=&\Q^{(i)}_{x}(0, \bar{z}')\in \End(E_{0,x}), \\
\til{F}^{(i)}(x, y)=&\left(F^{(i)}(y, x)\right)^{\ast}\in\End(E_{0,y}).
\end{split}\end{align}
Then $F^{i}$ and $\til{F}^{(i)}$ define smooth sections on a neighborhood
of the diagonal of $X\times X$. Clearly, the $\til{F}^{(i)}(x, y)$'s 
need not be
polynomials in $z_t$ and $\bar{z}'$.

Denote by $d(\cdot,\cdot)$ the Riemannian distance on $(X,g^{TX}_0)$.
Since we wish to define global operators induced by these kernels, we use
a cut-off function in the neighborhood of the diagonal. Pick a
smooth function $\eta\in \cC^{\infty}(\R)$ such that $\eta(u)=1$ for $|u|\leqslant \epsilon_0/2$ and $\eta(u)=0$ for $|u|\geqslant \epsilon_0$.
We denote by $P_{p,t}P_{p,0}F^{(i)}$ and $\til{F}^{(i)}P_{p,0}P_{p,t}$
the operators defined by the kernels
\begin{align}\label{4.22}
\eta(d(x, y))(P_{p,t}\tau^E_tP_{p,0})(x, y)F^{(i)}(x, y)\ \ \textup{and}\ \
\eta(d(x, y))\til{F}^{(i)}(x, y)(P_{p,0}\tau^{E,-1}_tP_{p,t})(x, y)
\end{align}
with respect to $dv_{X}(y)$. Set
\begin{align}\label{4.23}
\cT_{p}=T_{p,t}-\sum_{i\geqslant 0}(P_{p,t}P_{p,0}F^{(i)})p^{\frac{i}{2}}.
\end{align}
From \cref{GBTasyfla}, \cref{criterionexp} and \cref{4.23},
we deduce that in the sense
of \cref{defcong} for $Z=0$ and any $x_0\in X$, we have
\begin{align}\label{4.25}
p^{-n}\cT_{p,x_{0}}(0, Z')\cong
\sum_{r=1}^{\infty}(R_{r}\PP_{t,x_0}\cP_{0})(0, \sqrt{p}Z')p^{-\frac{r}{2}}
+\cO(p^{-\infty}),
\end{align}
for some polynomials $R_{r, x_{0}}$ of the same parity as $r$.
%

\begin{lem}\label{t4.5}
There exists $C>0$ such that for any
$p\in\N^*$ and $s\in \cinf(X, E_{p,0})$, we have
\begin{align}\label{4.27}
\begin{split}
\|\cT_{p}s\|_{p,t}\leqslant Cp^{-\frac{1}{2}}\|s\|_{p,0}, \\
\|\cT^{\ast}_{p}s\|_{p,t}\leqslant Cp^{-\frac{1}{2}}
\|s\|_{p,0}.
\end{split}\end{align}
\end{lem}

\begin{proof}
This is a consequence of the vanishing of the term of order $0$ in
\cref{4.25}. The proof is the same than the proof of the analogous result
in \cite[Lem.4.13]{MM08b}.
\end{proof}

For any $x_0\in X$ and  $Z,Z'\in\R^{2n}\simeq T_{x_0}X$ such that $|Z|,|Z'|<\epsilon_0$, recall that $\til{F}^{(i)}_{x_0}(Z,Z')\in\End(E_{0,x_0})$
denotes the image of
$\til{F}^{(i)}(x,y)\in\End(E_{0,x_0})$, with $x=\psi_{x_0}(Z)$, $y=\psi_{x_0}(Z')$
in the trivialization around $x_0\in X$ defined in \cref{trivfib}.
Let us consider the following Taylor expansion, for any $k\in\N$,
\begin{align}\label{4.37}
\til{F}^{(i)}_{x_0}(0,Z')=\sum_{|\alpha|\leqslant k}
\frac{\partial^{|\alpha|}\til{F}^{(i)}_{x_0}}{\partial Z'^{\alpha}}(0, 0)
\frac{(\sqrt{p}Z')^{\alpha}}{\alpha!}p^{-\frac{|\alpha|}{2}}+
O(|Z'|^{k+1}).
\end{align}
The next step of the proof of \cref{t4.3} is the following.

\begin{lem} \label{t4.6}
For any $x_0\in X$, we have
\begin{align}\label{4.38}
\frac{\partial^{|\alpha|}\til{F}^{(i)}_{x_0}}{\partial Z'^{\alpha}}(0, 0)=0\ \
\textup{for}\ i-|\alpha|> 0.
\end{align}
\end{lem}

\begin{proof}
The definition \cref{4.23} of $\cT_{p}$ shows that
\begin{align}\label{89}
\cT^{\ast}_{p}=T_{p,t}^*-\sum_{i\geqslant 0}p^{\frac{i}{2}}
(\til{F}^{(i)}P_{p,0}P_{p,t}).
\end{align}
Pick $x_0\in X$, and let us develop the sum on the right-hand side.
Combining the Taylor expansion \cref{4.37}
with the expansion \cref{GBTasyfla} where the roles of $J_0$ and $J_t$ are
swapped, for any $k\geq\deg \Q_{x_0}+1$ we obtain
\begin{multline}\label{90}
p^{-n}\sum_{i\geq 0}\big(\til{F}^{(i)}P_{p,0}P_{p,t}\big)_{x_{0}}(0, Z')
p^{\frac{i}{2}}\cong\sum_{i\geq 0}\sum_{|\alpha|, r\leqslant k}\Q_{r,t, x_{0}}(\tau^{E,-1})\cP_{0}\cP_{t,x_{0}}
(0, \sqrt{p}Z')\\
\frac{\partial^{|\alpha|}\til{F}^{(i)}_{x_0}}{\partial Z'^{\alpha}}(0, 0)
\frac{(\sqrt{p}Z')^{\alpha}}{\alpha!}p^{\frac{i-|\alpha|-r}{2}}
+\cO\left(p^{\frac{\deg \Q_{x_0}-k-1}{2}}\right).
\end{multline}
Having in mind the second inequality of \cref{4.27},
this is only possible if for every $j>0$, the coefficients of 
$p^{\frac{j}{2}}$ in the
right-hand side of \cref{90} vanish. Thus, we have for any $j>0$,
\begin{align}\label{91}
\sum^{\deg \Q_{x_0}}_{i=j}
\sum_{j+r=i-|\alpha|}\Q_{r,t, x_{0}}(\tau^{E,-1})(0, \sqrt{p}Z')
\frac{\partial^{|\alpha|}\til{F}^{(i)}_{x_0}}{\partial Z'^{\alpha}}(0, 0)
\frac{(\sqrt{p}Z')^{\alpha}}{\alpha!}=0.
\end{align}
Note that \eqref{4.20} implies that $\til{F}^{(i)}=0$ for $i>\deg \Q_{x_0}$,
so that \cref{4.38} automatically holds in that case.
From \cref{91}, we will prove by a descending recurrence
on $j>0$ that \cref{4.38} holds for $i-|\alpha|>j$.
As the first step of the recurrence, let us take
$j=\deg \Q_{x_0}$ in
\cref{91}. Since $\Q_{0,t, x_{0}}(\tau^{E,-1})=\tau^{E,-1}_{t,x_0}$
is invertible, we get
immediately $\til{F}^{(j)}_{x_0}(0, 0)=0$ in that case.
Hence \cref{4.38} holds for $i-|\alpha|\geq \deg \Q_{x_0}$.
Assume that \cref{4.38} holds for
$i-|\alpha|>j_0>0$. Then for $j=j_0$, the coefficient
with $r>0$ in \cref{91} is zero.
By the invertibility of $\Q_{0,t, x_{0}}(\tau^{E,-1})=\tau^{E,-1}_{t,x_0}$ once again, \cref{91} reads
\begin{align}\label{4.42}
\sum_{\alpha\in\N^{2n}}\frac{\partial^{|\alpha|}
\til{F}^{(j_0+|\alpha|)}_{x_0}}{\partial Z'^{\alpha}}(x_{0}, 0)
\frac{(\sqrt{p}Z')^{\alpha}}{\alpha!}=0,
\end{align}
which entails \cref{4.38} for $i-|\alpha|\geq j_0$. The proof of \cref{4.38}
is complete.
\end{proof}

\begin{lem}\label{t4.7}
For any $x_0\in X$, we have
\begin{align}\label{4.43}
\frac{\partial^{|\alpha|}\Q^{(i)}_{x_0}}{\partial z_t^{\alpha}}(0, 0)=0,
\ \text{for all}~|\alpha|\leqslant i\,.
\end{align}
Therefore $\Q^{(i)}_{x_0}(z_t,0)=0$ for all $i>0$ and $Z\in\R^{2n}$,
so that \textup{\cref{4.18}} holds true. Moreover,
\begin{align}\label{4.44}
\Q^{(i)}_{x_0}(0,\bar{z}')=0~~\textup{for all}~x_0\in X,\,i>0~~\textup{and all}~Z\in \R^{2n}.
\end{align}
\end{lem}

\begin{proof}
Let us start with some preliminary observations.
From \cref{t4.6} and \cref{90}, we get for any $x_0\in X$,
\begin{multline}\label{V2.1}
p^{-n}\sum_{i\geq 0}\big(\til{F}^{(i)}P_{p,0}P_{p,t}\big)_{x_{0}}(0, Z')
p^{\frac{i}{2}}\\
\cong\sum_{|\alpha|=i}
\frac{\partial^{|\alpha|}\til{F}^{(i)}_{x_0}}{\partial Z'^{\alpha}}(0, 0)
\frac{(\sqrt{p}Z')^{\alpha}}{\alpha!}\tau^{E,-1}_{t,x_0}\PP_0
\cP_{t,x_0}(0,\sqrt{p}Z')
+\cO\left(p^{-\frac{1}{2}}\right).
\end{multline}
On the other hand, taking the adjoint of \cref{criterionexp}
and using \cref{V2.3}, \cref{Qxdef} we get
\begin{align}\label{V2.2}
p^{-n}T_{p, x_{0}}^*(0, Z')\cong
\left(\Q_{x_{0}}(\sqrt{p}z_t',0)\right)^*\tau^{E,-1}_{t,x_0}
\cP_{0}\cP_{t,x_{0}}(0, \sqrt{p}Z')
+\cO(p^{-\frac{1}{2}}).
\end{align}
In view of \cref{t4.5},
comparing \cref{89} with \cref{V2.1,V2.2} for any $x_0\in X$ gives
\begin{align}\label{95}
\til{F}^{(i)}_{x_0}(0, Z')=\left(\Q^{(i)}_{x_0}(z_t',0)\right)^*+O(|Z'|^{i+1}).
\end{align}
By definition \eqref{4.20} of $\til{F}^{(i)}$, we take
the adjoint of \cref{95} and get
\begin{align}\label{4.46}
F^{(i)}_{x_0}(Z, 0)=\Q^{(i)}_{x_0}(z_t,0)
+O(|Z|^{i+1}).
\end{align}
Thus in order to prove the Lemma it suffices to show that
\begin{align}\label{98}
\frac{\partial^{|\alpha|}}{\partial z^{\alpha}}F^{(i)}_{x_0}(0,0)=0,~~\textup{for all}~x_0\in X~\text{and}~|\alpha|\leqslant i\,.
\end{align}
We prove this by induction over $|\alpha|$.
The case $|\alpha|=0$ immediately follows from the fact
that $\Q^{(i)}_{x_0}(z_t, \bar{z}')$
is a homogeneous
polynomial of degree $i>0$.
For the induction step,
assume that
\cref{98} holds for $|\alpha|=i_0<i$ at all $x_0\in X$.
Together with \cref{4.46},
this implies in particular
that all derivatives in $x$ up to order $i_0$ of $F^{(i)}(x,y)$
vanish for $x=y$. On the other hand, we know by definition
\eqref{4.20} of $F^{(i)}$ that
\begin{equation}
\frac{\partial}{\partial z_{j}'}F^{(i)}_{x_0}(0,Z')=0,
~~\text{for all}~x_0\in X~\text{and}~|Z'|<\epsilon_0\,.
\end{equation}
Let us write $j_{\Delta}: \R^{2n}\rightarrow \R^{2n}\times\R^{2n}$
for the diagonal injection.
Using the fact that $\til{F}^{(i)}(\psi_x(Z),\psi_x(Z'))=\til{F}^{(i)}_{x}(Z,Z')$ for
all $|Z|,|Z'|<\epsilon_0$ and $x\in X$ in the coordinate charts
$\psi_x:T_xX\to X$,
we then deduce from the induction hypothesis that
for any $1\leq j\leq n$ and $x\in X$,
\begin{align}\label{4.49}
\frac{\partial^{|\alpha|+1}
F^{(i)}_{x}}{\partial z^{\alpha}\partial z_j}(0,0)
=\left(\frac{\partial}{\partial z_j}j^{\ast}_{\Delta}
\frac{\partial^{|\alpha|}F^{(i)}_{x}}{\partial z^{\alpha}
}\right)(0)
-\frac{\partial^{|\alpha|}}
{\partial z^{\alpha}}
\frac{\partial F^{(i)}_{x}}
{\partial z_{j}'}(0,0)=0.
\end{align}
This gives \eqref{98}.
Now recall from \cref{dztdzbarbasis} that the set of vectors 
$\{\partial/\partial z_j, \partial/\partial\bar{z}_{j,t}\}_{1\leq j\leq n}$
form a basis of $\C^{2n}$, so that \cref{4.46} and \cref{98}
implies \cref{4.43}, and this is equivalent to \cref{4.18}.
The argument for \cref{4.44} is analogous,
setting $\til{F}^{(i)}(x,y):=(\Q^{(i)}(z_t,0))^*$ and
$F^{(i)}(x,y):=(\til{F}^{(i)}(y,x))^*$ in \eqref{4.20} and swapping their
roles in the rest of the argument.
This finishes the proof of \cref{t4.7}.
\end{proof}

\begin{lem}\label{t4.8}
We have $\Q^{(i)}_{x_0}(z_t, \bar{z}')=0$ for all $x_0\in X$ and $i>0$.
\end{lem}

\begin{proof}
For any $x_0\in X$, let us consider the operator
\begin{align}\label{4.52}
\frac{1}{\sqrt{p}}P_{p,t}\big(\nabla_v^{E_p}
T_{p,t}\big)  P_{p,0}
\ \ \textup{with}\ v\in \cC^{\infty}(X, TX_\C),
~v_{x_{0}}=\frac{\partial}{\partial z_{t,j}}.
\end{align}
By \cref{defcong} and \cref{Qxdef}, the operator \cref{4.52} admits an expansion as in \cref{criterionexp} except for the condition on the parity
of the polynomials,
with leading term at $x_0\in X$ equal to
\begin{align}
\tau^{E,-1}_{t,x_0}\Big(\frac{\partial \Q_{x_{0}}}{\partial z_{t,j}}\Big)(\sqrt{p}z_t, \sqrt{p}\bar{z}')
\cP_{x_{0}}(\sqrt{p}Z, \sqrt{p}Z').
\end{align}
On the other hand, note that the proofs of \cref{t4.7,t4.5,t4.6,t4.7}
did not use the parity condition on the polynomials in \cref{criterion},
so that \cref{t4.7} holds
for the operator \cref{4.52}. Following the notations above,
we thus get for $i>0$,
\begin{align}\label{4.54}
\frac{\partial\Q^{(i+1)}_{x_0}}{\partial z_{t,j}}(0, \bar{z}')=\Big(\frac{\partial \Q_{x_{0}}}{\partial z_{t,j}}\Big)^{(i)}(0, \bar{z}')=0.
\end{align}
Now \cref{4.18} tells us that the constant term of 
$(\frac{\partial}{\partial z_{j,t}}\Q_{x_{0}})(z_t, \bar{z}')$ vanishes,
so that \cref{4.54} holds as well for $i=0$.
%
Then repeating this reasoning with the adjoint of \cref{4.52},
we further get for any $i>0$ and $1\leq j\leq n$,
\begin{align}\label{4.55}
\frac{\partial\Q^{(i)}_{x_0}}{\partial \bar{z}'_{j}}(z_t, 0)=0.
\end{align}
By continuing this process, we show by induction that for all
$x_0\in X,\,i>0$ and
$\alpha\in \N^{n}$,
\begin{align}\label{4.56}
\frac{\partial^{|\alpha|}\Q^{(i)}_{x_0}}{\partial z^{\alpha}_t}(0, \bar{z}')
=\frac{\partial^{|\alpha|}\Q^{(i)}_{x_0}}{\partial \bar{z}'^{\alpha}}(z_t, 0)=0.
\end{align}
This proves \cref{t4.8} and \cref{4.17} holds true.
\end{proof}

\noindent
\cref{t4.8} finishes the proof of \cref{t4.3}.
\end{proof}

\begin{prop}\label{t4.9}
We have
\begin{align}\label{4.58}
T_{p,t}=P_{p,t}g_{0,t}P_{p,0}+O(p^{-1}),
\end{align}
i.e., relation \textup{\cref{4.8}} holds true in the
sense of \cref{genToepdef}.
\end{prop}

\begin{proof}
Comparing the asymptotic expansions \cref{criterionexp} and
\cref{GBTasyfla} of both sides of \cref{4.58} up to
order $2$, as in the proof of the analogous result in 
\cite[Prop.4.17]{MM08b}, it suffices to prove that
for any $x_0\in X$ and $Z,Z'\in\R^{2n}$,
\begin{align}\label{4.63}
(\Q_{1,t,x_0}-\Q_{1,t,x_0}(g_{0}))(Z,Z')=0.
\end{align}
The left hand side of \cref{4.63} is the polynomial associated
with the first 
coefficient of the expansion as in \cref{criterionexp} of
\begin{equation}\label{V2.6}
\sqrt{p}(T_{p,t}-P_{p,t}g_{0,t}P_{p}).
\end{equation}
As before, we see that this operator satisfies
the hypotheses of \cref{t4.7,t4.5,t4.6,t4.7,t4.8}, so that
all the homogeneous components of degree $i>0$ of
$\Q_{1,t, x_0}-\Q_{1,t, x_0}(g_{0})$ vanish.
Furthermore, $\Q_{1,t, x_0}-Q_{1,t, x_0}(g_{0})$ is an odd polynomial,
so that in particular
its constant term vanishes as well. This shows \cref{4.63}.
\end{proof}

By \cref{4.58}, the operator $p(T_{p,t}-P_{p,t}g_{0,t}P_{p,0})$ satisfies the hypotheses of \cref{criterion},
so that \cref{t4.3} and \cref{t4.9} applied to
$p(T_{p,t}-P_{p,t}g_{0,t}P_{p,0})$ yield
$g_{1,t}\in \cC^{\infty}(X, E_t\otimes E_0^*)$ such that
\begin{equation}
T_{p,t}=P_{p,t}g_{0,t}P_{p,0}+P_{p,t}g_{1,t}P_{p,0}+O(p^{-2})
\end{equation}
We can then continue this process to get
\cref{genToepdef} for
any $q\in\N$ by induction.
This completes the proof of \cref{criterion}.

\subsection{Parallel transport as a Toeplitz operator}\label{pt}

In this section, we show that the parallel transport
in the quantum bundle of \cref{defsec} satisfies the hypotheses
of \cref{criterion}, and we compute the first coefficient of
its asymptotic expansion \cref{genToepdef}
in terms of the local data of \cref{locmod}.

We first work in the setting of \cref{famberg}, so that $\pi: M\fl B$
is a prequantized fibration equipped with a relative compatible complex
structure $J\in\End(TX)$ and an auxiliary vector bundle
$(E,h^E,\nabla^E)$ over $M$.
We assume $B$ compact and fix $p_0\in\N$
in \cref{specdeltapphi} with $U=B$.
Recall the $L^2$-connection
$\nabla^{\HH_p}$ on $\HH_p$ of \cref{connectionL2},
for any $p\in\N^*$.

\begin{defi}\label{connectiontoep}
A family of connections
$\{\nabla^p\}_{p\in\N^*}$ on the quantum bundle
$\{\HH_p\}_{p\in\N^*}$ over $B$ is called a
\emph{Toeplitz connection} if it is of the form
\begin{equation}\label{connToepdef}
\nabla^p=\nabla^{\HH_p}+K_p\,
\end{equation}
for any $p\in\N$, with
$\{K_p\in\cinf(B,\End(\HH_p)\otimes T^*B)\}_{p\in\N^*}$
such that there exists a family
$\{\sigma_l\in\cinf(M,\End(E)\otimes T^*M)\}_{l\in\N}$
smooth in $t\in [0,1]$, such that for all $k\geq 0$ and any
$v\in\cinf(B,TB)$, there exists $C_k>0$ such that
for all $b\in B$,
\begin{equation}\label{connToepexp}
\Big\|K_{p}(v)-\sum_{l=0}^{k-1} p^{-l}P_{p}\,\sigma_{l}(v^H)P_{p}
\Big\|_{p,b}
\leq C_k p^{-k}.
\end{equation}
\end{defi}

From now on, we fix a Toeplitz connection
$\{\nabla^p\}_{p\in\N^*}$
and a path $\gamma:[0,1]\rightarrow B$. Let
$p\in\N^*$, and
recall that $\partial_t$ denotes the canonical vector field on $[0,1]$.
The \emph{parallel transport} along $\gamma$ with respect to
$\{\nabla^p\}_{p\in\N^*}$ is the family of endomorphisms
\begin{equation}
\Tau_{p,t}:\HH_{p,\gamma(0)}\rightarrow\HH_{p,\gamma(t)}
\end{equation}
satisfying the following differential equation in $t\in[0,1]$
for any $s_0\in\HH_{p,\gamma(0)}$ and $p\in\N^*$,
\begin{equation}\label{ptfla}
\left\{
\begin{array}{l}
  \nabla^p_{\delt}\Tau_{p,t}s_0=0, \\
  \\
  \Tau_{p,0}s_0 = s_0.
\end{array}
\right.
\end{equation}
Pulling back the fibration by $\gamma:[0,1]\fl B$,
we can assume that $B=[0,1]$ and work in the
setting of \cref{trivfib},
%
where $(L,h^L,\nabla^L)$ is identified over
$[0,1]\times X$
with the pullback of a fixed Hermitian bundle
with connection over $X$. Then by \cref{trivnabL}, for any
$t\in[0,1]$,
\cref{connectionL2} becomes
\begin{equation}\label{nabHptriv}
\nabla^{\HH_p}_{\delt}=P_{p,t}\nabla^E_{\delt^H} P_{p,t}.
\end{equation}
Let $\sigma_0\in\cinf(M,\End(E)\otimes T^*M)$ be the first coefficient
of the Toeplitz connection $\{\nabla^p\}_{p\in\N^*}$
in \cref{connToepexp}, and define a connection
$\nabla^{E,\sigma}$ on $E$ over $M$ by
\begin{equation}\label{nutil}
\nabla^{E,\sigma}=\nabla^E+\sigma_0.
\end{equation}
For any $t\in[0,1]$, let $\tau^{E,\sigma}_t$ be the
parallel transport with respect to $\nabla^{E,\sigma}$ along
horizontal curves of $[0,1]\times X$.
For any $t\in[0,1]$, let $\mu_t\in\cinf(X,\C)$ be the function with
value at $x_0\in X$ given by \cref{tilmutdef} in the trivialization of \cref{trivfib} around $x_0\in X$. 
The following theorem is the central
result of this paper.
%

\begin{theorem}\label{Approx}
There exists a family
$\{\mu_{l,t}\in\cinf(X,E_t\otimes E_0^*)\}_{l\in\N}$,
smooth in $t\in [0,1]$, such that for all $k\geq 0$, there exists $C_k>0$ such that
\begin{equation}\label{Approxfla}
\Big\|\Tau_{p,t}-
\sum_{l=0}^{k-1} p^{-l}P_{p,t}\mu_{l,t}P_{p,0}\Big\|_{p,0,t}
\leq C_k p^{-k},
\end{equation}
for all $p\in\N^*$ and $t\in[0,1]$.
Furthermore, its first coefficient
$\mu_{0,t}\in\CC^\infty(X,E_{p,t}\otimes E_{p,0}^*)$
is given by
\begin{equation}\label{Taucoeff}
\mu_{0,t}=\mu_t\tau^{E,\sigma}_t.
\end{equation}
\end{theorem}
\begin{proof}
For any $g_t\in\cinf(X,E_t\otimes E_0^*)$, smooth in $t\in[0,1]$, and
for all $p\in\N^*$, let us consider the operator
\begin{equation}\label{nabPfP}
\nabla^p_{\delt}P_{p,t}g_tP_{p,0}:\HH_{p,0}\fl\HH_{p,t}.
\end{equation}
Then by \cref{mut}, \cref{GBTasyfla}, \cref{connToepdef} and
\cref{nabHptriv}, the Schwartz
kernel of \cref{nabPfP} satisfies the assumptions of \cref{criterion}.
Using \cref{nutil}, the first coefficient $Q_{0,t,x_0}$
of its expansion \cref{criterionexp}
for any $x_0\in X,\,t\in[0,1]$, is the constant polynomial equal to
\begin{equation}\label{eqft}
\left(\nabla^{E,\sigma}_{\delt^H} g_t\right)(x_0)-\frac{1}{4}\Tr
\left[\Pi_{t,x_0}^0\dt(-J_{0,x_0}J_{t,x_0})
\right]g_t(x_0),
\end{equation}
where $\Pi_{t,x_0}^0\in\End(T_{x_0}X_\C)$ is the projection to
$T^{(1,0)}_{x_0}X_t$ with kernel $T^{(0,1)}_{x_0}X_0$.
Let $\mu_{0,t}\in\cinf(X,E_t\otimes E_0^*)$ be the section satisfying the following ordinary differential
equation in $t\in[0,1]$,
\begin{equation}\label{eqft=0}
\left\{
\begin{array}{l}
  \nabla^{E,\sigma}_{\delt^H} \mu_{0,t}-\frac{1}{4}\Tr
\left[\Pi_{t}^0\dt(-J_{0}J_{t})
\right]\mu_{0,t}=0, \\
  \\
  \mu_{0,0}=\Id_{E_0}.
\end{array}
\right.
\end{equation}
Then he have $P_{p,0}\mu_{0,0}P_{p,0} = P_{p,0}$
for all $p\in\N^*$, and the estimate
$\nabla^p_{\delt}P_{p,t}\mu_{0,t}P_{p,0}=O(p^{-1})$
holds in operator norm as $p\fl +\infty$
by \cref{eqft}.
For any $k\in\N^*$, let us assume by induction that we have sections
$\mu_{l,t}\in\cinf(X,E_t\otimes E_0^*)$ with
$\mu_{l,0}\equiv 0$ for all $0<l\leq k-1$, satisfying
\begin{equation}\label{nabPgP=Op}
\nabla^p_{\delt}\sum_{l=0}^{k-1} p^{-l}
P_{p,t}\mu_{l,t}P_{p,0}=O(p^{-k}),
\end{equation}
in operator norm as $p\fl +\infty$.
Then \cref{criterion} applies to the left hand side of
\cref{nabPgP=Op} multiplied by $p^k$.
Let $g_k\in\cinf(X,E_t\otimes E_0^*)$ denote its first coefficient
in the expansion \cref{genToepdef}, and let 
$\mu_{k,t}\in\cinf(M,E_t\otimes E_0^*)$ be the section
satisfying the following ordinary differential equation in
$t\in[0,1]$,
\begin{equation}\label{edogk}
\left\{
\begin{array}{l}
  \nabla^{E,\sigma}_{\delt}
  \mu_{k,t}-\frac{1}{4}\Tr
\left[\Pi_{t}^0\dt(-J_{0}J_{t})
\right]\mu_{k,t}+g_k=0 \\
  \\
  \mu_{k,0}=0.
\end{array}
\right.
\end{equation}
Then using \cref{criterion} as above,
we have
$\nabla^p_{\delt}\sum_{l=1}^k p^{-l} P_{p,t}\mu_{l,t}
P_{p,0}=O(p^{-k-1})$
in operator norm as $p\fl +\infty$.
We construct this way a sequence
$\{\mu_{l,t}\in\CC^\infty(X,E_{p,t}\otimes E_{p,0}^*)\}_{l\in\N}$,
with $\mu_{0,0}\equiv\Id_{E}$ and $\mu_{l,0}\equiv 0$
for all $l\in\N^*$,
satisfying \cref{nabPgP=Op} for all $k\in\N^*$.
Then by \cref{ptfla} and \cref{nabPgP=Op}, we have
$\nabla^p_{\delt}\left(\Tau_{p,t}-\sum_{l=1}^{k-1} p^{-l}
P_{p,t}\mu_{l,t}P_{p,0}\right)=O(p^{-k})$
in operator norm as $p\fl +\infty$ for any $k\in\N^*$,
so that there exists $C_k>0$ such that
for any $p\in\N^*$, $s_0\in\cinf(X,E_{p,0})$ and $t\in[0,1]$,
\begin{equation}\label{Gronest}
\begin{split}
\dt\Big\|\Big(&\Tau_{p,t}-\sum_{l=1}^{k-1} p^{-l} P_{p,t}\mu_{l,t}
P_{p,0}\Big)s_0\Big\|_{p,t}^2\\
&=2\re\left\langle\nabla^p_{\delt}\Big(\Tau_p-\sum_{l=1}^{k-1}
p^{-l}P_p\mu_{l,t}P_{p,0}\Big)s_0,\Big(\Tau_p-\sum_{l=1}^{k-1} p^{-l}
P_{p,t}\mu_{l,t}P_{p,0}\Big)s_0\right\rangle_{p,t}\\
&\leq 2C_k p^{-k} \|s_0\|_{p,0} \norm{\Big(\Tau_p-\sum_{l=1}^k
p^{-l} P_{p,t}\mu_{l,t}P_{p,0}\Big)s_0}_{p,t}.
\end{split}
\end{equation}
As $\Tau_{p,0}=\sum_{l=1}^{k-1} p^{-l} P_{p,0}
\mu_{l,0}P_{p,0}=P_{p,0}$,
by \cref{Gronest} and using Grönwall's lemma, we conclude
that $\Tau_p$ satisfies \cref{genToepdef},
with first coefficient $\mu_{0,t}$
equal to the solution of the ordinary differential equation \cref{eqft},
which is precisely \cref{Taucoeff} by definition \cref{tilmutdef}
of $\mu_t$.
\end{proof}

By \cref{GBTasy}, \cref{Approx} implies in particular that
for any $x_0\in X$, there exists a family
$\{G_{r,x_0}(Z,Z')\in E_{1,x_0}\otimes E_{0,x_0}^*\}_{r\in\N}$
of polynomials in
$Z,Z'\in\R^{2n}$ of the same parity as $r$ and smooth in
$x_0\in X,\,t\in[0,1]$, such that
\begin{equation}\label{expTau}
\Tau_{p,t}(Z,Z')\cong p^n\sum_{r=0}^{\infty} G_{r,t,x_0}
\PP_{t,x_0}\PP_0(\sqrt{p}Z,\sqrt{p}Z')p^{-\frac{r}{2}}+\cO(p^{-\infty}),
\end{equation}
in the sense of \cref{notcong}, with $G_{0,t,x_0}(Z,Z')=\mu_t(x_0)\tau^{E,\sigma}_t(x_0)$ for all $Z,\,Z'\in\R^{2n}$.

\section{Localization formulas}\label{locsec}

In this section, we use the results of \cref{pt} to prove an
asymptotic version of localization formulas of Lefschetz type
for the action of symplectic maps, proving \cref{indeqintro}.
In \cref{isosec}, we deal with the case of isolated fixed point.
In \cref{galsec}, we then establish the general case.

Let $(X,\om)$ be a compact symplectic manifold, and
let $\varphi:X\fl X$ be a diffeomorphism.
\begin{defi}\label{nondegdef}
The fixed point set $X^\varphi\subset X$ of a diffeomorphism $\varphi:X\fl X$
is said to be \emph{non degenerate} if $X^\varphi$ is a closed submanifold
such that $TX^\varphi=\Ker(\Id_{TX}-d\varphi)$ inside $TX$.
\end{defi}

Assume that $(X,\om)$ is equipped with a Hermitian line bundle
$(L,h^L)$ together with a
Hermitian connection
$\nabla^L$
whose curvature satisfies the prequantization condition
\cref{preq}, and that
$\varphi:X\fl X$ lifts
to a bundle map $\varphi^L:L\fl L$ preserving metric and connection,
so that in particular $\varphi$ preserves $\om$.
If $J_0$ is an almost complex structure on $X$ compatible with $\om$,
then the almost complex structure $J_1$ defined by
\begin{equation}\label{J_1}
J_1=d\varphi J_0d\varphi^{-1}
\end{equation}
is again compatible with $\om$. As the space of almost
complex structures compatible with $\om$ is contractible, there exists
a path $t\mapsto J_t$ joining $J_0$ to $J_1$ for $t\in[0,1]$, and
we can consider
the associated tautological fibration over $[0,1]$ as in \cref{trivfib}.
If $(E,h^E,\nabla^E)$ is an auxiliary Hermitian vector bundle with 
Hermitian connection over
$[0,1]\times X$, we
suppose that $\varphi$ lifts to a bundle map $\varphi^E:E_0\fl E_1$
over $X$, preserving metric and connection, and we write
$\varphi_p$ for the induced map on $E_p$ for any $p\in\N^*$.
The pullback of $s\in\cinf(X,E_p)$ by $\varphi_p$, defined for any
$x\in X$ by
\begin{equation}\label{pullbackphip}
({\varphi}_p^*s)(x)=\varphi_p^{-1}.s(\varphi(x)).
\end{equation}
induces by restriction a linear map $\varphi^*_p:\HH_{p,1}\fl\HH_{p,0}$
from the quantum space associated with 
$J_1$ to the one associated with $J_0$.

We omit in the sequel the subscript $1$ for any object depending on $t$
evaluated at $t=1$, and consider the local endomorphisms of
\cref{locmod} as functions
of $x_0\in X$ in the trivialization of \cref{trivfib}.
In particular, we write $\mu\in\cinf(X,\C)$ for the function
whose value at $x_0\in X$ is equal to $\mu_1$ in \cref{tilmutdef}.

In this section, the notation $O(p^{-k})$ is meant in its usual
sense as $p$ tends to infinity, uniformly in $x\in X$ and $t\in[0,1]$.
The notation $O(p^{-\infty})$ means $O(p^{-k})$ for all $k\in\N$.

\subsection{Isolated fixed points}\label{isosec}

We first deal with the case of $\varphi:X\fl X$ having only
non-degenerate isolated fixed points.
Recall the parallel transport operator $\{\Tau_{p,t}\}_{p\in\N^*}$
from $\HH_{p,0}$ to $\HH_{p,t}$ with respect to a Toeplitz connection
$\{\nabla^p\}_{p\in\N}$ defined by \cref{ptfla} for all $t\in[0,1]$.


For any $x\in X$ such that $\varphi(x)=x$,
write $\lambda_x:={\varphi}^{L,-1}_x\in\C$
through the canonical
identification $\End(L)\simeq\C$.
Recall the convention for square roots of complex determinants
stated at the end of \cref{setting}.
The following theorem is the main result of this section.

\begin{theorem}\label{asymplef}
Suppose that the fixed point set $X^\varphi\subset X$
of $\varphi:X\fl X$ is discrete and non-degenerate, and write
$X^{\varphi}=\{x_1,\dots,x_q\},\,q\in\N^*$.
Then for each $1\leq j\leq q$ and $r\in\N$,
there exists $a_{j,r}\in\C$,
such that for any $k\in\N$ and as $p\in\N^*$ tends to infinity,
\begin{equation}\label{leffle}
\Tr_{\HH_p}[\varphi^*_p\Tau_p]=\sum_{j=1}^q \lambda_{x_j}^p\sum_{r=0}^{k-1} p^{-r} a_{j,r} +O(p^{-k}).
\end{equation}
Furthermore, for any $1\leq j\leq q$, the following formula holds,
\begin{equation}\label{aj0lef}
a_{j,0}=\bar\mu^{-1}(x_j)\Tr_E[\varphi^{E,-1}_{x_j}
\tau^{E,\sigma}_{x_j}]
\det{}^{-\frac{1}{2}}\left[(\Pi_{0,x_j}^1-d\varphi_{x_j}^{-1}\overline{\Pi^0_1}{}_{,x_j})(\Id_{T_{x_j}X}-d\varphi_{x_j})\right].
\end{equation}
\end{theorem}
\begin{proof}

Recall that $\Tau_p$ admits a smooth Schwartz kernel with respect to
$dv_X$ for all $p\in\N^*$ and $t\in[0,1]$, so that in particular,
\begin{equation}\label{tr1}
\Tr_{\HH_p}[\varphi^*_p\Tau_p]=\int_X\Tr_{E_p}\left[{\varphi}^{-1}_p.\Tau_p(\varphi(x),x)\right] dv_X(x).
\end{equation} 
Let $\psi:TX\fl [0,1]\times X$
be a smooth family of vertical charts
constant along $[0,1]$ as in \cref{trivfib}, and let $\epsilon_0>0$ be 
such that $\psi$ restricted to $B^{T_xX}(0,\epsilon_0)$ is a
diffeomorphism on its image for any $x\in X$.
For all $x\in X$ and $0<\epsilon<\epsilon_0$, we write
$U_x(\epsilon):=\psi(B^{T_xX}(0,\epsilon))$ and we identify
$Z\in B^{T_xX}(0,\epsilon)$ with its image in
$U_{x_j}(\epsilon_0)$.

Let $\eta\in\R$ be the modulus of the smallest eigenvalue of
$\Id_{T_{x_j}X}-d\varphi_{x_j}\in\End(T_{x_j}X)$ for any $1\leq j\leq q$.
Then by \cref{nondegdef}, we know that $|\eta|>0$. Let us now consider
$\epsilon>0$ small enough so that
$\varphi(U_{x_j}(\epsilon))\subset U_{x_j}(\epsilon_0)$, for all
$1\leq j\leq q$. Then taking the Taylor expansion of $\varphi$,
we get $C_j>0$ such that for any $p\in\N^*,\,\theta\in[0,1]$
and all $Z\in B^{T_{x_j}X}(0,\epsilon)\simeq U_{x_j}(\epsilon)$ outside
$U_{x_j}(\epsilon p^{-\frac{\theta}{2}})$, we have
\begin{equation}\label{borneinf}
\begin{split}
|Z-\varphi(Z)| & \geq |(\Id_{T_{x_j}X}-d\varphi_{x_j})Z|-C|Z|^2\\
& \geq (\eta-C\epsilon p^{-\frac{\theta}{2}})\epsilon p^{-\frac{\theta}{2}}.
\end{split}
\end{equation}
On the other hand, by \cref{Approx}, we know that \cref{thetacong} holds
for $\Tau_p$, and we deduce from \cref{tr1}-\cref{borneinf} that there exists
$\epsilon>0$ such that for all $p\in\N^*$,
\begin{equation}\label{tr2}
\Tr_{\HH_p}[\varphi^*_p\Tau_p]=\sum_{j=1}^q
\int_{U_{x_j}(\epsilon p^{-\frac{\theta}{2}})}\Tr_{E_p}\left[{\varphi}^{-1}_p.\Tau_p(\varphi(x),x)\right] dv_{X}(x) +O(p^{-\infty}).
\end{equation}
%
Estimating \cref{tr2} term by term, we assume from now on that $\varphi$ has
only one fixed point, which we denote $x_0\in X$. Consider the 
trivialization around $x_0$ by parallel transport with respect to
$\nabla^L,\,\nabla^E$ along radial lines as in \cref{trivfib},
and identify $T_{x_0}X$ with $\R^{2n}$ by \cref{basis}.
Given any $s\in\cinf(U_{x_0}(\epsilon_0),L)$,
there exists a
smooth function $\lambda_{x_0}\in\cinf(B^{\R^{2n}}(0,\epsilon),\C)$,
defined for all $Z\in B^{\R^{2n}}(0,\epsilon)$ in these
coordinates by the formula
\begin{equation}\label{phiL=alpha}
\lambda_{x_0}(Z)s(Z)=({\varphi}^{L,*}s)(Z):={\varphi}^{L,-1}_Z.s(\varphi(Z)).
\end{equation}
In particular, as $\varphi(x_0)=x_0$, the unitary endomorphism
${\varphi}^{L,-1}_{x_0}$ acts on $L_{x_0}$ by multiplication by
$\lambda:=\lambda_{x_0}(0)\in\C$ with $|\lambda|=1$.

Recall that for any $v\in \cinf(X,TX)$ and $s\in\cinf(X,L)$,
we have by assumption,
%
%
\begin{equation}\label{varphinab}
\nabla^L_v({\varphi}^{L,*}s)={\varphi}^{L,*}(\nabla^L_vs).
\end{equation}
Let $Z:=\sum_{j=1}^{2n}Z_j\frac{\partial}{\partial Z_j}\in
\cinf(B^{\R^{2n}}(0,\epsilon),T_{x_0}X)$ be the \emph{radial vector field}
of the coordinates above.
It is a classical result,
which can be found for example in \cite[(1.2.31)]{MM07}, that in the trivialization of $L$ along radial lines,
the connection $\nabla^L$ at $Z\in B^{\R^{2n}}(0,\epsilon)$ has the form
\begin{equation}\label{nablatriv}
\nabla^L=d+\frac{1}{2}R^L(Z,.)+O(|Z|^2).
\end{equation}
Let us take $s\in\cinf(U_{x_0}(\epsilon_0),L)$ in \cref{phiL=alpha} to be
$1\in\C$ in our trivialization.
Then $s$ is parallel along radial lines, and the right hand side of
\cref{varphinab} vanishes for $v=Z$. Thus using \cref{phiL=alpha} and
\cref{nablatriv}, equation \cref{varphinab} becomes 
\begin{equation}\label{eqalpha}
\begin{split}
Z_j\frac{\partial}{\partial Z_j}\lambda_{x_0}(Z) &
=-\frac{1}{2}R^L(Z,Z)\lambda_{x_0}(Z)+O(|Z|^3)\lambda_{x_0}(Z)\\
& =O(|Z|^3)\lambda_{x_0}(Z).
\end{split}
\end{equation}
Solving the ordinary differential equation \cref{eqalpha},
we get a function $\zeta\in\cinf(B^{\R^{2n}}(0,\epsilon),\C)$
such that for all $Z\in B^{\R^{2n}}(0,\epsilon)$,
\begin{equation}\label{valalphaphi}
\lambda_{x_0}(Z)=\lambda e^{\zeta(Z)}\quad\text{with}\quad \zeta(Z)=O(|Z|^3)\,.
\end{equation}
Let $\{G_{r,x_0}(Z,Z')\in E_{1,x_0}\otimes E_{0,x_0}^*\}_{r\in\N}$ be the sequence
of polynomials in $Z,Z'\in\R^{2n}$ associated with
$\{\Tau_p\}_{p\in\N^*}$
in \cref{expTau}. By \cref{tr2},
\cref{phiL=alpha}
and \cref{valalphaphi}, for any $\delta\in\,]0,1[$ and $k\in\N$,
we get
a $\theta\in\,]0,1[$ such that for all
$p\in\N^*$,
\begin{multline}\label{trkerTau}
\Tr_{\HH_p}[\varphi^*_p\Tau_p]=p^n\lambda^p\sum_{r=0}^{k-1} p^{-\frac{r}{2}}
\int_{B^{\R^{2n}}(0,\epsilon p^{-\frac{\theta}{2}})}\Tr_{E}[{\varphi}^{E,-1}_Z G_{r,x_0}
\PP_{1,x_0}\PP_0(\sqrt{p}\varphi(Z),\sqrt{p}Z)]\\
e^{p\zeta(Z)}
dv_{X}(Z)+O(p^{-\frac{k}{2}+\delta}).
\end{multline}
Recall \cref{PtP0fla}. Considering the Taylor expansion up to second order
of $\varphi$, we know there exist smooth functions
$h_\alpha\in\cinf(B^{\R^{2n}}(0,\epsilon),\C)$ for all $\alpha\in\N^3$
such that for any $Z\in B^{\R^{2n}}(0,\epsilon)$,
\begin{multline}\label{Taylorphi3}
\frac{\pi}{2}\left(\<\Pi_{0,x_0}^1(\varphi(Z)-Z),(\varphi(Z)-Z)\>+2\sqrt{-1}\Omega(\varphi(Z),Z)\right)\\
=\frac{\pi}{2}\left(\<\Pi_{0,x_0}^1(d\varphi_{x_0}.Z-Z),(d\varphi_{x_0}.Z-Z)\>+2\sqrt{-1}\Omega(d\varphi_{x_0}.Z,Z)\right)\\
+\sum_{|\alpha|=3}Z^\alpha h_\alpha(Z).
\end{multline}
Using \cref{PPreal} and \cref{Taylorphi3}, we then get for any $Z\in B^{\R^{2n}}(0,\epsilon),\,k,p\in\N^*$,
\begin{multline}\label{TaylorPP}
\begin{split}
\PP_{1,x_0}\PP_0(\sqrt{p}&\varphi(Z),\sqrt{p}Z)\\
=\PP_{1,x_0}\PP_0&(\sqrt{p}d\varphi_{x_0}.Z,\sqrt{p}Z)\exp\left(-p\sum_{|\alpha|=3}Z^\alpha h_\alpha(Z)\right)\\
=\PP_{1,x_0}\PP_0&(\sqrt{p}d\varphi_{x_0}.Z,\sqrt{p}Z)\\
\Big(&\prod_{|\alpha|=3}\sum_{r=0}^{k-1} (-1)^r\frac{p^{-\frac{r}{2}}}{r!}(\sqrt{p}Z)^{r\alpha}h_\alpha(Z)^r+p^{-\frac{k}{2}}O(|\sqrt{p}Z|^k)\Big).
\end{split}
\end{multline}
Now for any multi-index $\alpha$ of length $3$ and any $r\in\N$, we can
consider Taylor expansion up to order $l\in\N$ of $h_\alpha^r$ as in
\cref{4.37} to get
%
%
$a_\beta\in\C$ for all $\beta\in\N^{2n}$
such that for any $Z\in B^{\R^{2n}}(0,\epsilon),\,k\in\N$
and $p\in\N^*$,
\begin{multline}
\label{PPphi=PPdphi}
\PP_{1,x_0}\PP_0(\sqrt{p}\varphi(Z),\sqrt{p}Z)=
\PP_{1,x_0}\PP_0(\sqrt{p}d\varphi_{x_0}.Z,\sqrt{p}Z)\\
\left(1+\sum_{j=1}^{k-1} p^{-\frac{j}{2}} \sum_{|\beta|=j} (\sqrt{p}Z)^{\beta}a_{\beta}
+p^{-\frac{k}{2}}O(|\sqrt{p}Z|^k)\right).
\end{multline}
Considering again the Taylor expansion of $\varphi$,
for any polynomial $G_{r,x_0}(Z,Z')$ in \cref{tr2},
we get a sequence $\{F_{r,j}(Z)\in E_{1,x_0}\otimes E_{0,x_0}^*\}_{j\in\N}$
of polynomials in $Z\in\R^{2n}$ of the parity of $r+j$
and some $d_k\in\N$, such that
\begin{equation}\label{TaylorHrphi}
G_{r,x_0}(\sqrt{p}\varphi^{-1}(Z),\sqrt{p}Z)=\sum_{j=0}^{k-1} p^{-\frac{j}{2}}
F_{r,j}(\sqrt{p}Z)+p^{-\frac{k}{2}}O(|\sqrt{p}Z|^{d_k}).
\end{equation}
Furthermore, by \cref{Taucoeff} we get that
$F_{0,0}(Z)=\mu(x_0)\tau^{E,\sigma}_{x_0}$ for any
$Z\in\R^{2n}$.
From \cref{PtP0fla} and \cref{borneinf}, we see that
the function $\PP_{1,x_0}\PP_0(d\varphi_{x_0}.Z,Z)$ is integrable over
$\R^{2n}$ and exponentially decreasing in $Z\in\R^{2n}$.
For any $k\in\N$, let $M_k\in\N$ be the sum of the degrees of
all polynomials involved in \cref{PPphi=PPdphi} and
\cref{TaylorHrphi}, and write
\begin{equation}
\delta'=\delta+(M_k+k+1+d_k)(1-\theta)/2\,.
\end{equation}
Then considering the Taylor expansion of
$\varphi^{E,-1}$ in
\cref{trkerTau}, 
applying the reasonning of \cref{TaylorPP}-\cref{PPphi=PPdphi} to
$e^{p\zeta(Z)}$ with $\zeta(Z)=O(|Z|^3)$ as in \eqref{valalphaphi}
and using \cref{PPphi=PPdphi}-\cref{TaylorHrphi},
we deduce the existence of a sequence
$\{H_{r,x_0}(Z)\in E_{1,x_0}\otimes E_{0,x_0}^*\}_{r\in\N}$ of
polynomials in $Z\in\R^{2n}$, of the
parity of $r$, smooth in $x_0$ and with
$H_{0,x_0}(Z)=\mu(x_0)\tau^{E,\sigma}_{x_0}$ for any
$Z\in\R^{2n}$,
such that via the change of variable $Z\mapsto\sqrt{p}Z$,
equation \cref{trkerTau} becomes
\begin{multline}\label{trkertau2}
\Tr_{\HH_p}[\varphi^*_p\Tau_p]=\lambda^p\sum_{r=0}^{k-1} p^{-\frac{r}{2}}\\
\int_{\R^{2n}}\Tr_{E}[{\varphi}^{E,-1}_{x_0}H_{r,x_0}\PP_{1,x_0}
\PP_0(\sqrt{p}d\varphi_{x_0}.Z,\sqrt{p}Z)] dZ
+O(p^{-\frac{k}{2}+\delta'}).
\end{multline}
Now by \cref{PtP0fla}, we see that
$\PP_{1,x_0}\PP_0(\sqrt{p}d\varphi_{x_0}.Z,\sqrt{p}Z)$ is even in
$Z\in\R^{2n}$, and as the polynomial $H_{2r+1}$ is odd for any $r\in\N$,
we get
\begin{equation}\label{odd}
\int_{\R^{2n}}\Tr_E[{\varphi}^{E,-1}_{x_0}H_{2r+1,t}\PP_{1,x_0}
\PP_0(\sqrt{p}d\varphi_{x_0}.Z,\sqrt{p}Z)] dZ=0.
\end{equation}
Thus from \cref{trkertau2} and \cref{odd}, we get
$b_r\in\C$ for all $r\in\N$ such that
\begin{equation}\label{trkertau3}
\Tr_{\HH_p}[\varphi^*_p\Tau_p]=\sum_{r=0}^{k-1} p^{-r}\lambda^p b_{r}
+O(p^{-\frac{k}{2}+\delta}).
\end{equation}
Using \cref{tr2} and \cref{trkertau3}, we can extend the above reasonning
to the case of $m$ fixed point to get \cref{leffle} in general.

Let us now compute 
\begin{equation}\label{M0}
b_{0,x_0}=\int_{\R^{2n}}\Tr_E[{\varphi}^{E,-1}_{x_0}\tau^{E,\sigma}_{x_0}]
\mu(x_0)\PP_{1,x_0}\PP_0(d\varphi_{x_0}.Z,Z)dZ.
\end{equation}
Using \cref{At0}-\cref{splitt0}, we know that
$\Pi_0^1+\overline{\Pi_1^0}=\Id_{\R^{2n}}$ and that
$\Om(\Pi_0^1\cdot,\cdot)=\Om(\cdot,\overline{\Pi_1^0}\cdot)$.
From \eqref{Pi0t'}, \cref{PtP0fla}, using $\Om(\cdot,\cdot)=\Om(d\varphi_{x_0}\cdot,d\varphi_{x_0}\cdot)$ and the identities above, for any $Z,Z'\in\R^{2n}$ we get
\begin{equation}\label{computP1P0phi}
\begin{split}
&\det(A_1^0)^{-\frac{1}{2}}\PP_{1,x_0}\PP_0(d\varphi_{x_0}.Z,Z)\\
&=\exp\left[-\pi\left(\left\langle\Pi_{0,x_0}^1(d\varphi_{x_0}.Z-Z),
(d\varphi_{x_0}.Z-Z)\right\rangle+\i\Om(d\varphi_{x_0}.Z,Z)\right)\right]\\
&=\exp\left[\i\pi\left(\Om((\Pi_{0,x_0}^1d\varphi_{x_0}-\Pi_{0,x_0}^1).Z,
d\varphi_{x_0}.Z-Z)+\Om(Z,d\varphi_{x_0}.Z-Z)\right)\right]\\
&=\exp\left[\i\pi\left(\Om(Z,(d\varphi_{x_0}^{-1}
\overline{\Pi_1^0}{}_{{}_{,\mathlarger{x_0}}}+\Pi_{0,x_0}^1)(d\varphi_{x_0}.Z-Z))\right)\right]\\
&=\exp\left[-\pi\left\langle(\Pi_{0,x_0}^1-
d\varphi_{x_0}^{-1}\overline{\Pi_1^0}{}_{{}_{,\mathlarger{x_0}}})(\Id_{\R^{2n}}
-d\varphi_{x_0}).Z,Z\right\rangle\right].
\end{split}
\end{equation}
For the last line, we used the fact that $J_0d\varphi_{x_0}^{-1}=d\varphi_{x_0}^{-1}J_1$ by definition. Using \cref{dtA0t} as in the proof of
\cref{barmut}, the formula \cref{M0} then becomes
\begin{multline}\label{b0x0}
b_{0,x_0}=\bar\mu^{-1}(x_0)
\Tr_E[{\varphi}^{E,-1}_{x_0}\tau^{E,\sigma}_{x_0}]\\
\int_{\R^{2n}}\exp\left[-\pi\left\langle(\Pi_{0,x_0}^1-d\varphi_{x_0}^{-1}
\overline{\Pi_1^0}{}_{{}_{,\mathlarger{x_0}}})
(\Id_{\R^{2n}}-d\varphi_{x_0})Z,Z\right\rangle\right]dZ.
\end{multline}
Using once again the identities mentioned above, we see that the
endomorphism inside the exponential of \cref{b0x0} is symmetric,
and its real part is obviously positive as $\PP_{1,x_0}\PP_0(d\varphi.Z,Z)$
decreases exponentially in $Z\in\R^{2n}$.
We can thus apply the classical formula for the Gaussian integral to get
\cref{aj0lef} from \cref{b0x0}. 
\end{proof}

\subsection{Higher dimensional fixed points}\label{galsec}

In this section, we extend the asymptotic expansion of \cref{asymplef} to the
case when the fixed point set $X^\varphi\subset X$
 of $\varphi$ is a submanifold of arbitrary
dimension satisfying the non-degeneracy condition of 
\cref{nondegdef}.

Let $N$ be a subbundle of $TX$ over $X^\varphi$ such that
$TX|_{X^\varphi}=N\oplus TX^\varphi$, and equip $N$ with the Euclidean
metric $g^N=g^{TX}_0|_N$. Let $|dv|_N$
be the density of $(N,g^{N})$.
We define a density $|dv|_{TX/N}$ over $X^\varphi$
by the formula
\begin{equation}\label{dvX=dvNdvXN}
|dv|_{TX}=|dv|_N|dv|_{TX/N}.
\end{equation}
Note that if $N$ is the normal bundle of $X^\varphi$ in $X$,
then $|dv|_{TX/N}$ is simply the Riemannian density of 
$(X^\varphi,g^{TX}_0|_{X^\varphi})$.
We write $P^N:TX|_{X^\varphi}\fl N$ for the orthogonal projection
with respect to $g^{TX}_0$. 

\begin{theorem}\label{indeq}
Suppose that the fixed point set $X^\varphi$ of $\varphi:X\fl X$ is
non-degenerate, and write
$X^\varphi=\coprod_{j=1}^q X^{\varphi}_j,\,q\in\N,$
for its decomposition
into connected components. Set
$d_j=\dim X^\varphi_j$, and for any $1\leq j\leq m$, let
$\lambda_j\in\C$ be the
constant value of ${\varphi}^{L,-1}$
restricted to $X^{\varphi}_j$. Then there are densities
$\nu_r$ over $X^\varphi$ for any $r\in\N$
such that for any $k\in\N$ and as $p\fl +\infty$,
\begin{equation}\label{indeqfle}
\Tr_{\HH_p}[\varphi^*_p\Tau_p]=\sum_{j=1}^q p^{d_j/2}\lambda_j^p
\left(\sum_{r=0}^{k-1} p^{-r} \int_{X_j^\varphi}\nu_r(x)
+O(p^{-k})\right).
\end{equation}
Furthermore, the following equality holds,
\begin{equation}\label{nu0}
\nu_0=\bar\mu^{-1}\Tr_E[\varphi^{E,-1}\tau^{E,\sigma}]
\det{}^{-\frac{1}{2}}_N
\left[P^N(\Pi_{0}^1-d\varphi^{-1}\overline{\Pi_1^0})
(\Id_{TX}-d\varphi)P^N\right]|dv|_{TX/N}.
\end{equation}
\end{theorem}
\begin{proof}
For any $1\leq j\leq q$, take a tubular neighborhood $U_j$ of
$X^{\varphi}_j$. Then by \cref{thetacong} and \cref{expTau}, we know that
\begin{equation}\label{trgal}
\Tr_{\HH_p}[\varphi^*_p\Tau_p]=\sum_{j=1}^q\int_{U_j} \Tr_{E_p}\left[{\varphi}_p^{-1}.\Tau_p(\varphi(x),x)\right] dv_X(x) +O(p^{-\infty}).
\end{equation}
Then as in the proof of \cref{asymplef}, we can assume 
without loss of generality
that $X^\varphi$ has only one connected component, and set
$d=\dim X^{\varphi}$. Furthermore, as all the computations are local
on $X^{\varphi}$, we can assume $X^\varphi$ oriented.
Write $dv_{X^\varphi}$ for the Riemannian volume form of
$(X^\varphi,g^{TX}_0|_{X^\varphi})$.

Let $|\cdot|_N$ be the norm on the subbundle $N\subset TX|_{X^\varphi}$ 
transverse to $TX^\varphi$ over $X^\varphi$
induced by $g^N=g^{TX}_0|_N$ as above.
Let $\psi$ be a smooth family of vertical charts constant along
$[0,1]$ as in \cref{trivfib},
and let $\epsilon_0>0$ be such that $\psi$
restricted to the ball bundle
$B^N(0,\epsilon_0):=\{w\in N~|~|w|_N<\epsilon_0\}$ is a diffeomorphism on
its image. Then $U(\epsilon_0):=\psi\left(B^N(0,\epsilon_0)\right)$
is a tubular neighborhood of $X^\varphi$ in $X$.

Let $dw$ be a Euclidean volume form on the fibres of $(N,g^{TN})$
such that
the volume form $dwdv_{X^\varphi}$ on the total space of $N$
is compatible
with the orientation of $X^\varphi$. Let $dv_{TX/N}$
be the volume form over $X^\varphi$ such that
for any $y\in X^\varphi$ and $w\in U_y(\epsilon_0)$,
\begin{equation}\label{volE2}
dv_X(y,w)=h(y,w)~dw dv_{X^\varphi}(y),
\end{equation}
for some function $h\in\cinf(U(\epsilon_0),\R)$ satisfying $h(y,0)=1$
for all $y\in X^\varphi$.
Following the proof of \cref{asymplef}, by \cref{nondegdef} and compactness
of $X^\varphi$, we know that
\begin{equation}
\inf\limits_{w\in N,\,|w|_N=1}|(d\varphi-\Id_{TX})w|_{g^{TX}_0}>0.
\end{equation}
Let us consider $\epsilon>0$ small enough so that
$\varphi(U(\epsilon))\subset U(\epsilon_0)$. As in \cref{borneinf}, we get
$\epsilon'>0$ such that for all $p\in\N^*$ and
$x\in X\backslash U(\epsilon p^{-\frac{\theta}{2}})$, we have
$\varphi(x)\in X\backslash U(\epsilon' p^{-\frac{\theta}{2}})$.
By \cref{Approx} and \cref{thetacong} as in \cref{tr2}, for
any $p\in\N^*$ we get
\begin{multline}\label{intNeps}
\Tr_{\HH_p}[\varphi^*_p\Tau_p]=\int_{U(\epsilon p^{-\frac{\theta}{2}})}\Tr_{E_p}\left[{\varphi}_p^{-1}.\Tau_p(\varphi(x),x)\right] dv_X(x)+O(p^{-\infty})\\
=\int_{y\in X^\varphi}\int_{B^{N_y}(0,\epsilon p^{-\frac{\theta}{2}})}\Tr_{E_p}\left[{\varphi}_p^{-1}.\Tau_p(\varphi(y,w),(y,w))\right]\\
h(y,w)dw dv_{TX/N}(y)+O(p^{-\infty}).
\end{multline}
Recall that we assumed $X^\varphi$ connected. By \cref{varphinab}, the unitary
endomorphism $\varphi^{L,-1}_y\in\End(L_y)\simeq\C$ identifies
with a constant complex number $\lambda\in\C$ such that $|\lambda|=1$.
Fix $x_0\in X^\varphi$. We will estimate the middle integral of the right hand side of \cref{intNeps} for $y=x_0$ in the coordinates and
trivialization of \cref{trivfib} .
From
\cref{expTau}, \cref{volE2}
and following \cref{phiL=alpha}-\cref{trkerTau}, we get a function
$\zeta\in\cinf(B^{T_{x_0}X}(0,\epsilon),\C)$ such that for any
$\delta\in\,]0,1[$ and $k\in\N$, there is $\theta\in\,]0,1[$ such that
for all $p\in\N^*$,
\begin{multline}\label{trlocgal}
\int_{B^{N_{x_0}}(0,\epsilon p^{-\frac{\theta}{2}})} \Tr_E[\varphi_{p,w}^{-1}\Tau_p(\varphi(w),w)]h(x_0,w) dw=\\
p^n\sum_{r=0}^{k-1} p^{-\frac{r}{2}}\lambda^p\int_{B^{N_{x_0}}(0,\epsilon p^{-\frac{\theta}{2}})}\Tr_E[\varphi^{E,-1}_wG_{r,x_0}\PP_{1,x_0}\PP_0(\sqrt{p}\varphi(w),\sqrt{p}w)]\\
e^{p\zeta(w)} dw+O(p^{-\frac{k}{2}+\delta}).
\end{multline}
By the argument of the proof of \cref{asymplef}, we deduce from
\cref{trlocgal} the existence of a sequence 
$\{H_{r,x_0}(Z,Z')\in E_{0,x_0}\otimes E_{1,x_0}^*\}_{r\in\N}$ of
polynomials in $Z,Z'\in\R^{2n}$ of the same parity as $r$, depending
smoothly in $x_0\in X$ and with
$H_{0,x_0}(Z,Z')=\mu(x_0)\tau^{E,\sigma}_{x_0}$,
such that for any $\delta\in\,]0,1[$ and $k\in\N$, there is a 
$\theta\in\,]0,1[$ such that for all $p\in\N^*$,
\begin{multline}\label{trlocgal2}
\int_{B^{N_{x_0}}(0,\epsilon p^{-\frac{\theta}{2}})} \Tr_E[\varphi_{p,w}^{-1}\Tau_p(\varphi(w),w)] h(x_0,w) dw\\
=p^{\frac{d}{2}}\sum_{r=0}^{k-1} p^{-\frac{r}{2}}\lambda^p\int_{N_{x_0}}\Tr_E\left[\varphi^{E,-1}_{x_0}H_{r,x_0}\PP_{1,x_0}\PP_0(d\varphi_{x_0}.w,w)\right]
dw
+p^{\frac{d}{2}}O(p^{-\frac{k}{2}+\delta}).
\end{multline}
Recall that $\Ker(d\varphi_{x_0}-\Id_{T_{x_0}X})\cap N_{x_0}=\{0\}$
by assumption, so that
the integrand of \cref{trlocgal2} decreases exponentially in $w\in N_{x_0}$
by \cref{PtP0fla}.
Then repeating the arguments of \cref{trkertau2}-\cref{trkertau3}
and by integrating with respect to $x_0\in X^\varphi$, we produce from
\cref{intNeps} and \cref{trlocgal2} a sequence $\{\nu_r\}_{r\in\N}$ of
densities over $X^\varphi$, such that under the assumption of a unique
connected component of $X^\varphi$, we have
\begin{equation}
\Tr_{\HH_p}[\varphi^*_p\Tau_p]=\lambda^p p^{\frac{d}{2}}\sum_{r=0}^{k-1} p^{-r}\int_{X^\varphi}\nu_r+p^{\frac{d}{2}}O(p^{-\frac{k}{2}+\delta}),
\end{equation}
and for any $y\in X$,
\begin{equation}\label{nu0comput}
\nu_0(y)=\Tr_E[\varphi^{E,-1}_{y}\tau^{E,\sigma}_y
]\mu(y)
\left(\int_{N_y}\PP_{1,y}\PP_0(d\varphi_{y}.w,w)dw\right)
|dv|_{TX/N}(y).
\end{equation}
Then the computation of \cref{nu0} from \cref{nu0comput}
is analogous to \cref{computP1P0phi}.
The case of multiple components follows by \cref{trgal} and linearity.
\end{proof}

\section{Applications}\label{Appli}

In \cref{geomint}, we interpret the formulas found in 
\cref{asymplef,indeq}, and
show that they are compatible
with the local equivariant index formula in the holomorphic case.
In \cref{hitsec}, we introduce the notion of a Hitchin connection,
and relate it with the notion of Toeplitz connection introduced
in \cref{connectiontoep}.
In \cref{Witsec}, we
present an application to Witten's asymptotic expansion conjecture
for the quantum representations of the mapping class group.

\subsection{Geometric interpretation}\label{geomint}
%
Recall from \cref{setting}
that the relative
canonical line bundle $K_X:=\det(T^{(1,0)*}X)$ of a prequantized
fibration is endowed with the connection $\nabla^{K_X}$
induced by the vertical Levi-Civita connection $\nabla^{TX}$
defined by \cref{LCTX}. In this section, we will make use of the
isomorphism
\begin{equation}\label{isoTXT*X}
\begin{split}
T^{(0,1)}X&\longrightarrow T^{(1,0)*}X\\
\overline{w}&\longmapsto g^{TX}(\overline{w},\cdot)
\end{split}
\end{equation}
of complex vector bundles over $X\times [0,1]$
induced by the relative Riemannian
metric $g^{TX}$ seen as a $\C$-bilinear form over $TX_\C$.
Via this isomorphism,
the line bundle
with connection $(K_X,\nabla^{K_X})$ identifies with
the line bundle $\det(T^{(0,1)}X)$ endowed with the connection
$\nabla^{\det(T^{(0,1)}X)}$ induced by $\nabla^{TX}$.
The following lemma gives a geometric interpretation of the
function $\mu_t\in\cinf(X,\C)$ defined by formula \cref{tilmutdef}
in the local model, and allows to deduce \cref{Approxintro}
from \cref{Approx}.

\begin{lem}\label{mutgeom}
For any $t\in[0,1]$, we have following formula
for the function $\mu_t\in\cinf(X,\C)$ appearing
\cref{Approx},
\begin{equation}\label{barmutiltau=L}
\bar{\mu}_t^{2}=\det(\overline{\Pi_t^0})^{-1}\tau^{K_X}_t\,,
\end{equation}
where $\tau^{K_X}_t:K_{X,0}\fl K_{X,t}$ is the parallel
transport in $K_X$ over
horizontal lines of $[0,1]\times X$
with respect to $\nabla^{K_X}$ and
$\det(\overline{\Pi_t^0}):K_{X,0}\fl K_{X,t}$ is the
bundle isomorphism induced by $\overline{\Pi_t^0}$ via the isomorphism
\eqref{isoTXT*X}.
\end{lem}
\begin{proof}
Consider the
setting of \cref{trivfib}. Using Koszul formula and the 
definition \cref{LCTX} of $\nabla^{TX}$, we know that
\begin{equation}
\nabla^{TX}_{\delt^H}=\dt+\frac{1}{2}J_t\left(\dt J_t\right)
\end{equation}
in the tautological fibration $\pi:[0,1]\times X\fl [0,1]$.
Thus by \cref{nab10=Pnab}, for all $t\in[0,1]$ we have
\begin{equation}\label{nabdt=PdtP}
\nabla^{T^{(0,1)}X}_{\delt^H}=P^{(0,1)}_t\dt P^{(0,1)}_t.
\end{equation}
Recall the notations of \cref{locmod}, and
for any $t\in[0,1]$, let
$\overline{\Pi_t^0}\in\End(TX_\C)$
be the projection on $T^{(0,1)}X_t$ with kernel
$T^{(1,0)}X_0$, which restricts to a natural bundle
isomorphism between $T^{(0,1)}X_0$ and $T^{(0,1)}X_t$.
Using \cref{At0}-\cref{Pi0t'}, for any $w\in T^{(0,1)}X_0$,
we compute
\begin{equation}\label{PdtPPit=muPit}
\begin{split}
P^{(0,1)}_t\dt P^{(0,1)}_t\overline{\Pi_t^0}.w&=
P^{(0,1)}_t\left(\dt A_t^0\right)P^{(0,1)}_0.w\\
&=\left(-\frac{1}{2}\overline{\Pi_t^0}\dt(-J_0J_t)\right)\overline{\Pi_t^0}.w.
\end{split}
\end{equation}
By the definition \cref{tilmutdef}
of $\mu_t$ and via the isomorphism \eqref{isoTXT*X} induced by
the relative Riemannian metric $g^{TX}$, this shows
\eqref{barmutiltau=L}.
%
\end{proof}

Recall \cref{nutil} and \cref{dvX=dvNdvXN}, and consider
the context of \cref{locsec}.
The following Lemma gives a geometric interpretation for the
localization formula \cref{nu0},
and allows to deduce \cref{indeqintro}
from \cref{indeq}.

\begin{lem}\label{coeffgeom}
Suppose that the tangent bundle $TX$ over the fixed point set
$X^\varphi$ admits a decomposition $TX=TX^\varphi\oplus N$
preserved by $d\varphi$ and $J_0$.
Then over any connected component of $X^\varphi$ of dimension $2d$,
the first coefficient \cref{nu0} in \cref{indeq} satisfies
\begin{equation}\label{aj02}
\nu_0=(-1)^{\frac{n-d}{2}}
\Tr_E[\varphi^{E,-1}\tau^{E,\sigma}]
(\varphi^{K_X}\tau^{K_X,-1})^{1/2}|\det{}_N(\Id_{TN}-d\varphi|_N)|^{-\frac{1}{2}}
\,|dv|_{TX/N},
\end{equation}
for some natural choice of square roots depending on $\varphi$
and the choice of the path.
\end{lem}
\begin{proof}
Let $\varphi:X\fl X$ be a diffeomorphism lifting to $(L,h^L,\nabla^L)$,
sending $J_0$ to $J_1$ as in \cref{J_1}.
Let $x\in X$ be a fixed point of $\varphi$. Using that
$\Pi_{0,x}^1,\,\overline{\Pi_1^0}{}_{{}_{\mathlarger{,x}}}
\in\End(T_xX_\C)$ are the projection operators on
$T^{(1,0)}_xX_0$, $T^{(0,1)}_xX_1$ with
kernel $T^{(0,1)}_xX_1,\,T^{(1,0)}_xX_0$, and that
$d\varphi.\,T^{(1,0)}_{x}X_0=T^{(1,0)}_{x}X_1$, we know that
\begin{equation}\label{Pi-dphipi}
\begin{split}
&(\Pi_{0,x}^1-d\varphi^{-1}_{x}\overline{\Pi_1^0}{}_{{}_{\mathlarger{,x}}}
).\,v=v
~~\text{for any}~~v\in T^{(1,0)}_xX_0,\\
&(\Pi_{0,x}^1-d\varphi^{-1}_{x}\overline{\Pi_1^0}{}_{{}_{\mathlarger{,x}}}
)\,d\varphi_x.v=
-v~~\text{for any}~~v\in T^{(0,1)}_xX_0.
\end{split}
\end{equation}
In particular, as $T^{(1,0)}X_0$ and
$d\varphi.\,T^{(1,0)}_{x}X_0=T^{(1,0)}_{x}X_1$ are transverse
by \cref{splitt0}, we get that
$(\Pi_{0,x}^1-d\varphi^{-1}_{x}\overline{\Pi_1^0}
{}_{{}_{\mathlarger{,x}}}
)$ is invertible, and
 if $\{w_j\}_{j=1}^n$ is a basis of $T^{(1,0)}_{x}X_0$,
\begin{equation}\label{varphidet}
\begin{split}
\det(\Pi_{0,x}^1-d\varphi^{-1}_{x}\overline{\Pi_1^0}
{}_{{}_{\mathlarger{,x}}})^{-1}&=
\frac{w_1\wedge\dots\wedge w_n\wedge
(-d\varphi_{x}).\overline{w}_1\wedge\dots\wedge(-d\varphi_{x}).
\overline{w}_n}
{w_1\wedge\dots\wedge w_n\wedge\overline{w}_1
\wedge\dots\wedge\overline{w}_n}\\
&=(-1)^n\frac{d\varphi_{x}.\overline{w}_1\wedge\dots\wedge
d\varphi_{x}.\overline{w}_n}
{\overline{\Pi_1^0}{}_{{}_{\mathlarger{,x}}}\overline{w}_1\wedge
\dots\wedge \overline{\Pi_1^0}{}_{{}_{\mathlarger{,x}}}
\overline{w}_n}\\
&=(-1)^n\varphi^{K_X}_{x}
\det(\overline{\Pi_1^0})^{-1}_{x}\,,
\end{split}
\end{equation}
with $\det(\overline{\Pi_t^0}):K_{X,0}\fl K_{X,t}$ as in \cref{mutgeom}
and $\varphi^{K_X}:K_{X,0}
\fl K_{X,1}$ the natural bundle map induced by $\varphi$
via the isomorphism \eqref{isoTXT*X}, which commutes with
the natural action of $\varphi$ as
$g^{TX}_1(\cdot,\cdot)=g^{TX}_0(d\varphi\cdot,d\varphi\cdot)$
by definition.
Then in the notations of \cref{locsec}, by \cref{barmutiltau=L} and \cref{varphidet} we get
\begin{equation}\label{varphidet1/2}
\begin{split}
\bar{\mu}^{-2}(x)\det{}(\Pi_{0,x}^1-d\varphi^{-1}_{x}
\overline{\Pi_1^0}{}_{{}_{\mathlarger{,x}}})
&=(-1)^{n}\varphi^{K_X}_{x}\bar{\mu}^{-2}(x)
\det(\overline{\Pi_1^0})^{-1}_{x}\\
&=(-1)^{n}\varphi^{K_X}_{x}
(\tau^{K_X}_{x})^{-1}.
\end{split}
\end{equation}
Assume now that $d\varphi$ and $J_0$ preserve a decompostion
$TX=TX^\varphi\oplus N$ over $X^\varphi$.
Then by \cref{Pi-dphipi}, we see that
$(\Pi_{0}^1-d\varphi^{-1}\overline{\Pi_1^0})$ over
$X^\varphi$ preserves $N$ as well, and that its restriction to $N$
is invertible. Then by computations analogous to
\cref{varphidet} and by \cref{varphidet1/2}, over any connected
component of $X^\varphi$ of dimension $2d$, we get
\begin{equation}\label{coeffint}
\begin{split}
\bar{\mu}^{-2}\det{}_N
\Big[P^N(\Pi_{0}^1-&d\varphi^{-1}\overline{\Pi_1^0})
(\Id_{TX}-d\varphi)P^N\Big]^{-1}\\
&=\bar{\mu}^{-2}\det{}_N\left[(\Pi_{0}^1-d\varphi^{-1}
\overline{\Pi_1^0})|_N\right]^{-1}\det{}_N(\Id_N-d\varphi|_N)^{-1}\\
&=(-1)^{n-d}\varphi^{K_X}(\tau^{K_X})^{-1}
\det{}_{N}(\Id_N-d\varphi|_N)^{-1}.
\end{split}
\end{equation}
This together with \cref{nu0} shows \cref{aj02}.
\end{proof}
We see that formula \cref{nu0} acquires a natural interpretation
in terms of parallel transport in a square root of $K_X$, which always exists locally.
In particular, if $c_1(TX)$ is even, so that
we can take $E=:K_X^{1/2}$ whose square is equal to $K_X$ with
induced metric and
connection, and if there is a lift of $\varphi$ to
$K_X^{1/2}$ squaring to $\varphi^{K_X}$ on $K_X$,
Then \cref{aj02} simplifies and we recover \cite[Th.5.3.1]{Cha10}
as a special case, when $(X,J,\om)$ Kähler and $X^\varphi$ discrete.
Consider now the context of \cref{genTeopsec}.
\begin{lem}\label{Tautologiclem}
Let $\pi:B\times X\fl B$ be a tautological fibration with relative complex structure $J\in\End(TX)$
and auxiliary vector bundle $(E,h^E,\nabla^E)$,
and let $\gamma:[0,1]\fl B$ be such that $\gamma(0)=\gamma(1)=b_0\in B$.
Let
$\Tau_p\in\End(\HH_{p,b_0})$ be the parallel transport in $\HH_p$
along $\gamma$ with respect to a Toeplitz connection,
for all $p\in\N^*$ big enough. Then there exists $C>0$ such that as
$p\fl +\infty$,
\begin{equation}\label{Tautologic}
\|\Tau_p-P_p\,(\tau^{K_X})^{-\frac{1}{2}}\,
\tau^{E,\sigma}P_p\|_{p,b_0}<Cp^{-1},
\end{equation}
for an appropriate choice of square root of the parallel transport
$\tau^{K_X}\in\End(K_X)\simeq\cinf(X,\C)$ in $K_X$ along $\gamma$.
\end{lem}
\begin{proof}
Restricting the fibration over $\gamma$ as in \cref{trivfib}, we see that
the diffeomorphism $\tau_1:X_0\fl X_1$ defined in \cref{tausmall}
is the identity of $X\simeq X_1=X_0$, so that $J_1=J_0$, 
$P_{p,1}=P_{p,0}$ and $\det(\overline{\Pi_1^0})=1$.
Then \cref{Tautologic} is a consequence of \cref{Approx},
comparing \cref{Taucoeff} with \cref{barmutiltau=L}.
\end{proof}
%
%
In the case $\gamma$ is contractible in $B$, if the fibration is
holomorphic as in \cref{hol} and for the $L^2$-connection, this result
is a consequence of the computations in \cite[Th.2.1]{MZ07} of the
curvature of $\nabla^{\HH_p}$.
Note that \cite[Th.2.1]{MZ07} applies for general Kähler fibrations, which need not be tautological, and the same is true
for \cref{Approx}.
Taking $\varphi:X\fl X$ to be the diffeomorphism identifying
$X_1$ with $X_0$ via \cref{tausmall} on a loop $\gamma:[0,1]\fl B$,
we recover in general the first coefficient of \cite[(8)]{MZ07}
as the contribution of $\varphi_p:L^p\fl L^p$ through the description of the curvature as the derivative of the parallel transport.

\subsection{Hitchin connections}
\label{hitsec}

In this section, we describe an important class of
Toeplitz connections called \emph{Hitchin connections},
which appear naturally in the context of geometric quantization
of moduli spaces. These are the connections used
in the application of \cref{indeq} to Witten's asymptotic
expansion conjecture, which we describe in \cref{Witsec}.

Consider a holomorphic prequantized fibration 
$\pi:M\fl B$, and fix $p_0\in\N$
as in \cref{specdeltapphi} for $U=B$.
Recall that $J\in\End(TX)$ over $M$ denotes the associated
relative compatible complex structure.
For any $v\in\cinf(B,TB)$, let $\tau_t^v$ be the flow on $M$
generated by
its horizontal lift $v^H\in\cinf(M,TM)$ with respect to
\cref{split} at time $t\in\R$. Then $\tau_t^v$
preserves the fibres, so that its differential
restricts to a bundle map $d\tau_t^{v}:TX\fl TX$.
Define the Lie derivative of $J$ by the formula
\begin{equation}
L_v J=\dt\Big|_{t=0}d\tau_t^{v}.\,J.\,(d\tau_t^{v})^{-1}\in\End(TX).
\end{equation}
Then $L_vJ\in\End(TX)$ exchanges $T^{(1,0)}X$ and $T^{(0,1)}X$
inside $TX$ as in \cref{splitc}.
For any $w\in TX_\C$, recall that $\iota_w\in\End(\Lambda (T^*X_\C))$
denotes the contraction by $w$.
For any $w_1,\,w_2\in TX_\C$, we define
\begin{equation}
\begin{split}
i_{w_1\otimes w_2}:\Lambda (T^* X_\C) &\fl TX_\C\otimes
\Lambda (T^* X_\C)\\
\alpha &\mapsto w_1\otimes \iota_{w_2}\alpha,
\end{split}
\end{equation}
and we extend this
definition to all of $TX_\C\otimes TX_\C$ by linearity.
Then for any $A\in TX_\C\otimes TX_\C$, we can consider
$i_A\om\in TX_\C\otimes T^*X_\C$ as an element of $\End(TX_\C)$.
Following \cite[\S\,1]{And12}, we define a section
$G\in\cinf(M,\pi^*T^*B\otimes T^{(1,0)}X\otimes T^{(1,0)}X)$
by the following
formula, for all $v\in TB$ and $w\in T^{(0,1)}X$,
\begin{equation}\label{defG}
L_vJ.w=2\pi(i_{G(v)}\om).w.
\end{equation}
%
%
%
%
We still write $\nabla^{T^{(1,0)}X}$ for the connection on
$T^{(1,0)}X\otimes T^{(1,0)}X$ induced by \cref{nab10=Pnab}.
The following definition can be found in
\cite[Def.1]{And12}.

\begin{defi}\label{rigid}
An holomorphic prequantized fibration is said to be \emph{rigid}
if for all $v\in\cinf(B,TB)$, we have
\begin{equation}
\nabla^{T^{(1,0)}X}G(v)=0.
\end{equation}
\end{defi}

For any vector bundle with connection $(E,\nabla^E)$ over $M$ and any
$A\in T^{(1,0)}M^{\otimes 2}$,
let $\Delta^{E}_A$ be the second order 
differential operator in the fibres of $\pi:M\fl B$,
defined for any $w_1,w_2\in T^{(1,0)}X$ by the formula
\begin{equation}
\Delta^{E}_{w_1\otimes w_2}s=\Tr[\nabla^{T^{(1,0)}X\otimes E}(w_1\otimes\nabla^E_{w_2}s)].
\end{equation}
%
Recall \cref{split}, and suppose that there exists $k\in\N^*$
and a function $\rho\in\cinf(M,\R)$ such that
\begin{equation}\label{ric=om+rho}
k\,\om^X
=\frac{\sqrt{-1}}{2\pi}\big[\Tr R^{T^{(1,0)}X}
+\overline{\partial}\partial\rho\big]^X.
\end{equation}
Using the notations of \cref{setting}, we are now ready to state
the following theorem, which is originally due to
\cite[\S\,3]{Hit90} and
\cite[\S\,4.b]{ADW91}. In this generality, it is due to
\cite[Th.1]{And12}.

\begin{theorem}\label{hitchin}
Let $\pi:M=B\times X\fl B$ be a rigid holomorphic
tautological fibration with $E=\C$.
Suppose that $X$ is simply connected and that \cref{ric=om+rho}
holds for some $\rho\in\cinf(X,\R)$.
For any $p\in\N^*$ and $v\in\cinf(B,TB)$,
let $\nabla^p_v$ be the differential operator acting on
$\cinf(M,L^p)$ by
\begin{equation}\label{hitchinfla}
\nabla^p_v=\nabla^{L^p}_{v^H}+\frac{1}{4p+2k}
\left(\Delta^{L^p}_{G(v)}-
\nabla^{L^p}_{i_{G(v)}d\rho}+2p\,\partial\rho.v^H\right).
\end{equation}
Then $\nabla^p_v$ preserves holomorphicity in the fibres for any
$v\in\cinf(B,TB)$ and $p\in\N^*$, and thus induces by restriction
a family of connections $\{\nabla^p\}_{p\in\N^*}$ on the 
quantum bundle $\{\HH_p\}_{p\in\N^*}$ over $B$.
\end{theorem}
In particular, using parallel transport with respect to
\cref{hitchinfla}, this shows
that the dimension of the space of holomorphic sections is constant,
so that the quantum bundle is well defined for any $p\in\N^*$.
%
Recall \cref{connectiontoep}.
\begin{lem}\label{hitchintoep}
Under the hypotheses of \cref{hitchin},
the family of connections $\{\nabla^p\}_{p\in\N^*}$
defined by
\cref{hitchinfla} is a Toeplitz
connection.
Furthermore, its first coefficient in \cref{connToepexp}
satisfies
\begin{equation}\label{nu=f/2}
\sigma_0=\frac{1}{2}\partial\rho.
\end{equation}
\end{lem}
\begin{proof}
By \cref{hitchin}, the differential operator
$\nabla^p_v$ defined in \cref{hitchinfla} for any $v\in\cinf(B,TB)$
and $p\in\N^*$ preserves the fibrewise holomorphic sections, so that
\cref{hitchinfla} rewrites
\begin{equation}\label{hitchintoepfla}
\begin{split}
\nabla^p_v&=P_p\left[\nabla^{L_p}_{v^H}+\frac{1}{4p+2k}
\left(\Delta^{L^p}_{G(v)}-\nabla^{L^p}_{i_{G(v)}d\rho}
+2p\,\partial\rho.v^H\right)\right]P_p\\
&=\nabla^{\HH_p}_v+\frac{1}{4p+2k}P_p
\left(\Delta^{L^p}_{G(v)}-\nabla^{L^p}_{i_{G(v)}d\rho}
+2p\,\partial\rho.v^H\right)P_p
\end{split}
\end{equation}
Now by straightforward computations in the spirit of
\cite[Th.2.1]{Tuy87},
there exist functions $g,\,h\in\cinf(M,\C)$ such that
\begin{equation}\label{Tuy}
\begin{split}
P_p\Delta^{L^p}_{G(v)}P_p&=P_p g P_p,\\
P_p\nabla^{L^p}_{i_{G(v)}d\rho}P_p&=P_p h P_p.
\end{split}
\end{equation}
We can then take the expansion in $p\in\N^*$ of the multiplicative
constant in front of the second term of \cref{hitchintoepfla}
to get the result by \cref{connToepexp}.
\end{proof}
Let us now describe a variant of \cref{hitchin} including a square
root of $K_X$, which is originally due to 
\cite[\S\,3]{SS95}, and which in this generality
is due to \cite[Th.1.2]{AGL12}.
%
%
%
%
%
%

\begin{theorem}\label{metahitchin}
Under the assumptions of \cref{hitchin}, with $c_1(TX)$ even and
taking $E=K_X^{1/2}$,
there exists a $1$-form $\beta\in\cinf(M,T^*M_\C)$ such that for any
$p\in\N^*$ and $v\in\cinf(B,TB)$, the differential operator
$\nabla^p_v$
acting on $\cinf(M,L^p\otimes K_X^{1/2})$ defined by
\begin{equation}\label{metahitchinfla}
\nabla^p_v=\nabla^{E_p}_{v^H}+\frac{1}{4p}\left(
\Delta^{E_p}_{G(v)}+\beta(v^H)\right),
\end{equation}
preserves holomorphicity in the fibres, for any
$v\in\cinf(B,TB)$ and $p\in\N^*$, and thus induces by restriction
a family of connections $\{\nabla^p\}_{p\in\N^*}$ on the 
quantum bundle $\{\HH_p\}_{p\in\N^*}$ over $B$.
\end{theorem}

We then have the following analogue of \cref{hitchintoep}, whose proof is strictly analogous.

\begin{lem}\label{hitchintoepmeta}
Under the hypotheses of \cref{metahitchin},
the family of connections $\{\nabla^p\}_{p\in\N^*}$ defined by
\cref{metahitchinfla} is a Toeplitz
connection.
Furthermore, its first coefficient in \cref{connToepexp}
satisfies $\sigma_0=0$.
\end{lem}


Consider now a diffeomorphism $\varphi:X\fl X$ lifting to a bundle
map $\varphi:L\fl L$ preserving metric and connection,
and assume that the induced action
of $\varphi$ on the space of compatible complex structures preserves
$B$.
%
%
%
Then by \cref{hitchintoep}, if $\varphi$
has non-degenerate fixed point set, we can apply
\cref{indeq} to the situation of \cref{hitchin}.
We recover in this way the
result of \cite[Th.1.2]{AP18}, and moreover, we give an explicit
formula for the first coefficient.
On the other hand, by \cref{hitchintoepmeta} we can apply
\cref{indeq} to the situation of \cref{metahitchin},
under the additional hypothesis that there exists a lift of
$\varphi$ to $K_X^{1/2}$ squaring to $\varphi^{K_X}$.
In the case $\dim X^\varphi=0$, this is
the result of \cite[Th.6.3,\,Th.6.4]{Cha16}, with
the corresponding formula \cref{aj02} for the first coefficient.

\subsection{Witten's aymptotic conjecture}\label{Witsec}
%
%
In this section, we explain how the results of \cref{locsec} 
apply to
Witten's asymptotic expansion conjecture
for quantum representation of the mapping class
group described in the introduction.

Let $\Sigma$ be an oriented compact surface of genus $g\geq 2$,
and let $P$ be the trivial $\SU(m)$-principal bundle over $\Sigma$,
with $m\geq 2$.
Let $D$ be a disk inside $\Sigma$, and choose $d\in(\Z/m\Z)^*$.
We consider the space $\cal{A}$
of flat connections on $P$
over $\Sigma\backslash D$ with
holonomy along its boundary equal to
$e^{\frac{2\pi\sqrt{-1}d}{m}}\,\Id_{\C^m}\in\SU(m)$.
The gauge group $\Aut(P)$ of automorphisms of $P$ acts
naturally on $\AA$, and it is a basic fact that
all $A\in\AA$ are \emph{irreducible}, meaning that the stabiliser
of $A$ in $\Aut(P)$ is the center of $\Aut(P)$, which identifies with
the center of $\SU(m)$.

Let $\AA'\subset\AA$ be the connections
with some fixed standard form in a neighbourhood of the boundary
of $\Sigma\backslash D$, and note that any $[A]\in\AA/\Aut(P)$ has a representative in $\AA'$. If $t\mapsto A_t\in\AA'$
is smooth in $t\in\R$,
we can extend $\frac{d}{dt}\big|_{t=0} A_t$ by $0$
over $D$ and consider it as an element of
$\Om^1(\Sigma,\mathfrak{su}(m))$.
Any $A\in\AA'$ induces a flat connection on the trivial adjoint bundle
$\ad P$ over $\Sigma\backslash D$
with trivial holonomy along the boundary,
so that it extends to a covariant derivative $d_A$
on $\Om^\bullet(\Sigma,\mathfrak{su}(m))$ satisfying $d_A^2=0$. Let
$H^\bullet_A(\Sigma)$ denote its cohomology.
The following result is classical and is essentially due to
\cite[p.\,587,\,Th.9.12,\,\S\,14]{AB83} and
\cite[\S\,1.2,\,\S\,1.8]{Gol84}.
\begin{prop}\label{MM}
The quotient $\MM=\AA/\Aut(P)$ is simply connected and
has a natural structure of a smooth compact manifold.
For any $A\in\cal{A}'$, there is a natural
isomorphism
\begin{equation}\label{tgtMM}
T_{[A]}\MM\simeq H^1_A(\Sigma),
\end{equation}
sending the tangent vector at $t=0$ of a smooth curve
$t\mapsto A_t\in\AA'$ with $A_0=A$ to the cohomology
class of $\frac{d}{dt}\big|_{t=0} A_t\in\Om^1(\Sigma,\mathfrak{su}(m))$.
Furthermore, for any $\alpha,\,\beta\in\Om^1(\Sigma,\mathfrak{su}(m))$
with $d_A\alpha=d_A\beta=0$, the formula
\begin{equation}\label{omAB}
\om^\MM([\alpha],[\beta])=
-\frac{m}{4\pi^2}\int_\Sigma\Tr(\alpha\wedge\beta)
\end{equation}
defines a symplectic form $\om^\MM$ on $\MM$ via \cref{tgtMM}.
\end{prop}
Note that the symplectic form \cref{omAB} is $m$ times the
canonical symplectic form of $\MM$, as defined in \cite[p.587]{AB83}.

Let $f\in\Diff^+(\Sigma,D)$ be an orientation
preserving diffeomorphism of $\Sigma$ fixing $D$ pointwise,
inducing a diffeomorphism 
\begin{equation}\label{varphif}
\begin{split}
\varphi:\MM&\longrightarrow\MM\\
[A]&\longmapsto[f^*A]
\end{split}
\end{equation}
preserving \cref{omAB}.
The following result is classical and can be deduced from
 \cite[\S\,4.2]{BL99} and
\cite[\S\,3.3,\,\S 4.3]{MW01}.

\begin{prop}\label{LCS}
There is a natural Hermitian line bundle
with connection $(L,h^{L},\nabla^{L})$ over $(\MM,\om^\MM)$ whose
curvature satisfies the prequantization condition \cref{preq}, and
a natural lift $\varphi^L$
of \cref{varphif} to $L$ preserving metric and connection.
\end{prop}
Note that this line bundle is the $m$-th power of the \emph{Chern-Simons
bundle} of \cite[\S\,3.3]{MW01}, which is only an orbifold line
bundle in general. The lift $\varphi^L$ is the one acting trivially on the
second summand of $\AA\times\C$ in the description of
this line bundle as a quotient given at the end of
\cite[\S\,3.3]{MW01}.

Recall that the group $\Diff_0(\Sigma)$ of diffeomorphisms isotopic
to the identity acts naturally on the space $\JJ_\Sigma$ of
complex structures on $\Sigma$. As explained for
example in \cite[\S\,3-4]{AC09},
the \emph{Teichmüller space}
$\cT_\Sigma:=\JJ_\Sigma/\Diff_0(\Sigma)$
has a natural structure of a complex manifold.
Following \cite[\S\,2]{Hit90}, we identify complex structures on $\Sigma$
with Hodge-star operators on $\Om^1(\Sigma)$, and write $*\in\JJ_\Sigma$.
Recall that  $\alpha\in\Om^1(\Sigma,\mathfrak{su}(m))$ satisfying
$d_A\alpha=*d_A\alpha=0$ is called \emph{harmonic}.
%
%
%
Recalling \cref{rigid}, the following result follows from
\cite[(2.6),\,(2.15)]{Hit90} and
\cite[\S\,2]{ADW91}.

\begin{prop}\label{preqMM}
The first projection $\pi:\cT_\Sigma\times\MM\fl\cT_\Sigma$
has a structure of a rigid holomorphic tautological fibration, with
associated line bundle
$(L,h^L,\nabla^L)$ as above. For any $*\in\JJ_\Sigma$
and $A\in\AA'$, the associated relative
complex structure $J\in\End(T\MM)$
is defined over $([*],[A])\in\cT_\Sigma\times\MM$
via \cref{tgtMM} and Hodge theory by
\begin{equation}\label{J=*}
J\alpha=-*\alpha,
\end{equation}
where $\alpha\in\Om^1(\Sigma,\mathfrak{su}(m))$ is the unique harmonic
representative of $[\alpha]\in H^1_A(\Sigma)$.
\end{prop}
The associated relative Riemannian metric on the relative tangent
bundle $T\MM$ over $([*],[A])\in\cT_\Sigma\times\MM$ is then given
via \cref{tgtMM} and Hodge theory by
\begin{equation}\label{Hodgemet}
g^{T\MM}([\alpha],[\beta])=\frac{m}{4\pi^2}
\int_\Sigma\<\alpha,\beta\>_{T^*\Sigma}\,dv_\Sigma,
\end{equation}
where $\alpha,\,\beta\in\Om^1(\Sigma,\mathfrak{su}(m))$ are the unique
harmonic representative of
$[\alpha],\,[\beta]\in H^1_A(\Sigma)$, for
$\<\cdot,\cdot\>_{T^*\Sigma}$ and $dv_\Sigma$ induced by
any Riemannian metric $g^{T\Sigma}$ associated with
$*\in\JJ_\Sigma$, which we take to be the associated
hyperbolic metric.
%

For any $A\in\AA'$ and $*\in\JJ_\Sigma$, let $d_A=:\partial_A+\dbar_A$
be the decomposition of $d_A$ with respect
to the splitting \cref{splitc} of $(T\Sigma,*)$.
Consider the universal family of Riemann surfaces over $\cT_\Sigma$
as in \cite[Th.5.6]{AC09}, and give it the structure of a holomorphic
prequantized fibration as in \cite[\S\,5.1]{KMMW17}.
By the construction of \cite[Th.1.9]{BGS88c},
we can consider the
holomorphic \emph{determinant line bundle} $\det(\dbar_A)$
of the induced family of $\dbar$-operators over $\MM\times\cT_\Sigma$.
The following theorem is a consequence of the curvature formula
of \cite[Th.1.9]{BGS88c}.
It follows from \cite[\S\,2]{Hit90} and the computations of
\cite[(4.17)]{ADW91}.
\begin{prop}\label{hitMM}
The canonical bundle $K_\MM$ over $\MM\times\cT_\Sigma$
is isomorphic to the dual of the determinant line bundle
$\det(\dbar_A)$ considered above.
Furthermore, the Chern curvature of its natural Hermitian
metric $h^{K_\MM}$ induced by \cref{Hodgemet}
satisfies
\begin{equation}\label{FIT}
\frac{\sqrt{-1}}{2\pi}\big[-R^{K_\MM}+\dbar\partial \rho\big]^\MM
=2\om^\MM,
\end{equation}
where $\rho\in\cinf(\MM\times\cT_\Sigma,\R)$ is the
\emph{analytic torsion} of the associated family of
$\dbar$-operators induced by the hyperbolic metric
associated with $*\in\TT_\Sigma$.
\end{prop}
The metric $h^Q:=e^{-\rho} h^{K_\MM}$ on $K_\MM$ is called the
\emph{Quillen metric}, and the Chern connection
$\nabla^Q$ of $(K_\MM,h^Q)$ satisfies
\begin{equation}\label{nabQ}
\nabla^Q=\nabla^{K_\MM}-\partial\rho.
\end{equation}
For $*\in\JJ_\Sigma$ and $f\in\Diff^+(\Sigma,D)$ as above,
let $\gamma:[0,1]\fl\cT_\Sigma$ be a path joining $[*]$
to $[f^* *]$, and consider the pullback of the universal family
of Riemann surfaces over $\cT_\Sigma$ by $\gamma$.
This induces a holomorphic prequantized fibration
\begin{equation}\label{maptor}
\pi_f:\Sigma_f\simeq\Sigma\times[0,1]/[(0,x)\sim(1,f(x))]
\longrightarrow\R/\Z,
\end{equation}
and we call $\Sigma_f$ the \emph{mapping torus} of $f$.
We write $(\Sigma\backslash D)_f\subset\Sigma_f$ for the mapping torus
of $f|_{\Sigma\backslash D}$.

By \cref{MM,LCS,preqMM,hitMM},
we can apply \cref{hitchin} with $k=2$ and consider the parallel transport
with respect to the connection $\{\nabla^p\}_{p\in\N^*}$
of \cref{hitchinfla} in the quantum bundle $\{\HH_p\}_{p\in\N^*}$
associated with $\pi:\cT_\Sigma\times\MM\fl\cT_\Sigma$ along
$\gamma:[0,1]\fl\cT_\Sigma$. This is precisely the canonical
projectively flat connection of
\cite[\S\,4.b]{ADW91} on the \emph{Verlinde bundle}.
By \cref{hitchintoep}, we are then
precisely in the context
of \cref{locsec}, with $(X,\om)=(\MM,\om^\MM)$ and $\varphi:\MM\fl\MM$
defined in \cref{varphif}.
\begin{proof}[Proof of \cref{WRTintro} and \cref{WRTintro2}]

%
%
%
%
%
Let $\AA_f$ be the space of flat $\SU(m)$-connections over
$(\Sigma\backslash D)_f$ with holonomy around the boundary equal to
$e^{\frac{2\pi\sqrt{-1}d}{m}}\,\Id_{\C^m}\in\SU(m)$, and let $\MM_f$
be the associated moduli space defined as above.
As all $A\in\AA$ are irreducible, the same is true for all $A_f\in\AA_f$. By assumption, the analogue of \cref{MM} applies to $\MM_f$, so that it is smooth and its
tangent space at any $[A_f]\in\MM_f$ identifies with
the first cohomology group $H^1_{A_f}(\Sigma_f)$
of $(\Om^\bullet_{A_f}(\Sigma_f),d_{A_f})$ defined as above.
Following \cite[\S\,5.2.1]{Jef92} and
\cite[\S\,8.1]{Cha16}, consider the map
$r:\MM_f\fl\MM$ defined by restriction over any fibre of \cref{maptor},
which satisfies $\Im\,r=\MM^\varphi$ and $\#\,r^{-1}([A])=m$, where
$\MM^\varphi$
denotes the fixed point set of $\varphi:\MM\fl\MM$
defined in \cref{varphif}.
%

For any $A_f\in\AA_f$,
let $A\in\AA$ be the restriction of $A_f$ to any fibre
of \cref{maptor},
so that $r([A_f])=[A]$.
Recall as in \cite[p.\,359]{Hit90} that the
irreducibility of $A\in\AA$
implies that $H^0_A(\Sigma)=H^2_A(\Sigma)=\{0\}$.
Let $dt$ be the canonical volume form of $\R/\Z$.
Following \cite[(4.2)]{And13},
we have the following long exact sequence in cohomology
\begin{equation}\label{Wang}
0\longrightarrow H^1_{A_\varphi}(\Sigma_\varphi)\xrightarrow{\,~~r^*~~\,}
H^1_{A}(\Sigma)\xrightarrow{~\Id-f^*~}
H^1_{A}(\Sigma)\xrightarrow{\bullet\,\wedge\,\pi_f^*dt\,}
H^2_{A_\varphi}(\Sigma_\varphi)\longrightarrow 0,
\end{equation}
where the first map is induced by restriction on any fibre of
\cref{maptor}, and thus identifies with the differential of
$r:\MM_f\fl\MM$ via \cref{tgtMM}. This shows that
$r:\MM_f\fl\MM$ is a smooth immersion,
and thus a smooth $m$-covering on its image $\Im\,r=\MM^\varphi$,
which is smooth as well.
Furthermore, the exactness of \cref{Wang} together with
\cref{tgtMM,varphif} implies
that $T\MM^\varphi=\Ker(\Id_{T\MM}-d\varphi)$.
Thus the fixed point set
of $\varphi$ is non-degenerate, and we can apply \cref{indeq} to this
situation.
The asymptotic expansion \cref{WAC} is then an immediate consequence
of \cref{indeq}, where the locally constant value of $\varphi^L$
restricted to $\MM^\varphi$ has been computed in \cite[Prop.8.4]{Cha16}
to be equal to \cref{corCS}. The densities over $\MM_f$ are obtained by pullback by $r:\MM_f\fl\MM^\varphi$.

Let us now show \cref{unWRTexp}.
Restrict the prequantized fibration
$\pi:\cT_\Sigma\times\MM\fl\cT_\Sigma$ over the path
$\gamma:[0,1]\fl\cT_\Sigma$ as in \cref{trivfib}, and
let $[A]\in\MM$ be fixed by $\varphi$.
%
%
Then under the assumptions of \cref{coeffgeom},
with $E=\C$ and $\tau^{E,\sigma}$ determined by
\cref{hitchintoep},
we get over any connected component of $\MM^\varphi$ of dimension
$2d$ and
for some coherent choices of square roots, 
\begin{equation}\label{nu0MM}
\nu_0=(-1)^{\frac{n-d}{2}}(\varphi^{K_\MM}
\tau^{Q,-1})^{\frac{1}{2}}
\,|\det{}_N(\Id_{N}-d\varphi|_N)|^{-\frac{1}{2}}\,|dv|_{T\MM/N},
\end{equation}
where $\tau^Q$ is the parallel
transport on $K^\MM$ with respect to \eqref{nabQ} and $|dv|_{T\MM/N}$
is the density over $\MM^\varphi$ satisfying \cref{dvX=dvNdvXN}.
Following \cite[Cor.4.3]{AH12}, the \emph{Reidemeister torsion}
$\tau_{\Sigma_\varphi}(\ad A_\varphi)$ of $\ad A_\varphi$
over $\Sigma_\varphi$
is equal to the torsion of the complex \cref{Wang},
which is identified with an element of
$\det H^1_{A_\varphi}(\Sigma_\varphi)^{-2}$ via Poincaré duality.
Thus the absolute value $|\tau_{\Sigma_\varphi}(\ad A_\varphi)
|^{\frac{1}{2}}$ can be identified with a density over $\MM^\varphi$.
Following \cite[Th.5.6]{AH12} and using the fact that $\Id-d\varphi$
preserves a transverse subbundle $N$ of $T\MM^\varphi$ as in
\cite[Prop.5.4]{AH12}, we deduce from \cref{dvX=dvNdvXN}
and \cref{Wang} that
\begin{equation}\label{torsReid}
|\tau_{\Sigma_\varphi}|^{\frac{1}{2}}=|\det{}_{N}
(\Id_{N}-d\varphi|_N)|^{-\frac{1}{2}}\,|dv|_{T\MM/N}.
\end{equation}
%
%
Now let $d_A+d_A^*$ be the
\emph{even signature operator} on
$\Om^\bullet_A(\Sigma,\mathfrak{su}(m))$, where $d_A^*$ is the dual 
of $d_A$ with respect to \cref{Hodgemet}.
By a construction of
\cite{BF86a}, we can consider the associated determinant line bundle
over $\cT_\Sigma\times\MM$, equipped with the
\emph{Bismut-Freed connection} of \cite[Def.1.17]{BF86a}.
We then have the following lemma, which follows
from an argument of \cite[\S\,4,\,(4.12)]{Ati87} using
\cref{hitMM} and \cite{BGS88c}.
\begin{lem}\label{QBF}
The square $K_\MM^2$ of the relative canonical line bundle of
$\pi:\cT_\Sigma\times\MM\fl\cT_\Sigma$ is
isomorphic to the determinant line bundle of the even signature operator
over $\cT_\Sigma\times\MM$. Furthemore, this isomorphism sends
the connection induced by $\nabla^Q$ on $K_\MM^2$ to the
Bismut-Freed connection.
\end{lem}
Then as explained in
\cite[\S\,4]{Ati87}, the holonomy theorem of
\cite[Th.3.16]{BF86b} implies that the parallel transport with respect
to the Bismut-Freed connection over the loop $(\gamma,[A])$
as above is given by $\exp(-\sqrt{-1}\pi\eta^0(A_f))$,
where $\eta^0(A_f)$ is the limit as $\epsilon\fl 0$
of the $\eta$-invariant of the \emph{odd signature operator}
$(-1)^k(d_{A_f}*_\epsilon+*_\epsilon d_{A_f})$ acting on
the odd forms
$\oplus_{k=1}^2\Om^{2k-1}_{A_f}(\Sigma_f,\mathfrak{su}(m))$,
where $*_\epsilon$ is the Hodge-star operator of the metric
$g^{T\Sigma}\oplus \epsilon^{-2} g^{T\R}$
constructed via \cref{split}
over the fibration $\pi_f:\Sigma_f\fl\R/\Z$.
By \cref{QBF}, this shows that
\begin{equation}\label{etacompute}
(-1)^{\frac{n-d}{2}}(\varphi^{K_\MM}\tau^{Q,-1})^{\frac{1}{2}}
=(\sqrt{-1})^k
\exp\left(\frac{\sqrt{-1}\pi}{4}\eta^0(\ad A_\varphi)\right),
\end{equation}
with $k\in\Z/4\Z$ locally constant by continuity of the other terms.
%
Then the integral
over $\MM_f$ descends to $m$ times an integral over $\MM^\varphi$ via
the $m$-covering $r:\MM_f\fl\MM^\varphi$.
This completes the proof of \cref{unWRTexp}.

Let us finally show \cref{WRTintro2}.
Taking the expansion of the fractional power in $p\in\N^*$,
the associated asymptotic expansion as in \cref{WAC}
follows from \cref{indeq} as above, and
to compute the first coefficient \cref{WRTexp}, we can assume that
$E=\pi^*\det(\dbar_\Sigma)^{-\frac{m^2-1}{2}}$. Using \cite[\S\,4]{Ati87} and
\cite{BGS88c} as above, the analogue of \cref{QBF} holds for $\det(\dbar_\Sigma)$
over $\cT_\Sigma$ instead of $K_\MM$, replacing $d_A$ by the usual
exterior differential $d$ on $\Om^\bullet(\Sigma,\C)$.
Thus the parallel transport in $E$ along $\gamma$ with respect to
the connection induced by \eqref{nabQ} is given by
$(\i)^k\exp(-\frac{\sqrt{-1}\pi}{4}(m^2-1)\eta^0(0))$
for some $k\in\Z/4\Z$, where $\eta^0(0)$ is the limit as
$\epsilon\fl 0$ of the $\eta$-invariant
of the usual odd signature operator
$(-1)^k(d*_{\epsilon}+*_{\epsilon}d)$ acting
on $\oplus_{k=1}^2\Om^{2k-1}(\Sigma_f,\C)$.
By definition of the $\rho$-invariant in
\cite[Th.2.4]{APS75b} and using $\dim\mathfrak{su}(m)=m^2-1$,
we get
\begin{equation}\label{rhodef}
\rho(\ad A_\varphi)=\eta^0(\ad A_\varphi)-(m^2-1)\eta^0(0).
\end{equation}
This completes the proof of \cref{WRTexp}.
\end{proof}

\providecommand{\bysame}{\leavevmode\hbox to3em{\hrulefill}\thinspace}
\providecommand{\MR}{\relax\ifhmode\unskip\space\fi MR }
\providecommand{\MRhref}[2]{%
  \href{http://www.ams.org/mathscinet-getitem?mr=#1}{#2}
}
\providecommand{\href}[2]{#2}

\ \\
Tel Aviv University - School of Mathematical Sciences,\\
Ramat Aviv, Tel Aviv 69978, Israël\\
\\
\emph{E-mail adress}: louisioos@mail.tau.ac.il

\end{document}